\documentclass[a4paper,10pt]{article}
\usepackage[utf8x]{inputenc}
\usepackage{amsthm,amsmath,amssymb,oldgerm}
\usepackage[margin=3cm]{geometry}
\theoremstyle{plain}
\newtheorem{thm}{Theorem}[section]
\newtheorem{lem}[thm]{Lemma}
\newtheorem{prop}[thm]{Proposition}
\newtheorem{cor}[thm]{Corollary}
\theoremstyle{definition}
\newtheorem{defn}[thm]{Definition} 
\theoremstyle{definition}
\newtheorem{notn}[thm]{Notation}
\newtheorem{rem}[thm]{Remark}

\let\Im\relax
\DeclareMathOperator{\Im}{Im}
\DeclareMathOperator{\del}{\Delta_{hyp}}
\DeclareMathOperator{\hyp}{\mu_{hyp}} 
\DeclareMathOperator{\can}{\mu_{can}}

\DeclareMathOperator{\hatcan}{\widehat{\mu}_{can}}
\DeclareMathOperator{\shyp}{\mu_{shyp}}

\DeclareMathOperator{\kh}{{\it{K_{\mathbb{H}}}}}
\DeclareMathOperator{\kxhyp}{{\it{K_{X,\mathrm{hyp}}}}}
\DeclareMathOperator{\hkxhyp}{{\it{HK}_{X,\mathrm{hyp}}}}
\DeclareMathOperator{\ekxhyp}{{\it{EK_{X,\mathrm{hyp}}}}}
\DeclareMathOperator{\pkxhyp}{{\it{PK_{X,\mathrm{hyp}}}}}
\DeclareMathOperator{\gxhyp}{\it{g_{X,\mathrm{hyp}}}}
\DeclareMathOperator{\gxcan}{\it{g_{X,\mathrm{can}}}}  
\DeclareMathOperator{\gh}{\it{g_{\mathbb{H}}}}
\DeclareMathOperator{\eeis}{\mathcal{E}_{{\it{X}},\mathrm{ell},\mathfrak{e}}}
\DeclareMathOperator{\epar}{\mathcal{E}_{X,\mathrm{par},p}}
\DeclareMathOperator{\boundb}{\it{B_{X,\varepsilon,\alpha,\delta}}}
\DeclareMathOperator{\boundbtwo}{\it{B_{X,\varepsilon\slash 2,\alpha,\delta}}}
\DeclareMathOperator{\boundc}{\it{C_{X,\varepsilon,\alpha,\delta}}}

\let\Re\relax
\DeclareMathOperator{\Re}{Re}
\DeclareMathOperator{\vx}{\mathrm{vol_{\mathrm{hyp}}}}
\let\id\relax
\DeclareMathOperator{\id}{\mathrm{id}}

\newcommand{\Rmnum}[1]{\expandafter\@slowromancap\romannumeral #1@}
\title{Bounds for Green's functions on noncompact hyperbolic Riemann orbisurfaces of finite volume}
{\small\author{Anilatmaja Aryasomayajula}}
\date{}
\begin{document}
\maketitle
\begin{abstract}
\noindent 
In 2006, J.~Jorgenson and J.~Kramer derived bounds for the canonical Green's function 
and the hyperbolic Green's function defined on a compact hyperbolic Riemann surface. In this 
article, we extend these bounds to noncompact hyperbolic Riemann orbisurfaces of finite 
volume and of genus greater than zero, which can be realized as a quotient space of the action of a Fuchsian subgroup of first 
kind on the hyperbolic upper half-plane. 
 
\vspace{0.2cm}\noindent

\vspace{0.2cm}\noindent
\end{abstract}
\section*{Introduction}
\paragraph{Notation}
Let $X$ be a noncompact hyperbolic Riemann orbisurface of finite volume $\vx(X)$ with genus $g_{X}\geq 1$, 
and can be realized as the quotient space $\Gamma_{X}\backslash\mathbb{H}$, where $\Gamma_{X}
\subset \mathrm{PSL}_{2}(\mathbb{R})$ is a Fuchsian subgroup of the first kind acting on the hyperbolic 
upper half-plane $\mathbb{H}$, via fractional linear transformations. Let $\mathcal{P}_{X}$ and  
$\mathcal{E}_{X}$ denote the set of cusps and the set of  elliptic fixed points of $\Gamma_{X}$, respectively. 
Put $\overline{X}=X\cup \mathcal{P}_{X}$. Then, $\overline{X}$ admits the structure of a Riemann surface. 

\vspace{0.2cm}
Let $\hyp(z)$ denote the (1,1)-form associated to hyperbolic metric, which is the natural metric on $X$, 
and of constant negative curvature minus one. Let $\shyp(z)$ denote the rescaled hyperbolic metric 
$\hyp(z)\slash \vx(X)$, which measures the volume of $X$ to be one. 

\vspace{0.2cm}
The Riemann surface $\overline{X}$ is embedded in its Jacobian variety $\mathrm{Jac}(\overline{X})$ via the 
Abel-Jacobi map. Then, the pull back of the flat Euclidean metric by the Abel-Jacobi map is called the 
canonical metric, and the (1,1)-form associated to it is denoted by $\hatcan(z)$. We denote its 
restriction to $X$ by $\can(z)$.

\vspace{0.2cm}
For $\mu=\shyp(z)$ or $\can(z)$, let $g_{X,\mu}(z,w)$ defined on $X\times X$ denote the Green's function 
associated to the metric $\mu$. The Green's function $g_{X,\mu}(z,w)$ is uniquely determined by 
the differential equation (which is to be interpreted in terms of currents)
\begin{equation}\label{diffintro}
d_{z}d^{c}_{z}g_{X,\mu}(z,w)+ \delta_{w}(z)=\mu(z),
\end{equation}
with the normalization condition
\begin{equation*}
\int_{X}g_{X,\mu}(z,w)\mu(z)=0. 
\end{equation*}
The Green's function $\gxcan(z,w)$ associated to the canonical metric $\can(z)$ is called the 
canonical Green's function. Similarly the Green's function $\gxhyp(z,w)$ associated to the 
(rescaled) hyperbolic metric $\shyp(z)$ is called the hyperbolic Green's function. 

\vspace{0.2cm}
From differential equation \eqref{diffintro}, we can deduce that for a fixed $w\in X$, as a function in the 
variable $z$, both the Green's functions $\gxcan(z,w)$ and $\gxhyp(z,w)$ are $\log$-singular at $z=w$. Recall that 
$\hyp(z)$ is singular at the cusps and at the elliptic fixed points, and $\can(z)$ the pull back of the smooth and 
flat Euclidean metric is smooth on $X$. Hence, from the elliptic regularity of the $d_{z}d_{z}^{c}$ operator, it 
follows that $\gxhyp(z,w)$ is $\log\log$-singular at the cusps, and $\gxcan(z,w)$ remains smooth at the cusps.     


\vspace{0.2cm}
From a geometric perspective, it is very interesting to compare the two metrics $\hyp(z)$ and 
$\can(z)$, and study the difference of the two Green's functions 
\begin{align}\label{diff}
\gxhyp(z,w)-\gxcan(x,w).
\end{align}
on compact subsets of $X$. 

\vspace{0.2cm}
In \cite{jk}, J.~Jorgenson and J.~Kramer have already established these tasks, when $X$ is a compact 
Riemann surface devoid of elliptic fixed points. They proved a key-identity that relates the 
hyperbolic metric $\hyp(z)$ and the canonical metric $\can(z)$ via the hyperbolic heat kernel. Using 
the key-identity, they expressed the difference \eqref{diff} in terms of integrals which 
involve only the hyperbolic heat kernel and the hyperbolic metric. This allowed them to derive 
bounds for the difference \eqref{diff} in terms of invariants coming from the hyperbolic geometry 
of $X$, namely, the injectivity radius of $X$ and the first non-zero eigenvalue $\lambda_{X,1}$ 
of the hyperbolic Laplacian $\del$ acting on smooth functions defined on $X$. 

\vspace{0.2cm}
In \cite{anilpaper}, we extend the key-identity from \cite{jk} to cusps and elliptic fixed 
points at the level of currents. This relation serves as a starting point for extending the bounds 
for the canonical and the hyperbolic Green's function from \cite{jk} to noncompact hyperbolic 
Riemann orbisurfaces of finite volume. 

\vspace{0.2cm}
In this article, using the key-identity from \cite{anilpaper} and by extending the methods used 
in \cite{jk}, we study the difference \eqref{diff} on compact subsets of $X$, and as an application, we derive upper 
bounds for the canonical Green's function $\gxcan(z,w)$ on $X$. Our bounds are similar to the ones derived in 
\cite{jk}.  
\paragraph{Statement of main results}
We now describe our results for the modular curve $Y_{0}(N)=\Gamma_{0}(N)\backslash\mathbb{H}$. 
However, our results hold true for any noncompact hyperbolic Riemann orbisurface of finite volume and of genus 
greater than zero. Let $N\in\mathbb{N}_{>0}$ be 
such that the modular curve $Y_{0}(N)$ has genus $g_{Y_{0}(N)}\geq 1$. 
Let $0< \varepsilon< 1$ be small enough such that it satisfies the conditions elucidated in Notation \ref{epsilondefn}. 

\vspace{0.2cm}
For any cusp $p\in \mathcal{P}_{Y_{0}(N)}$, let $U_{N,\varepsilon}(p)$ denote an open coordinate disk 
of radius $\varepsilon$ around the cusp $p$. For any elliptic fixed point $\mathfrak{e}\in \mathcal{E}_{Y_{0}(N)}$, let 
$U_{N,\varepsilon}(\mathfrak{e})$ denote an open coordinate disk around the elliptic fixed point 
$\mathfrak{e}$, which is as described in condition (3) in Notation \ref{epsilondefn}. Put
\begin{align*}
Y_{0}(N)_{\varepsilon}= Y_{0}(N)\backslash\Bigg(\bigcup_{p\in\mathcal{P}_{Y_{0}(N)}}U_{\varepsilon}(p)
\cup \bigcup_{\mathfrak{e}\in \mathcal{E}_{Y_{0}(N)}}U_{\varepsilon}(\mathfrak{e})\Bigg).
\end{align*}
For any $\delta > 0$ and a fixed $z,w\in X$, identifying $Y_{0}(N)$ with its fundamental domain, we define the set 
\begin{align*}
S_{\Gamma_{Y_{0}(N)}}(\delta;z,w)=\big\lbrace{\gamma\in\mathcal{H}(\Gamma_{0}(N))\cup\lbrace\id\rbrace\big|
\,d_{\mathbb{H}}(z,\gamma w) <\delta\big\rbrace},
\end{align*}
where $\mathcal{H}(\Gamma_{0}(N))$ denotes the hyperbolic elements of $\Gamma_{0}(N)$. 
Furthermore, let $\gh(z,w)$ denote the free-space Green's function defined on $\mathbb{H}\times
\mathbb{H}$, which is given by the formula 
\begin{align*}
\gh(z,w)=\log\bigg{|}\frac{z-\overline{w}}{z-w}\bigg{|}^{2}.  
\end{align*}
From \cite{selberg}, recall that the first non-zero eigenvalue of the hyperbolic Laplacian $\del$ satisfies the 
lower bound $\lambda_{Y_{0}(N),1}\geq 3\slash 16$. With notation as above, for any $\delta >0$, using the dependence 
of the genus $g_{Y_{0}(N)}$, the number of cusps $|\mathcal{P}_{Y_{0}(N)}|$, and the number of elliptic fixed points 
$|\mathcal{E}_{Y_{0}(N)}|$ in terms of $N$ from p. 22--25 in \cite{shimura}, we derive the following estimates 
\begin{align}
&\sup_{z,w\in Y_{0}(N)_{\varepsilon}}\big|g_{Y_{0}(N),\mathrm{can}}(z,w)-g_{Y_{0}(N),\mathrm{hyp}}
(z,w)\big|=\notag\\&O_{\varepsilon,\delta}\bigg(\frac{\big(|\mathcal{P}_{Y_{0}(N)}|+|\mathcal{E}_{Y_{0}(N)}|\big)}{g_{Y_{0}(N)}}
\bigg(1+\frac{1}{\lambda_{Y_{0}(N),1}}\bigg)\bigg)=O_{\varepsilon,\delta}(1);\label{mainthm1}
\end{align}
\begin{align}\label{mainthm2}
&\sup_{z,w\in Y_{0}
(N)_{\varepsilon}}\bigg|g_{Y_{0}(N),\mathrm{can}}(z,w)-\sum_{\gamma\in S_{\Gamma_{Y_{0}(N)}}(\delta;z,w)}\gh(z,\gamma w)
\bigg|=\notag\\[0.4em]& O_{\varepsilon,\delta}\bigg(\frac{\big(|\mathcal{P}_{Y_{0}(N)}|+|\mathcal{E}_{Y_{0}(N)}|
\big)}{g_{Y_{0}(N)}}\bigg(1+\frac{1}{\lambda_{Y_{0}(N),1}}\bigg)\bigg)=O_{\varepsilon,\delta}(1).
\end{align}
We even derive bounds for the canonical Green's function $g_{Y_{0}(N),\mathrm{can}}(z,w)$
at cusps and at elliptic fixed points. 
\paragraph{Arithmetic significance}
In 1974, in \cite{arakelov}, Arakelov defined an intersection theory for divisors on an arithmetic surface by incorporating 
the associated compact Riemann surface with its complex analytic geometry. The contribution at infinity is 
calculated by using canonical Green's functions defined on the corresponding Riemann surfaces. 

\vspace{0.2cm}
In \cite{edix}, B.~Edixhoven, J.-M.~Couveignes, and R. S. de Jong devised an algorithm which for a 
given prime $\ell$,  computes the Galois representations modulo $\ell$ associated to a fixed modular 
form of arbitrary weight, in time polynomial in $\ell$.  

\vspace{0.2cm}
To show that the complexity of the algorithm is polynomial in $\ell$, they needed an upper bound for 
the canonical Green's function associated to the compactified modular surface $X_{1}(\ell)$, 
and the upper bound provided by F.~Merkl (also published in \cite{edix}) proved sufficient.  

\vspace{0.2cm}
Bounds for the canonical Green's function from \cite{jk} when restricted to $X_{1}(\ell)$ yield 
better bounds than the ones derived by F.~Merkl. 

\vspace{0.2cm}
In 2011, in \cite{bruin}, while extending the algorithm of Edixhoven-Couveignes-de Jong, following the 
methods of F.~Merkl, P.~Bruin has derived bounds for the canonical Green's function, 
which for a given modular curve $Y_{0}(N)$ are of the form $O(N^{2})$, which will appear as \cite{bruinpaper}. 

\vspace{0.2cm}
Furthermore, using the bounds of P.~Bruin for the canonical Green's function, A.~Javanpeykar has derived bounds 
for various Arakelovian invariants like the Faltings delta function and Faltings height function in 
\cite{ariyan}. 

\vspace{0.2cm}
Our bounds for the canonical Green's function are stronger than the ones derived by P.~Bruin, 
and are optimally derived by following the methods from \cite{jk}. Furthermore, our bounds 
for the canonical Green's function $g_{X,\mathrm{can}}(z,w)$ at cusps are essential for calculating the 
Faltings height of any modular curve $X$. We are hopeful that our results together with \cite{ariyan} will lead 
to better bounds for the Arakelovian invariants considered in \cite{ariyan}.


\vspace{0.2cm}
This article also completes the program of J.~Jorgenson and J.~Kramer of estimating Arakelovian 
invariants of modular curves via techniques coming from global analysis and theory of heat kernels. However it 
would be interesting to study Edixhoven-Couveignes-de Jong's algorithm from \cite{edix}, using our bounds for the canonical Green's 
function, and we hope our bounds lead to a better complexity for the algorithm. 

\vspace{0.2cm}
Moreover, for any noncompact hyperbolic Riemann orbisurface $X=\Gamma_{X}\backslash \mathbb{H}$, we have studied the 
convergence of the following series
\begin{align}\label{convseries}
\sum_{\gamma\in\mathcal{P}(\Gamma_{X})}\gh(z,\gamma z),\quad \sum_{\gamma\in\mathcal{E}(\Gamma_{X})}
\gh(z,\gamma z),\quad \int_{X}\bigg(\sum_{\gamma\in\mathcal{H}(\Gamma_{X})}\kh(t;z,\gamma z)-\frac{1}{\vx(X)}\bigg)dt,
\end{align}
where $\mathcal{P}(\Gamma_{X})$, $\mathcal{E}(\Gamma_{X})$, and $\mathcal{H}(\Gamma_{X})$ denote 
the parabolic, elliptic, and hyperbolic elements of $\Gamma_{X}$, respectively, and the quantity 
$\kh(t;z,w)$ denotes the hyperbolic heat kernel on $\mathbb{H}\times\mathbb{H}$. We have also 
studied the behavior of the above stated series at the cusps and at the elliptic fixed points. We 
believe that this analysis helps in the generalization of the work of J.~Jorgenson and J.~Kramer from 
\cite{jk} and \cite{jkannals} to noncompact hyperbolic Riemann orbisurfaces and to higher dimensions. 
\paragraph{Organization of the paper}
In the first section, we set up our notation, introduce basic notions, and results. In section 2, we 
prove convergence of the automorphic functions mentioned in \eqref{convseries}. In section 3, 
using the existing bounds for the heat kernel from \cite{jk}, we derive bounds for the hyperbolic 
Green's function $\gxhyp(z,w)$ on compact subsets of $X$, and then extend these bounds to the neighborhoods 
of cusps and elliptic fixed points.  In section 4, using the convergence results from section 2, and bounds for the hyperbolic 
Green's function, we derive bounds for the canonical Green's function $\gxcan(z,w)$ on compact subsets of $X$, and then extend these bounds to the neighborhoods 
of cusps and elliptic fixed points. Finally, in section 5, we extend our bounds to certain sequences 
of admissible noncompact Riemann orbisurfaces to prove estimates \eqref{mainthm1} and \eqref{mainthm2}. 
{\small{\paragraph{Acknowledgements}
This article is part of the PhD thesis of the author, which was completed under the supervision 
of  J.~Kramer at Humboldt Universit\"at zu Berlin. The author would like to express his 
gratitude to J.~Kramer for his advice and for pointing out several errors in the first proof. 
The author would also like to extend his gratitude to J.~Jorgenson for sharing new 
scientific ideas, and to R.~S.~de~Jong for many interesting scientific discussions.}}  
\section{Background material}\label{section1}
In this section, we recall the basic notions and results required for next sections. 

\vspace{0.2cm}
Let $\Gamma_{X} \subset \mathrm{PSL}_{2}(\mathbb{R})$ be a Fuchsian subgroup of the first kind acting by 
fractional linear transformations on the upper half-plane $\mathbb{H}$. Let $X$ be the quotient space 
$\Gamma_{X}\backslash \mathbb{H}$, and let $g_{X}\geq 1$ denote the genus of $X$. The quotient 
space $X$ admits the structure of a Riemann orbisurface.  

\vspace{0.2cm}
Let $\mathcal{P}_{X}$ and $\mathcal{E}_{X}$ denote the finite set of cusps and finite set of elliptic 
fixed points of $X$, respectively. For $\mathfrak{e}\in \mathcal{E}_{X}$, let $m_{\mathfrak{e}}$ denote the order of 
$\mathfrak{e}$; for $p\in \mathcal{P}_{X}$, put $m_{p}=\infty$; for $z\in X\backslash \mathcal{E}_{X}$, put $m_{z}=1.$ 
Let $\overline{X}$ denote $\overline{X}=X\cup\mathcal{P}_{X}$.

\vspace{0.2cm}
Locally, away from cusps and elliptic fixed points, we identity $\overline{X}$ with its universal cover $\mathbb{H}$, 
and hence, denote the points on $\overline{X}\backslash (\mathcal{P}_{X}\cup \mathcal{E}_{X})$ by the same letter as 
the points on $\mathbb{H}$.
\paragraph{Structure of $\overline{X}$ as a Riemann surface} 
The quotient space $\overline{X}$ admits the structure of a compact Riemann surface. We refer the reader to 
section 1.8 in $\cite{miyake}$, for the details regarding the structure of $\overline{X}$ as a compact 
Riemann surface. For the convenience of the reader, we recall the coordinate functions for the 
neighborhoods of cusps and elliptic fixed points.

\vspace{0.2cm}
Let $p\in\mathcal{P}_{X}$ be a cusp, and let $U(p)$ denote a coordinate disk around the cusp $p$. 
Then, for any $w\in U(p)$, the coordinate function $\vartheta_{p}(w)$ for the open coordinate disk 
$U(p)$ is given by
\begin{equation*}
\vartheta_{p}(w)= e^{2\pi i \sigma_{p}^{-1}w},
\end{equation*}
where $\sigma_{p}$ is a scaling matrix of the cusp $p$ satisfying the following relations
\begin{align}
\sigma_{p}i\infty = p \quad \mathrm{and} \quad \sigma_{p}^{-1}\Gamma_{X,p}\sigma_{p} = \langle\gamma_{\infty}
\rangle,\quad\mathrm{where}\,\,\, \gamma_{\infty}=\left(\begin{array}{ccc} 1 & 1\\ 0 & 1  \end{array}\right)
\quad&\mathrm{and}\quad\Gamma_{X,p}=\langle\gamma_{p}\rangle \label{parscaling}
\end{align}
denotes the stabilizer of the cusp $p$ with generator $\gamma_{p}$.

\vspace{0.2cm}
Similarly, let $\mathfrak{e}\in\mathcal{E}_{X}$ be an elliptic fixed point, and let 
$U(\mathfrak{e})$ denote a coordinate disk around the elliptic fixed point $\mathfrak{e}$. 
Then, for any $w\in U(\mathfrak{e})$, the coordinate function $\vartheta_{\mathfrak{e}}(w)$ for 
the open coordinate disk $U(\mathfrak{e})$ is given by
\begin{equation*}\label{ellipticcoordinatefn}
\vartheta_{\mathfrak{e}}(w)= \bigg(\frac{w-\mathfrak{e}}{w-{\overline{\mathfrak{e}}}}\bigg)^{m_{\mathfrak{e}}}.
\end{equation*} 
\paragraph{Hyperbolic metric} 
We denote the (1,1)-form corresponding to the hyperbolic metric of $X$, which is compatible with the complex 
structure on $X$ and has constant negative curvature equal to minus one, by $\hyp(z)$. Locally, for 
$z\in X\backslash \mathcal{E}_{X}$, it is given by
\begin{equation*}
 \hyp(z)= \frac{i}{2}\cdot\frac{dz\wedge d\overline{z}}{{\Im(z)}^{2}}.
\end{equation*} 
Let $\vx(X)$ be the volume of $X$ with respect to the hyperbolic metric $\hyp$. It is given by 
the formula 
\begin{equation*}
\vx(X) = 2\pi\bigg(2g -2 + |\mathcal{P}_{X}|+\sum_{\mathfrak{e} \in \mathcal{E}_{X}}\bigg(1-
\frac{1}{m_{\mathfrak{e}}}\bigg)\bigg). 
\end{equation*}
The hyperbolic metric $\hyp(z)$ is singular at the cusps and at the elliptic fixed points, and 
the rescaled hyperbolic metric 
\begin{equation*}
 \shyp(z)= \frac{\hyp(z)}{ \vx(X)}
\end{equation*}
measures the volume of $X$ to be one. 

\vspace{0.2cm}
Locally, for $z\in X$, the hyperbolic Laplacian $\Delta_{\mathrm{hyp}}$ on $X$ is given by
\begin{equation*}
\Delta_{\mathrm{hyp}} = -y^{2}\bigg(\frac{\partial^{2}}{\partial x^{2}} +
 \frac{\partial^{2}}{\partial y^{2}}\bigg) = -4y^{2}\bigg(\frac{\partial^{2}}{\partial z
\partial \overline{z}} \bigg). 
\end{equation*}
Recall that $d=\left(\partial + \overline{\partial} \right), 
$ $d^{c}=\dfrac{1}{4\pi i}\left( \partial - \overline{\partial}
\right)$, and $dd^{c}= -\dfrac{\partial\overline{\partial}}{2\pi i}$. Furthermore, we have 
\begin{align}\label{ddcreln}
 d_{z}d_{z}^{c}=\del\hyp(z).
\end{align}
\paragraph{Canonical metric}
Let $S_{2}(\Gamma_{X})$ denote the $\mathbb{C}$-vector space of cusp forms of weight 2 with respect to 
$\Gamma_{X}$ equipped with the Petersson inner-product. Let $\lbrace f_{1},\ldots,f_{g_{X}}\rbrace $ denote 
an orthonormal basis of $S_{2}(\Gamma_{X})$ with respect to the Petersson inner product. Then, the (1,1)-form 
$\can(z)$ corresponding to the canonical metric of $X$ is given by 
\begin{equation*}
 \can(z)=\frac{i}{2g_{X}} \sum_{j=1}^{g_{X}}\left|f_{j}(z)\right|^{2}dz\wedge d\overline{z}.
\end{equation*}
The canonical metric $\can(z)$ remains smooth at the cusps and at the elliptic fixed points, 
and measures the volume of $X$ to be one. 

\vspace{0.2cm}
For $z\in X$, we put,
\begin{align}\label{defndx}
d_{X}=\sup_{z\in X}\frac{\can(z)}{\shyp(z)}.
\end{align}
As the canonical metric $\can(z)$ remains smooth at the cusps and at the elliptic fixed points, 
and the hyperbolic metric is singular at these points, the quantity $d_{X}$ is well-defined. 
\paragraph{Canonical Green's function}
For $z, w \in X$, the canonical Green's function $\gxcan(z,w)$ is defined as 
the solution of the differential equation (which is to be interpreted in terms of currents)
\begin{equation}\label{diffeqngcan}
d_{z}d^{c}_{z}\gxcan(z,w)+ \delta_{w}(z)=\can(z),
\end{equation}
with the normalization condition
\begin{equation*}\label{normcondgcan}
 \int_{X}\gxcan(z,w)\can(z)=0. 
\end{equation*}
From equation \eqref{diffeqngcan}, it follows that $\gxcan(z,w)$ admits a $\log$-singularity at 
$z=w$, i.e., for $z, w\in X$, it satisfies 
\begin{equation}\label{gcanbounded}
\lim_{w\rightarrow z}\big(\gxcan(z,w)+ \log |\vartheta_{z}(w)|^{2}\big)= O_{z}(1). 
\end{equation}
\paragraph{Parabolic Eisenstein Series} 
For $z\in X$ and $s\in\mathbb{C}$ with $\Re(s)> 1$, the parabolic Eisenstein series 
$\mathcal{E}_{X,\mathrm{par},p}(z,s)$ corresponding to a cusp $p\in\mathcal{P}_{X}$ is defined by the 
series
\begin{equation*}
\mathcal{E}_{X,\mathrm{par},p}(z,s) = \sum_{\eta \in \Gamma_{X,p}\backslash \Gamma_{X}}
\Im(\sigma_{p}^{-1}\eta z)^{s}.
\end{equation*}
The series converges absolutely  and uniformly for $\Re(s) >1 $. It admits a meromorphic continuation to all 
$s\in\mathbb{C}$ with a simple pole at $s = 1$, and the Laurent expansion at $s=1$ is of the form 
\begin{equation}\label{laurenteisenpar}
\mathcal{E}_{X,\mathrm{par},p}(z,s) = \frac{1}{(s-1)\vx(X)} + \kappa_{X,p}(z) + O_{z}(s-1),
\end{equation}
where $\kappa_{X,p}(z)$ the constant term of $\mathcal{E}_{X,\mathrm{par},p}(z,s)$ at $s=1$ is called 
Kronecker's limit function (see Chapter 6 of \cite{hi}).

\vspace{0.2cm}
For $z\in X$, and $p,q \in \mathcal{P}_{X}$, the Kronecker's limit function 
$\kappa_{X,p}(\sigma_{q}z)$ satisfies the following equation (see Theorem 1.1 of \cite{jsu} for the proof)
\begin{align}\label{fourierkappaeqn}
\kappa_{X,p}(\sigma_{q}z)= \sum_{n < 0} k_{p,q}(n)e^{2 
\pi in\overline{z}}+ \delta_{p,q}\Im(z)+
k_{p,q}(0)- \frac{\log\big(\Im(z)\big)}{\vx(X)} + 
\sum_{n > 0}k_{p,q}(n)e^{2\pi i nz},   
\end{align}
with Fourier coefficients $k_{p,q}(n)$ $\in$ $\mathbb{C}$.

\vspace{0.2cm}
For $p,q\in$ $\mathcal{P}_{X}$, as $z\in X$ approaches $q$, the Eisenstein series 
$\mathcal{E}_{X,\mathrm{par},p}(z,s)$ corresponding to the cusp $p\in\mathcal{P}_{X}$ satisfies the 
following equation (see Corollary 3.5 in \cite{hi})
\begin{align}\label{fouriereisen}
\mathcal{E}_{X,\mathrm{par},p}(z,s)= \delta_{p,q}\Im(\sigma_{q}^{-1}z)^{s} +
\alpha_{p,q}(s)\Im(\sigma_{q}^{-1}z)^{1-s} +
O\left(\big(1+\Im(\sigma_{q}^{-1}z)^{-\Re(s)}\big)e^{-2\pi
\Im(\sigma_{q}^{-1}z)}\right),
\end{align}
where the Fourier coefficient  $\alpha_{p,q}(s)$ is given by equation (3.21) in \cite{hi}.
\paragraph{Elliptic Eisenstein series}
Let $\mathfrak{e}\in\mathcal{E}_{X}$ be an elliptic fixed point of order $m_{\mathfrak{e}}$ with stabilizer 
subgroup $\Gamma_{X,\mathfrak{e}}$. Let $\sigma_{\mathfrak{e}}$ be a scaling matrix of 
$\mathfrak{e}$ satisfying the conditions 
\begin{align}\label{ellipticscalingmatrix}
\sigma_{\mathfrak{e}}i =\mathfrak{e} \quad \mathrm{and} \quad \sigma_{\mathfrak{e}}^{-1}
\Gamma_{X,\mathfrak{e}}\sigma_{\mathfrak{e}} = \langle\gamma_{i}\rangle,\quad \mathrm{where}\,\,\, 
\gamma_{i}=\left(\begin{array}{ccc} \cos(\pi\slash m_{\mathfrak{e}}) & \sin(\pi\slash m_{\mathfrak{e}})\\ 
-\sin(\pi\slash m_{\mathfrak{e}}) & \cos(\pi\slash m_{\mathfrak{e}})\end{array}\right).
\end{align} 
Let $\rho(z)$ denote the hyperbolic distance $d_{\mathbb{H}}(z,i)$. Then, for $z\in X$ and $s\in\mathbb{C}$ with $\Re(s)> 1$, 
the elliptic Eisenstein series $\mathcal{E}_{X,\mathrm{ell},\mathfrak{e}}(z,s)$ corresponding to an elliptic fixed point 
$\mathfrak{e}\in\mathcal{E}_{X}$ is defined by the series
\begin{equation*}
\eeis(z,s)= \sum_{\eta \in \Gamma_{X,\mathfrak{e}}\backslash \Gamma_{X}}
\sinh^{-s}\big(\rho(\sigma_{\mathfrak{e}}^{-1}\eta z)\big). 
\end{equation*}
The series converges absolutely and uniformly for $\Re(s)>1$ and $z\not=\mathfrak{e}$ (see \cite{anna}). From its 
definition, as $z\in X\backslash \mathcal{E}_{X}$ approaches an elliptic fixed point $\mathfrak{e}\in\mathcal{E}_{X}$, 
for any $s\in\mathbb{C}$ with $\Re(s)> 1$, we find
\begin{align}\label{eeisbound}
\eeis(z,s)-\sinh^{-s}\big(\rho(\sigma_{\mathfrak{e}}^{-1} z)\big)=O_{z}(1).
\end{align}
Moreover, for any $z\in X$, $s\in\mathbb{C}$ with $\Re(s)> 1$, and any cusp $p\in\mathcal{P}_{X}$, it follows that
\begin{align}\label{eeiscusp}
 \lim_{z\rightarrow p}\eeis(z,s)=0.
\end{align}
\paragraph{Space of square-integrable functions}
Let $L^{2}(X)$ denote the space of square integrable functions on $X$ with 
respect to the hyperbolic (1,1)-form $\hyp(z)$. There exists a natural inner-product 
$\langle\cdot ,\cdot\rangle$ on $L^{2}(X)$ given by
\begin{align*}
\langle f,g \rangle=\int_{X}f(z)\overline{g(z)}\hyp(z),
\end{align*}
where $f, g \in L^{2}(X)$, making $L^{2}(X)$ into a Hilbert space. 

\vspace{0.2cm}
Furthermore, every $f\in L^{2}(X)$ admits the spectral expansion 
\begin{align}\label{spectralf}
&f(z)=\sum_{n =0}^{\infty}\langle f, \varphi_{X,n}(z)\big\rangle\varphi_{X,n}(z) + 
\frac{1}{4\pi}\sum_{p\in \mathcal{P}_{X}}
\int_{-\infty}^{\infty}\big\langle f,\mathcal{E}_{X,\mathrm{par},p}(z,1\slash 2+ir)\big\rangle 
\mathcal{E}_{X,\mathrm{par},p}(z,1\slash 2+ir)dr,
\end{align}
where $\lbrace{\varphi_{X,n}(z)\rbrace}$ denotes the set of orthonormal 
eigenfunctions for the discrete spectrum of $\Delta_{\mathrm{hyp}}$, 
and $\lbrace{\mathcal{E}_{X,\mathrm{par},p}(z,1\slash 2+ ir)\rbrace}$ denotes 
the set of eigenfunctions for the continuous spectrum of $\del$, with $\mathcal{E}_{X,\mathrm{par},p}
(z,s)$ denoting the parabolic Eisenstein series for the cusp $p\in \mathcal{P}_{X}$. 

\vspace{0.2cm}
The eigenfunctions $\lbrace{\varphi_{X,n}(z)\rbrace}$ corresponding to the discrete spectrum 
can all be chosen to be real-valued, and for the rest of this article we continue to assume so.
\paragraph{Heat Kernels}
For $t \in \mathbb{R}_{> 0}$  and $z, w \in \mathbb{H}$, the hyperbolic heat kernel $K_{\mathbb{H}}(t;z,w)$ on $\mathbb{R}_{> 0}\times
\mathbb{H} \times\mathbb{H}$ is given by the formula
\begin{equation}\label{defnkh}
K_{\mathbb{H}}(t;z,w)= \frac{\sqrt{2}e^{- t\slash 4}}{(4\pi t)^{3\slash 2}}
\int_{d_{\mathbb{H}}(z,w)}^{\infty}\frac{re^{-r^{2}\slash 4t}}{\sqrt{\cosh(r)-\cosh (d_{\mathbb{H}}(z,w))}}dr,
\end{equation}
where $d_{\mathbb{H}}(z,w)$ is the hyperbolic distance between $z$ and $w$.

\vspace{0.2cm}
For $t \in  \mathbb{R}_{> 0}$ and $z, w \in X$, the hyperbolic heat kernel $\kxhyp(t;z,w)$ on $\mathbb{R}_{> 0}\times X\times X$ is defined as 
\begin{equation*}\label{defnkhyp}
\kxhyp(t;z,w)=\sum_{\gamma\in\Gamma_{X}}K_{\mathbb{H}}(t;z,\gamma w).
\end{equation*}
For notational brevity, we denote $\kxhyp(t;z,w)$ by $\kxhyp(t;z)$, when $z=w$.

\vspace{0.2cm}
The hyperbolic heat kernel $\kxhyp(t;z,w)$ admits the spectral expansion 
\begin{align}
&\kxhyp(t;z,w)  = \sum_{n=0}^{\infty}\varphi_{X,n}(z)\varphi_{X,n}(w)e^{-\lambda_{X,n}t} +\notag\\ &
\frac{1}{4\pi}\sum_{p \in \mathcal{P}_{X}} \int_{-\infty}^{\infty}\mathcal{E}_{X,\mathrm{par},p}(z,1
\slash2+ir)\mathcal{E}_{X,\mathrm{par},p}(w,1\slash 2-ir)e^{-(r^{2}+ 1\slash4)t}dr,
\label{spectralexpansionheatkernel}  
 \end{align} 
where $\lambda_{X,n}$ denotes the eigenvalue of the normalized eigenfunction 
$\varphi_{X,n}(z)$ and $(r^{2}+1\slash 4)$ is the eigenvalue of the eigenfunction 
$\mathcal{E}_{X,\mathrm{par},p}(z,1\slash 2+ir)$, as above.

\vspace{0.2cm}
Let $\mathcal{P}(\Gamma_{X})$, $\mathcal{E}(\Gamma_{X})$, and $\mathcal{H}(\Gamma_{X})$ (here $\id$ is not treated as 
a parabolic element) denote the sets of parabolic, elliptic, and hyperbolic elements of the Fuchsian subgroup 
$\Gamma_{X}$, respectively. For $t\in \mathbb{R}_{\geq 0}$ and $z\in X$, put 
\begin{align*}
&\pkxhyp(t;z)=\sum_{\gamma\in\mathcal{H}(\Gamma_{X})}\kh(t;z,\gamma z), \quad \ekxhyp(t;z)=
\sum_{\gamma\in\mathcal{E}(\Gamma_{X})}\kh(t;z,\gamma z)\\
&\hkxhyp(t;z)=\sum_{\gamma\in\mathcal{P}(\Gamma_{X})}\kh(t;z,\gamma z).
\end{align*}
The convergence of the above series follows from the convergence of the hyperbolic heat kernel 
$\kxhyp(t;z)$ and the fact that $\kh(t;z,\gamma z)$ is positive for all $t\in\mathbb{R}_{\geq 0}$, $z\in \mathbb{H}$, and 
$\gamma\in\Gamma_{X}$. 

\paragraph{Selberg constant}
The hyperbolic length of the closed geodesic determined by a primitive non-conjugate hyperbolic element 
$\gamma\in \mathcal{H}(\Gamma_{X})$ on $X$ is given by 
\begin{align*}
\ell_{\gamma}=\inf\lbrace{d_{\mathbb{H}}(z,\gamma z)|\,z\in\mathbb{H}\rbrace}.
\end{align*}
The length of the shortest geodesic $\ell_{X}$ on $X$ is given by
\begin{align*}
\ell_{X}= \inf\big\lbrace{d_{\mathbb{H}}(z,\gamma z)\big|\,
\gamma\in\mathcal{H}(\Gamma_{X}),\,\gamma\,
\mathrm{hyperbolic},\,z\in\mathbb{H}\big\rbrace}.
\end{align*}
From the definition, it is clear that $ \ell_{X}>0.$

\vspace{0.2cm}
For $s\in\mathbb{C}$ with $\Re(s)>1$, the Selberg zeta function associated to $X$ is defined as
\begin{equation*}
Z_{X}(s)= \prod_{\gamma\in \mathcal{H}(\Gamma_{X})}Z_{\gamma}(s), \,\,\,\,\,\,\mathrm{where}\,\,\,Z_{\gamma}(s)= 
\prod_{n=0}^{\infty}\big(1-e^{(s+n)\ell_{\gamma}}\big).
\end{equation*}
The Selberg zeta function $Z_{X}(s)$ admits a meromorphic continuation to all $s\in\mathbb{C}$, with 
zeros and poles characterized by the spectral theory 
of the hyperbolic Laplacian. Furthermore, $Z_{X}(s)$ has a simple zero at 
$s=1$, and the following constant is well-defined
\begin{equation}\label{defnselbergzeta}
c_{X}= \lim_{s\rightarrow 1}\bigg(\frac{Z^{'}_{X}(s)}{Z_{X}(s)}-
\frac{1}{s-1}\bigg).
\end{equation}
For $t\in\mathbb{R}_{\geq 0}$, the hyperbolic heat trace is given by the integral
\begin{align*}
H\mathrm{Tr}\kxhyp(t)=\int_{X}\hkxhyp(t;z)\hyp(z). 
\end{align*}
The convergence of the integral follows from the celebrated Selberg trace formula. Furthermore, from Lemma 4.2 in \cite{jk1}, 
we have the following relation    
\begin{align}\label{selbergconstant}
\int_{0}^{\infty} \big(H\mathrm{Tr}\kxhyp(t)-1\big)dt=c_{X}-1.
\end{align}
\paragraph{Bounds on heat kernels}
There exist constants $c_{0}$ and $c_{\infty}$ such that for $0 < t < t_{0}$ and $\eta\geq 0$, 
we have 
\begin{align*}
K_{\mathbb{H}}(t;\eta)\leq \frac{c_{0}}{4\pi t}e^{-\eta^{2}\slash (4t)};
\end{align*}
furthermore, for $t \geq t_{0}$ and $\eta\geq 0$, we get 
\begin{align}\label{khdecreasing}
 K_{\mathbb{H}}(t;\eta)\leq c_{\infty}e^{-t\slash 4}.
\end{align}
The above two formulae follow directly from the expression for the heat kernel $\kh(t;\eta)$ stated in 
equation (\ref{defnkh}). 
\begin{defn}
We fix a constant $0< \beta < 1\slash 4$, such that for $t\geq t_{0}$ and a fixed $\eta\geq 0$, 
the function 
\begin{align}\label{khdecreasingrem}
e^{\beta t}\kh(t;\eta) 
\end{align}
is a monotone decreasing function in the variable $t$.
\end{defn}
Furthermore, there exists a $\delta_{0} > 0$, such that for $\eta > \delta_{0}$ and a fixed 
$0<t\leq t_{0}$, the function $K_{\mathbb{H}}(t;\eta)$ is a 
monotone decreasing function in the variable $\eta$. We now fix a $\delta_{X}$ 
satisfying $\delta_{X} > \max\,\lbrace{\delta_{0},\,4\ell_{X}+5\rbrace}$. 

\vspace{0.2cm}
As a function in the variable $z$, the sum $\ekxhyp(t_{0},z)+\hkxhyp(t_{0};z)$ remains bounded on $X$ and 
also at the cusps. So we put
\begin{align*}
C^{HK}_{X}=\max_{z\in X}\big(K_{\mathbb{H}}(t_{0};z)+\ekxhyp(t_{0};z)+\hkxhyp(t_{0};z)\big).
\end{align*}
\paragraph{Automorphic Green's function}
For $z, w \in \mathbb{H}$ with $z \not= w$, and  $s$ $\in$ $\mathbb{C}$ with $\Re(s)> 0$, the 
free-space Green's function $g_{\mathbb{H},s}(z,w)$ is defined as
\begin{equation*} 
g_{\mathbb{H},s}(z,w) = g_{\mathbb{H},s}(u(z,w))= \dfrac{\Gamma(s)^{2}}{\Gamma(2s)}u^{-s}
F(s,s;2s,-1\slash u),
\end{equation*} 
where $u=u(z,w)=|z-w|^{2}\slash( 4\Im(z)\Im(w))$ and $F(s,s;2s,-1\slash u)$ is the 
hypergeometric function. 

\vspace{0.2cm}
For $z, w \in\mathbb{H}$ with $z \not = w$ and $s=1$, we put $g_{\mathbb{H}}(z,w)=g_{\mathbb{H},1}(z,w)$, 
and by substituting $s=1$ in the definition of $g_{\mathbb{H},s}(z,w)$, we get
\begin{equation}\label{defngh}
\gh(z,w)=\log\bigg(1+\frac{1}{u(z,w)}\bigg) =\log\bigg{|}\frac{z-\overline{w}}{z-w}\bigg{|}^{2}\geq 0.
\end{equation}
Using the formula from equation (1.3) in \cite{hi}, we get
\begin{align}\label{ghineq}
\cosh(d_{\mathbb{H}}(z,w)= 1 + 2u(z,w)\Longrightarrow 
\gh(z,w)=\log\bigg(1+\frac{1}{\sinh^{2}\big(d_{\mathbb{H}}(z,w)\slash 2\big)}\bigg).
\end{align}
Furthermore, for $z,w\in\mathbb{H}$ with $z\not=w$, we have the following relation 
\begin{align}\label{ghkhreln}
 \gh(z,w)=\int_{0}^{\infty}\kh(t;z,w)dt.
\end{align}
For $z, w \in X$ with $z\not = w$, and $s\in\mathbb{C}$ with $\Re(s) > 1$, the automorphic Green's function 
$g_{X,\mathrm{hyp},s}(z,w)$ is defined as
\begin{equation*}
 g_{X,\mathrm{hyp},s}(z,w) = \sum_{\gamma\in\Gamma_{X}}g_{\mathbb{H},s}(z,\gamma w).
\end{equation*}
The series converges absolutely and uniformly for $z\not = w$ and $\Re(s) > 1$ (see Chapter 5 in \cite{hi}). 

\vspace{0.2cm}
For $z, w \in X$ with $z \not = w$, and $s\in\mathbb{C}$ with $\Re(s) > 1$, the automorphic 
Green's function satisfies the following properties (see Chapters 5 and 6 in \cite{hi}):

\vspace{0.2cm}
(1) The automorphic Green's function $g_{X,\mathrm{hyp},s}(z,w)$ admits a meromorphic continuation to 
all $s\in\mathbb{C}$ with a simple pole at $s=1$ with residue $4\pi\slash\vx(X)$, and the Laurent 
expansion at $s=1$ is of the form
\begin{equation*}
g_{X,\mathrm{hyp},s}(z,w)= \frac{4\pi}{s(s-1)\vx(X) } + 
g^{(1)}_{X,\mathrm{hyp}}(z,w) + O_{z,w}(s-1),
\end{equation*} 
where $g_{X,\mathrm{hyp}}^{(1)}(z,w)$ is the constant term of $g_{X,\mathrm{hyp},s}(z,w)$ at $s=1$. 

\vspace{0.2cm}
(2) Let $p,q\in \mathcal{P}_{X}$ be two cusps. Put
 \begin{align*}
C_{p,q}  = \min \bigg{\lbrace} c > 0\,\bigg{|} \bigg(\begin{array}{ccc} a &b\\
 c & d  \end{array}\bigg) \in \sigma_{p}^{-1}\Gamma_{X} \sigma_{q} \bigg{\rbrace},
 \quad  C_{p,p}=C_{p}.
\end{align*} 
Then, for $z,w\in X$ with $\Im(z) > \Im(w)$ and  $\Im(z)\Im(w) > C_{p,q}^{-2}$, and $s\in\mathbb{C}$ with 
$\Re(s) > 1$, the automorphic Green's function admits the Fourier expansion
\begin{align}\label{fourierautghyp}
&g_{\mathrm{hyp},s}(\sigma_{p}z,\sigma_{q}w)= \frac{4\pi\Im(z)^{1-s}}{2s-1}
\mathcal{E}_{\mathrm{par},p}(\sigma_{q}w,s) +\delta_{p,q}
\sum_{\substack{n\not = 0}}\frac{1}{|n|}W_{s}(nz)\overline{
V_{\overline{s}}(nw)} + O\big(e^{-2\pi(\Im(z)-\Im(w))}\big).  
\end{align}
This equation has been proved as Lemma 5.4 in \cite{hi}, and one of the terms was wrongly estimated in the proof of the lemma. 
We have corrected this error, and stated the corrected equation.  
\paragraph{The space $C_{\ell,\ell\ell}(X)$} 
Let $C_{\ell,\ell\ell}(X)$ denote the set of complex-valued functions $f:X\rightarrow \mathbb{P}^{1}(
\mathbb{C})$, which admit the following type of singularities at finitely many points $\mathrm{Sing}(f)\subset X$, 
and are smooth away from $\mathrm{Sing}(f)$: 

\vspace{0.2cm}
(1) If $s\in\mathrm{Sing}(f)$, then as $z$ approaches $s$, the function $f$ satisfies  
\begin{align}\label{fsingular}
f(z)= c_{f,s}\log|\vartheta_{s}(z)| + O_{z}(1),
\end{align}
for some $c_{f,s}\in \mathbb{C}$.

\vspace{0.2cm}
(2) As $z$ approaches a cusp $p\in \mathcal{P}_{X}$, the function $f$ satisfies
\begin{align}\label{fcusp}
f(z)=c_{f,p}\log\big(-\log|\vartheta_{p}(z)|\big) + O_{z}(1),
\end{align}
for some $c_{f,p}\in \mathbb{C}$. 
\paragraph{Hyperbolic Green's function}
For $z, w \in X$ and $z\not = w$, the hyperbolic Green's function is defined as 
\begin{equation*}
\gxhyp(z,w) = 4\pi\int_{0}^{\infty}\bigg(\kxhyp(t;z,w)-\frac{1}{\vx(X)}\bigg)dt.
\end{equation*}
For $z, w \in X$ with $z \not = w$, the hyperbolic Green's function satisfies the 
following properties:

\vspace{0.2cm}
(1) For $z, w \in X$, the hyperbolic Green's function is uniquely 
determined by the differential equation (which is to be interpreted in terms of currents)
\begin{align}
d_{z}d_{z}^{c}\gxhyp(z,w) +\delta_{w}(z)& = \shyp(z), \label{diffeqnghyp} \\
\intertext{ with the normalization condition} 
\int_{X}\gxhyp(z,w)\hyp(z) & = 0. \label{normcondghyp}
\end{align}
(2) From equation \eqref{diffeqnghyp}, it follows that $\gxhyp(z,w)$ admits a $\log$-singularity 
at $z=w$, i.e., for $z, w\in X$, it satisfies 
\begin{equation}\label{ghypbounded}
\lim_{w\rightarrow z}\big( \gxhyp(z,w) + \log{|\vartheta_{z}(w)|^{2}}\big)= O_{z}(1).
\end{equation}
(3) For $z,w\in X$ and $z\not=w$, we have
\begin{equation}\label{laurentghyp}
 \gxhyp(z,w)= g^{(1)}_{X,\mathrm{hyp}}(z,w)= \lim_{s\rightarrow 1}\bigg(g_{X,\text{hyp},s}(z,w) - 
\frac{4\pi}{s(s-1)\vx(X)}\bigg).
\end{equation}
The above properties follow from the properties of the heat kernel $\kxhyp(t;z,w)$ or from the 
properties of the automorphic Green's function $g_{X,\mathrm{hyp},s}(z,w)$. 

\vspace{0.2cm}
(4) From Proposition 2.1 in \cite{anilpaper}, (or from Proposition 2.4.1 in \cite{anilthesis}) 
for a fixed $ w\in X$, and for $z\in X$ with $\Im(\sigma_{p}^{-1}z)>\Im(\sigma_{p}^{-1} w)$, and 
$\Im(\sigma_{p}^{-1}z)\Im(\sigma_{p}^{-1}w) >C_{p}^{-2}$, we have
\begin{align}
& \gxhyp(z,w) = 4\pi\kappa_{X,p}(w) - 
\frac{4\pi}{\vx(X)}-\frac{4\pi\log\big(\Im(\sigma_{p}^{-1}z)\big)}{\vx(X)}-\notag \\ &
 \log\big{|}1-e^{2\pi i(\sigma_{p}^{-1}z - \sigma_{p}^{-1}w)}\big{|}^{2}+
O\big(e^{-2\pi (\Im(\sigma_{p}^{-1}z)-\Im(\sigma_{p}^{-1}w))}\big),\label{ghypcusp}
\end{align}
i.e., for a fixed $w\in X$, as $z\in X$ approaches a cusp $p\in\mathcal{P}_{X}$, we have
\begin{align*}
\gxhyp(z,w) & = -\frac{4\pi\log\big(\Im(\sigma_{p}^{-1}z)\big)}{\vx(X)}+
O_{z,w}(1) = -\frac{4\pi\log\big(-\log|\vartheta_{p}(z)|\big)}{\vx(X)}+ O_{z,w}(1).
\end{align*}
(5) For any $f\in C_{\ell,\ell\ell}(X)$ and for any fixed $w\in X\backslash \mathrm{Sing}(f)$, from Corollary 2.5 in 
\cite{anilpaper} (or from  Corollary 3.1.8  in \cite{anilthesis}), we have the equality of integrals 
\begin{align}\label{ghypcurrent}
\int_{X}\gxhyp(z,w)d_{z}d_{z}^{c}f(z) + f(w)+
\sum_{s\in \mathrm{Sing}(f)}
\frac{c_{f,s}}{2}\gxhyp(s,w)= \int_{X}f(z)\shyp(z).
\end{align}
\paragraph{An auxiliary identity}
From Definition 8.1 in \cite{K}, for $z\in X\backslash \mathcal{E}_{X}$, we have the following relation
\begin{align*}
4\pi\int_{0}^{\infty}\del \kxhyp(t;z)dt=\sum_{\gamma\in\Gamma_{X}\backslash \lbrace\id\rbrace}\del \gh(z,\gamma z). 
\end{align*}
Furthermore, from Lemmas 5.2 and 6.3, Proposition 7.3, the right-hand side of above equation 
remains bounded at the cusps and at the elliptic fixed points. Hence, as in \cite{anilpaper}, we extend 
Definition 8.1 in \cite{K} and the above relation to cusps and elliptic fixed points to conclude that 
the following quantity is well-defined on $X$ and remains bounded at the cusps and at the elliptic fixed 
points
\begin{align*}
 \int_{0}^{\infty}\del \kxhyp(t;z)dt. 
\end{align*}
\begin{defn}
For notational brevity, put
\begin{align*}
&C_{X,\mathrm{hyp}}= \\&\int_{X}\int_{X}\gxhyp(\zeta,\xi)\bigg(\int_{0}^{\infty}\del 
\kxhyp(t;\zeta)dt\bigg)\bigg(\int_{0}^{\infty}\del \kxhyp(t;\xi)dt\bigg)\hyp(\xi)\hyp(\zeta).
\end{align*}
\end{defn}
From Proposition 2.8 in \cite{anilpaper} (or from Proposition 2.6.4 in \cite{anilthesis}), for 
$z,w\in X$, we have
\begin{equation}\label{phi}
\gxhyp(z,w)-\gxcan(z,w)= \phi_{X}(z) + \phi_{X}(w),
\end{equation}
where from Remark 2.16 in \cite{anilpaper} (or from Corollary 3.2.7 in \cite{anilthesis}), 
the function $\phi_{X}(z)$ is given by the formula 
\begin{align}\label{phi(z)formula}
&\phi_{X}(z)= \frac{1}{2g_{X}}\int_{X}\gxhyp(z,\zeta)
\left(\int_{0}^{\infty}\del \kxhyp(t;\zeta)dt\right)\hyp(\zeta)-\frac{C_{X,\mathrm{hyp}}}{8g_{X}^{2}}.
\end{align} 
\paragraph{Key-identity}
From Corollary 2.15 in \cite{anilpaper} (or from Corollary 3.2.5 in \cite{anilthesis}), for any $f\in C_{\ell,\ell
\ell}(X)$, we have following identity, which is a generalization of Theorem 3.4 from \cite{jk} to cusps and elliptic 
fixed points at the level of currents
\begin{align}\label{keyidentity}
& g\int_{X}f(z)\can(z) = \notag\\&\bigg(\frac{1}{4\pi}+\frac{1}{\vx(X)} \bigg)
\int_{X}f(z)\hyp(z) + \frac{1}{2}\int_{X}f(z)\bigg(\int_{0}^{\infty}
\del \kxhyp(t;z)dt \bigg)\hyp(z).  
\end{align}
\section{Certain convergence results}
In this section, we prove the absolute and uniform convergence of certain series, and compute 
their asymptotics at cusps and at elliptic fixed points. The analysis of this section allows us 
to decompose the integrals involved in \eqref{phi(z)formula} into expressions, which 
we will bound in section 4. 
\subsection{Parabolic case}
\begin{defn}
For $z\in\mathbb{H}$, put
\begin{equation*}
P_{X}(z)= \sum_{\gamma\in\mathcal{P}(\Gamma_{X})}g_{\mathbb{H}}(z,\gamma z).
\end{equation*}
The function $P_{X}(z)$ is invariant under the action of $\Gamma_{X}$, and hence, defines a 
function on $X$ (recall that $\id\not\in \mathcal{P}(\Gamma_{X})$). 
\end{defn}
\begin{lem}\label{lem2}
For $z\in X$, the series $P_{X}(z)$ converges absolutely and uniformly.
\begin{proof}
We have the following decomposition of parabolic elements of $\Gamma_{X}$
\begin{align*}
\mathcal{P}(\Gamma_{X})= \bigcup_{p\in \mathcal{P}_{X}}
\bigcup_{\eta\in\Gamma_{X,p}\backslash\Gamma_{X}}
\big(\eta^{-1}\Gamma_{X,p}\eta\backslash \lbrace\mathrm{id}\rbrace\big)= 
\bigcup_{p\in \mathcal{P}_{X}}\bigcup_{\eta\in\Gamma_{X,p}\backslash\Gamma_{X}}
\bigcup_{n\not = 0}\big\lbrace\eta^{-1}\gamma_{p}^{n}\eta\rbrace,
\end{align*}
where $\gamma_{p}$ is a generator of the stabilizer subgroup $\Gamma_{X,p}$ of the cusp 
$p\in\mathcal{P}_{X}$. This implies that formally, we have
\begin{align} 
P_{X}(z)=&\sum_{\gamma\in\mathcal{P}(\Gamma_{X})}g_{\mathbb{H}}(z,\gamma z) = 
\sum_{p\in \mathcal{P}_{X}}\sum_{\eta\in\Gamma_{X,p}\backslash\Gamma_{X}}
\sum_{ n\not = 0}g_{\mathbb{H}}(z,\eta^{-1}\gamma_{p}^{n}\eta z)\notag\\
=&\sum_{p\in \mathcal{P}_{X}}\sum_{\eta\in\Gamma_{X,p}\backslash\Gamma_{X}}
\sum_{ n\not = 0}g_{\mathbb{H}}(\eta z,\gamma_{p}^{n}\eta z)=\sum_{p\in \mathcal{P}_{X}}
\sum_{\eta\in\Gamma_{X,p}\backslash\Gamma_{X}}P_{\mathrm{gen},p}(\eta z),
\label{decomposition}
\end{align}
where $P_{\mathrm{gen},p}(z)=\displaystyle\sum_{ n\not = 0}g_{\mathbb{H}}(z,\gamma_{p}^{n}z)$. 
We first prove the absolute convergence of the function $P_{\mathrm{gen},p}(z)$. From the definition of $\gh(z,w)$ as given in \eqref{defngh}, for any cusp 
$p\in\mathcal{P}_{X}$, observe that 
\begin{align}
&P_{\mathrm{gen},p}(z)=\sum_{ n\not = 0}g_{\mathbb{H}}(\sigma_{p}^{-1}z,\gamma_{\infty}^{n}
\sigma_{p}^{-1}z)=\sum_{ n\not = 0}\log\bigg{(}\frac{4\Im(\sigma_{p}^{-1}z)^{2}+n^{2}}{n^{2}}\bigg{)}
\leq\ \notag\\ &2\log\big(4\Im(\sigma_{p}^{-1}z)^{2}+1\big) +2\int_{1}^{\infty}\log\bigg{(}
\frac{4\Im(\sigma_{p}^{-1}z)^{2}+t^{2}}{t^{2}}\bigg{)}dt= \notag\\
&4\pi \Im(\sigma_{p}^{-1}z)-8\Im(\sigma_{p}^{-1}z)
\tan^{-1}\left(\frac{1}{2\Im(\sigma_{p}^{-1}z)}\right)\leq 32\Im(\sigma_{p}^{-1}z)^{2},
\label{lem2eqn2}
\end{align}
where $\sigma_{p}$ is a scaling matrix associated to the cusp $p\in\mathcal{P}_{X}$ 
as in \eqref{parscaling} (for the details regarding the computation of the last inequality, we refer 
the reader to Proposition 4.2.3 in \cite{anilthesis}). This proves the absolute convergence of the 
function $P_{\mathrm{gen},p}(z)$. 

\vspace{0.2cm}
Hence, combining equation \eqref{decomposition} with inequality \eqref{lem2eqn2}, we arrive 
at the estimate
\begin{align*}
P_{X}(z)\leq  32\sum_{p\in \mathcal{P}_{X}}\sum_{\eta\in\Gamma_{X,p}\backslash\Gamma_{X}}
\Im(\sigma_{p}^{-1}\eta z)^{2}= 32\sum_{p\in \mathcal{P}_{X}}\mathcal{E}_{X,\mathrm{par},p}(z,2),
\end{align*}
which proves the uniform convergence of the series $P_{X}(z)$. Furthermore, each term of the 
series $P_{X}(z)$ is positive, hence, it converges absolutely.
\end{proof}
\end{lem}
\begin{lem}\label{lem3}
As $z\in X$ approaches a cusp $p\in \mathcal{P}_{X}$, the function $P_{X}(z)$ satisfies the estimate
\begin{align*}
P_{X}(z) = 4\pi\Im(\sigma_{p}^{-1}z)-\log\big(4\Im(\sigma_{p}^{-1}z)^2\big)+ O_{z}(1). 
\end{align*} 
\begin{proof}
Let $z\in X$ approach a cusp $p\in \mathcal{P}_{X}$. From equation 
(\ref{decomposition}), we obtain the decomposition 
\begin{align}\label{lem3eqn1}
&P_{X}(z)=\sum_{\substack{q\in \mathcal{P}_{X}\\q\not = p}}\sum_{\eta\in \Gamma_{X,q}\backslash
\Gamma_{X}}P_{\mathrm{gen},q}(\eta z) +\sum_{\substack{\eta\in \Gamma_{X,p}\backslash\Gamma_{X}
\\\eta\not = \mathrm{id}}}P_{\mathrm{gen},p}(\eta z) + P_{\mathrm{gen},p}(z).
\end{align}
We now estimate the right-hand side of the above equation term by term. Using inequality (\ref{lem2eqn2}), 
we derive the following upper bounds for the first and second terms
\begin{align}
&\sum_{\substack{q\in \mathcal{P}_{X}\\q\not = p}}\sum_{\eta\in \Gamma_{X,q}\backslash\Gamma_{X}}
P_{\mathrm{gen},q}(\eta z)\leq    32\sum_{\substack{q\in \mathcal{P}_{X}\\q\not = p}}\sum_{\eta\in 
\Gamma_{X,q}\backslash\Gamma_{X}}\Im(\sigma_{q}^{-1}\eta z )^{2} =32\sum_{\substack{q\in 
\mathcal{P}_{X}\\q\not = p}}\mathcal{E}_{X,\mathrm{par},q}(z,2);\\
&\sum_{\substack{\eta\in \Gamma_{X,p}\backslash\Gamma_{X}\\\eta\not = \mathrm{id}}}P_{\mathrm{gen},p}
(\eta z)\leq  32\sum_{\substack{\eta\in \Gamma_{X,p}\backslash\Gamma_{X}\\ \eta\not =
\mathrm{id}}}\Im(\sigma_{p}^{-1}\eta z )^{2}=32\big(\mathcal{E}_{\mathrm{par},p}
(z,2)-\Im(\sigma_{p}^{-1}z)^{2}\big).\label{lem3useful}
\end{align}
So using the above upper bounds, for $z\in X$ approaching $p\in \mathcal{P}_{X}$, from equation 
\eqref{fouriereisen}, we have the following estimate for the first and second terms
\begin{align}
&\sum_{\substack{q\in \mathcal{P}_{X}\\q\not = p}}\sum_{\eta\in \Gamma_{X,q}\backslash\Gamma_{X}}
P_{\mathrm{gen},q}(\eta z)+\sum_{\substack{\eta\in \Gamma_{X,p}\backslash\Gamma_{X}\\\eta\not = 
\mathrm{id}}}P_{\mathrm{gen},p}(\eta z)=O\left(\Im(\sigma_{p}^{-1}z)^{-1}\right).\label{lem3eqn2} 
\end{align}
As $z\in X$ approaches $p\in \mathcal{P}_{X}$, we are now left to investigate the behavior of 
the third term
\begin{align}
P_{\mathrm{gen},p}( z)=\sum_{n\neq 0}g_{\mathbb{H}}
(\sigma_{p}^{-1}z,\gamma_{\infty}^{n}\sigma_{p}^{-1}z)=
\lim_{w\rightarrow z}\lim_{s\rightarrow 1}\bigg(\sum_{n=-\infty}^{\infty}g_{\mathbb{H},s}(\sigma_{p}^{-1}w,\gamma_{\infty}^{n}\sigma_{p}^
{-1}z)-g_{\mathbb{H},s}(\sigma_{p}^{-1}z,\sigma_{p}^{-1}w)\bigg).\label{lem3eqn3}
\end{align}
From Lemma 5.1 in Chapter 5 of \cite{hi}, for $\Im({\sigma_{p}^{-1}z)}> \Im(\sigma_{p}^{-1}w)$, and 
$s\in\mathbb{C}$ with $\Re(s) >1$, we have
\begin{align}\label{lem3usefuleqn}
\sum_{n=-\infty}^{\infty}g_{\mathbb{H},s}(\sigma_{p}^{-1}w,\gamma_{\infty}^{n}
\sigma_{p}^{-1}z)=\frac{4\pi}{2s-1}\Im(\sigma_{p}^{-1}w)^{s}\Im(\sigma_{p}^{-1}z)^{1-s} +
\sum_{n\neq0}\frac{1}{|n|}W_{s}(n\sigma_{p}^{-1}z)
\overline{V_{\overline{s}}(n\sigma_{p}^{-1}w)}.
\end{align}
Substituting the above expression in equation (\ref{lem3eqn3}), we get 
\begin{align}
P_{\mathrm{gen},p}( z)= 4\pi\Im(\sigma_{p}^{-1}z)+\lim_{w\rightarrow z}\lim_{s\rightarrow 1}\bigg(
\sum_{n\neq0}\frac{1}{|n|}W_{s}(n\sigma_{p}^{-1}z)\overline{V_{\overline{s}}
(n\sigma_{p}^{-1}w)}-g_{\mathbb{H},s}(\sigma_{p}^{-1}z,\sigma_{p}^{-1}w)
\bigg).\label{lem3eqn4} 
\end{align}
From the proof of Lemma 5.4 in \cite{hi} (there is a slight error in the calculation of this lemma, 
which has been corrected in Corollary 1.9.5 in \cite{anilthesis}), we have the estimate 
\begin{align*}
\sum_{n\neq0}\frac{1}{|n|}W_{s}(n\sigma_{p}^{-1}z)
\overline{V_{\overline{s}}(n\sigma_{p}^{-1}w)}= -\log\big{|}1-e^{2\pi i
(\sigma_{p}^{-1}z-\sigma_{p}^{-1}w)}\big{|}^{2} + O\big(
e^{-2\pi(\Im(\sigma_{p}^{-1}z)-\Im(\sigma_{p}^{-1}w))}\big).
\end{align*}
Using the estimate stated in above equation, we compute 
\begin{align}
&\lim_{w\rightarrow z}\lim_{s\rightarrow 1}\bigg(\sum_{n\neq0}
\frac{1}{|n|}W_{s}(n\sigma_{p}^{-1}z)\overline{V_{\overline{s}}(n\sigma_{p}^{-1}w)}
-g_{\mathbb{H},s}(\sigma_{p}^{-1}z,\sigma_{p}^{-1}w)\bigg)= 
-\log\big(4\Im(\sigma_{p}^{-1}z)^{2}\big)+O_{z}(1).\label{lem3eqn5}
\end{align}
Combining equations (\ref{lem3eqn4}) and (\ref{lem3eqn5}), we arrive at the estimate
\begin{align}
P_{\mathrm{gen},p}(z)=\lim_{w\rightarrow z}\bigg(-\log\big{|}1-e^{2\pi i(\sigma_{p}^{-1}z-
\sigma_{p}^{-1}w)}\big{|}^{2}- \log\bigg{|}\frac{\sigma_{p}^{-1}{z}-\overline{\sigma_{p}^{-1}w}}{\sigma_{p}^{-1}{z}-
\sigma_{p}^{-1}w}\bigg{|}^{2}\bigg)+ O_{z}(1)=\notag\\
4\pi\Im(\sigma_{p}^{-1}z)-\log\big(4\Im(\sigma_{p}^{-1}z)^2\big) + O_{z}(1),\label{lem3usefulbound}
\end{align}
which along with the estimate obtained in equation (\ref{lem3eqn2}) completes the proof of the 
proposition.
\end{proof}
\end{lem}
\begin{rem}
From Lemma 5.2 in \cite{K}, the following series
\begin{align*}
\sum_{\gamma\in\mathcal{P}(\Gamma_{X})}\del g_{\mathbb{H}}(z,\gamma z)  
\end{align*}
converges absolutely and uniformly for all $z\in X$, and the above series remains 
bounded at the cusps of $X$. Furthermore, from the absolute and uniform convergence of the 
series $P_{X}(z)$ and that of the above series, we have the following relations
\begin{align}\label{delP}
&\sum_{\gamma\in\mathcal{P}(\Gamma_{X})}\del g_{\mathbb{H}}(z,\gamma z)  =\del P_{X}(z)= 
\sum_{p\in \mathcal{P}_{X}}\sum_{\eta\in \Gamma_{X,p}\backslash \Gamma_{X}}\del P_{\mathrm{gen},p}(
\eta z),\notag\\& \del P_{\mathrm{gen},p}(z)=\sum_{n\not = 0}\del g_{\mathbb{H}}(\sigma_{p}^{-1} z,\gamma_{\infty}^{n}
\sigma_{p}^{-1}z)= 2\bigg(\frac{2\pi \Im(\sigma_{p}^{-1}z)}{\sinh(2\pi
\Im(\sigma_{p}^{-1}z))}\bigg)^{2}-2.
\end{align}
Put
\begin{align}\label{delpdefn}
 C_{X,\mathrm{par}}^{\mathrm{aux}}=\sup_{z\in X}\big|\del P_{X}(z)\big|.
\end{align}
\end{rem}
\subsection{Elliptic case}
\begin{defn}
For $z\in\mathbb{H}$, put
\begin{align*}
E_{X}(z)=\sum_{\gamma\in\mathcal{E}(\Gamma_{X})}\gh(z,\gamma z). 
\end{align*}
The function is $\Gamma_{X}$-invariant and hence, defines a function on $X$.
\end{defn}
\begin{lem}\label{lem1}
For $z\in X\backslash \mathcal{E}_{X}$, the series $E_{X}(z)$ converges absolutely and uniformly, and as 
$z\in X$ approaches an elliptic fixed point $\mathfrak{e}\in\mathcal{E}_{X}$, we have
\begin{align}\label{lem1eqn}
E_{X}(z)=-\frac{m_{\mathfrak{e}}-1}{m_{\mathfrak{e}}}\log|\vartheta_{\mathfrak{e}}(z)|^{2}+O_{z}(1).
\end{align}
Furthermore, the function $E_{X}(z)$ is zero at the cusps. 
\begin{proof}
We have the following decomposition of elliptic elements of $\Gamma_{X}$
\begin{align*}
\mathcal{E}(\Gamma_{X})=\bigcup_{\mathfrak{e}\in \mathcal{E}_{X}}\bigcup_{\eta\in\Gamma_{X,\mathfrak{e}}
\backslash\Gamma_{X}}\big\lbrace\eta^{-1}\Gamma_{X,\mathfrak{e}}\eta\backslash \lbrace\mathrm{id}\rbrace\big
\rbrace= \bigcup_{\mathfrak{e}\in \mathcal{E}_{X}}\bigcup_{\eta\in\Gamma_{X,\mathfrak{e}}\backslash\Gamma_{X}}
\bigcup_{n=1}^{m_{\mathfrak{e}}-1}\big\lbrace\eta^{-1}\gamma_{\mathfrak{e}}^{n}\eta\rbrace,
\end{align*}
where $\Gamma_{X,\mathfrak{e}}$ denotes the stabilizer subgroup of the elliptic fixed point $\mathfrak{e}
\in\mathcal{E}_{X}$, and $\gamma_{\mathfrak{e}}$ denotes a generator of $\Gamma_{X,\mathfrak{e}}$. 
Using the above decomposition, formally we have
\begin{align}
E_{X}(z)\,&=\,\sum_{\gamma\in\mathcal{E}(\Gamma_{X})}g_{\mathbb{H}}(z,\gamma z)=\sum_{\mathfrak{e}\in \mathcal{E}_{X}}
\sum_{\eta\in\Gamma_{X,\mathfrak{e}}\backslash\Gamma_{X}}\sum_{ n=1}^{m_{\mathfrak{e}}-1}g_{\mathbb{H}}(
z,\eta^{-1}\gamma_{\mathfrak{e}}^{n}\eta z)\notag\\&=\,\sum_{\mathfrak{e}\in \mathcal{E}_{X}}
\sum_{\eta\in\Gamma_{X,\mathfrak{e}}\backslash\Gamma_{X}}\sum_{n=1}^{m_{\mathfrak{e}}-1}g_{\mathbb{H}}
(\sigma_{\mathfrak{e}}^{-1}\eta z,\gamma_{i}^{n}\sigma_{\mathfrak{e}}^{-1}\eta z),\label{lem1eqn1}
\end{align}
where $\sigma_{\mathfrak{e}}$ denotes a scaling matrix of the elliptic fixed point $\mathfrak{e}\in
\mathcal{E}_{X}$ as given in \eqref{ellipticscalingmatrix}. Now for any $\mathfrak{e}\in\mathcal{E}_{X}$, 
$0< n\leq m_{\mathfrak{e}}-1 $, and $\eta\in\Gamma_{X,\mathfrak{e}}
\backslash\Gamma_{X}$, let $w=u+iv$ denote $\sigma_{\mathfrak{e}}^{-1}\eta z$. Using formula (\ref{defngh}) 
and the relation 
\begin{align*}
u^{2}+v^{2}+1=2v\cosh(\rho( w)),
\end{align*}
where $\rho (u)$ denotes $d_{\mathbb{H}}(z,i)$ the hyperbolic distance between the points $z$ and $i$, 
we compute 
\begin{align}
&g_{\mathbb{H}}(w,\gamma_{i}^{n}w)= \log\bigg|\frac{-\sin(n\pi\slash m_{\mathfrak{e}})(|w|^{2}+1)+\cos(n\pi\slash 
m_{\mathfrak{e}})(w-\overline{w})}{-\sin(n\pi\slash m_{\mathfrak{e}})(w^{2}+1)}\bigg|^{2}=\notag\\&
\log\Bigg(\frac{\sin^{2}(n\pi\slash m_{\mathfrak{e}})\cosh^{2}(\rho(w))+\cos^{2}(n\pi\slash m_{\mathfrak{e}})}{\sin^{2}(n\pi\slash m_{\mathfrak{e}})
\cosh^{2}(\rho(w))-\sin^{2}(n\pi\slash m_{\mathfrak{e}})}\Bigg)=\notag\\&\log\bigg(1+\frac{1}{\sin^{2}(n\pi\slash m_{\mathfrak{e}})\sinh^{2}
(\rho(w))}\bigg)\leq \frac{1}{\sin^{2}(n\pi\slash m_{\mathfrak{e}})\sinh^{2}(\rho(w))}.\label{lem1eqn2}
\end{align}
Put
\begin{align}\label{defncellsmall}
c_{X,\mathrm{ell}}=\max\big\lbrace 1\slash\sin^{2}(n\pi\slash m_{\mathfrak{e}})\big|\,\mathfrak{e}
\in\mathcal{E}_{X},0\,< n\leq m_{\mathfrak{e}}-1 \big\rbrace .
\end{align}
Then, from decomposition \eqref{lem1eqn1} and inequality \eqref{lem1eqn2}, we derive
\begin{align}\label{boundE(z)}
E_{X}(z)\leq\sum_{\mathfrak{e}\in \mathcal{E}_{X}}\sum_{\eta\in\Gamma_{X,\mathfrak{e}}\backslash\Gamma_{X}}
\sum_{n=1}^{m_{\mathfrak{e}}-1}\frac{c_{X,\mathrm{ell}}}{\sinh^{2}(\rho(\sigma_{\mathfrak{e}}^{-1}\eta z))}= 
c_{X,\mathrm{ell}}\sum_{\mathfrak{e}\in \mathcal{E}_{X}}(m_{\mathfrak{e}}-1)\eeis(z,2),
\end{align}
which proves the uniform convergence of the series $E_{X}(z)$. Furthermore, each term of the series 
$E_{X}(z)$ is positive, hence, it converges absolutely. The asymptotic relation stated in 
\eqref{lem1eqn} follows trivially from decomposition \eqref{lem1eqn1}.

\vspace{0.2cm}
Moreover, for any $z,w\in\mathbb{H}$ with $z\not = w$, any $\gamma\in \Gamma_{X}\backslash 
\mathcal{P}(\Gamma_{X})$, and any cusp $p\in\mathcal{P}_{X}$, observe that
\begin{align*}
\lim_{z\rightarrow p}\gh(z,\gamma w)=0.
\end{align*}
From the above relation, it trivially follows that the function $E_{X}(z)$ is zero at the cusps.    
\end{proof}
\end{lem}
\begin{rem}\label{remelliptic}
From Lemma \ref{lem1}, it follows that the function $E_{X}(z)$ admits $\log$-singularities at elliptic 
fixed points, and is zero at the cusps. So we can conclude that $E_{X}(z)\in C_{\ell,\ell\ell}
(X)$ with $\mathrm{Sing}(E_{X}(z))=\mathcal{E}_{X}$ and $c_{E_{X},\mathfrak{e}}=-
2(m_{\mathfrak{e}}-1)\slash m_{\mathfrak{e}}$, for any $\mathfrak{e}\in\mathcal{E}_{X}$. 

\vspace{0.2cm}
From Lemma 6.3 in \cite{K}, the following series
\begin{align*}
\sum_{\gamma\in\mathcal{E}(\Gamma_{X})}\del g_{\mathbb{H}}(z,\gamma z)\leq 0  
\end{align*}
converges absolutely and uniformly for all $z\in\mathbb{H}$, and the above series 
remains bounded at the cusps. Furthermore, from the absolute and uniform convergence of the series 
$E_{X}(z)$ and that of the above series, we have the following relation
\begin{align}\label{delE}
\del E_{X}(z)= \sum_{\gamma\in\mathcal{E}(\Gamma_{X})}\del g_{\mathbb{H}}(z,\gamma z)\leq 0.
\end{align}
\end{rem}
\subsection{Hyperbolic case}
\begin{defn}
For $z\in X$, put
\begin{equation}\label{H(z)defneqn}
H_{X}(z)=4\pi\int_{0}^{\infty}\bigg(\hkxhyp(t;z)-\frac{1}{\vx(X)}\bigg)dt.  
\end{equation}
The function $H_{X}(z)$ is invariant under the action of $\Gamma_{X}$, and hence, defines a 
function on $X$.
\end{defn}
\begin{prop}\label{prop4}
The function $H_{X}(z)$ is well-defined on $X$. Moreover it satisfies 
\begin{align}\label{prop4eqn}
H_{X}(z)=  \lim_{w\rightarrow z}\big{(}\gxhyp(z,w)-g_{\mathbb{H}}(z,w)\big{)}-E_{X}(z)-P_{X}(z).
\end{align}
\begin{proof}
From Lemmas \ref{lem2}, \ref{lem1}, we know that the series
\begin{align*}
&P_{X}(z)=\sum_{\gamma\in\mathcal{P}(\Gamma_{X})}g_{\mathbb{H}}(z,\gamma z)=
\sum_{\gamma\in\mathcal{P}(\Gamma_{X})}4\pi\int_{0}^{\infty} K_{\mathbb{H}}(t;z,\gamma z)dt,\\
&E_{X}(z)=\sum_{\gamma\in\mathcal{E}(\Gamma_{X})}g_{\mathbb{H}}(z,\gamma z)=
\sum_{\gamma\in\mathcal{E}(\Gamma_{X})}4\pi\int_{0}^{\infty} K_{\mathbb{H}}(t;z,\gamma z)dt.
\end{align*}
converge absolutely for all $z\in X$, respectively. So, we can interchange summation and integration in the above 
integrals. Moreover, the integral
\begin{align}\label{prop4eqn1}
\int_{0}^{\infty}\bigg(\kxhyp(t;z)-K_{\mathbb{H}}(t;0)-\frac{1}{\vx(X)}\bigg)dt 
\end{align}
converges for all $z\in X$. So we can write
\begin{align}
&H_{X}(z)=4\pi\int_{0}^{\infty}\bigg(\hkxhyp(t;z)-\frac{1}{\vx(X)}\bigg)dt=\notag\\&
4\pi\int_{0}^{\infty}\bigg(\kxhyp(t;z)-K_{\mathbb{H}}(t;0)-\frac{1}{\vx(X)}-
\sum_{\gamma\in\mathcal{E}(\Gamma_{X})}K_{\mathbb{H}}(t;z,\gamma z)-\sum_{\gamma\in\mathcal{P}(
\Gamma_{X})}K_{\mathbb{H}}(t;z,\gamma z)\bigg)dt =\notag\\&4\pi\int_{0}^{\infty}\bigg(
\kxhyp(t;z)-K_{\mathbb{H}}(t;0)-\frac{1}{\vx(X)}\bigg)dt-E_{X}(z)-P_{X}(z),\label{prop4eqn2}
\end{align}
which proves the convergence of the function $H_{X}(z)$. 

\vspace{0.2cm}
From the convergence of the integral in (\ref{prop4eqn1}), and an application of Fatou's lemma from 
real analysis, we can interchange limit and integration in the following expression to derive
\begin{align}
&\lim_{w\rightarrow z}\big{(}\gxhyp(z,w)- g_{\mathbb{H}}(z,w)\big{)}= 
4\pi\int_{0}^{\infty}\bigg(\kxhyp(t;z)-K_{\mathbb{H}}(t;0)-\frac{1}{\vx(X)}\bigg)dt.\label{prop4eqn3} 
\end{align}
Combining equations (\ref{prop4eqn2}) and (\ref{prop4eqn3}) proves equation (\ref{prop4eqn}). 
\end{proof}
\end{prop}
In the following proposition, we describe the behavior of the automorphic function $H_{X}(z)$ at the 
cusps.
\begin{prop}\label{prop5}
As $z\in X$ approaches a cusp $p\in \mathcal{P}_{X}$, we have  
\begin{align*}
&E_{X}(z)+H_{X}(z)= \frac{8\pi\log\big(\Im(\sigma_{p}^{-1}z)\big)}{\vx(X)}-\frac{4\pi}{\vx(X)}+
4\pi k_{p,p}(0)+O\big(\Im(\sigma_{p}^{-1}z)^{-1}\big),
\end{align*}
where $k_{p,p}(0)$ is the zeroth Fourier coefficient in the Fourier expansion of Kronecker's limit 
function $\kappa_{X,p}(z)$ associated to the cusp $p\in \mathcal{P}_{X}$ (see equation 
\eqref{fourierkappaeqn}).
\end{prop}
\begin{proof}
Combining equations (\ref{prop4eqn}) and (\ref{lem3eqn1}), we have
\begin{align*}
&E_{X}(z)+H_{X}(z)=\lim_{w\rightarrow z}\bigg(\gxhyp(z,w)-\sum_{n=-\infty}^{\infty}
g_{\mathbb{H}}(\sigma_{p}^{-1}w,\gamma_{\infty}^{n}\sigma_{p}^{-1}z) \bigg)-\\& \sum_{
\substack{q\in \mathcal{P}_{X}\\q\not = p}}\sum_{\eta\in \Gamma_{X,q}\backslash\Gamma_{X}}
P_{\mathrm{gen},q}(\eta z) -\sum_{\substack{\eta\in \Gamma_{X,p}\backslash\Gamma_{X}\\\eta\not = \mathrm{id}}}P_{\mathrm{gen},p}
(\eta z).
\end{align*}
We now estimate the right-hand side of the above equation term by term. As $z\in X$ approaches the cusp 
$p\in \mathcal{P}_{X}$, from equation \eqref{lem3eqn2}, we arrive at the estimate   
\begin{align}
 E_{X}(z)+H_{X}(z)= &\lim_{w\rightarrow z}\bigg(\gxhyp(z,w)-\sum_{n=-\infty}^{\infty}
g_{\mathbb{H}}(\sigma_{p}^{-1}w,\gamma_{\infty}^{n}\sigma_{p}^{-1}z)\bigg) + 
O\big(\Im(\sigma_{p}^{-1}z)^{-1}\big).\label{prop5eqn1}
\end{align}
We are now left to compute the asymptotics of the limit
\begin{align}
&\lim_{w\rightarrow z}\bigg(g_{\mathrm{hyp}}(z,w)-\sum_{n=-\infty}^{\infty}g_{\mathbb{H}}
(\sigma_{p}^{-1}w,\gamma_{\infty}^{n}\sigma_{p}^{-1}z)\bigg)=\notag
\\&\lim_{w\rightarrow z}\lim_{s\rightarrow 1}\bigg(g_{\mathrm{hyp},s}(z,w)-
\frac{4\pi}{s(s-1)\vx(X)}-\sum_{n=-\infty}^{\infty}
g_{\mathbb{H},s}(\sigma_{p}^{-1}w,\gamma_{\infty}^{n}\sigma_{p}^{-1}z)\bigg).
\label{prop5eqn2}
\end{align}
As $z\in X$ approaches $p\in \mathcal{P}_{X}$, combining estimates \eqref{fourierautghyp} and 
\eqref{lem3usefuleqn}, we have 
\begin{align*}
&g_{X,\mathrm{hyp},s}(z,w)-\sum_{n=-\infty}^{\infty}g_{\mathbb{H},s}(\sigma_{p}^{-1}w,
\gamma_{\infty}^{n}\sigma_{p}^{-1}z)=\\&\frac{4\pi\Im(\sigma_{p}^{-1}z)^{1-s}}{2s-1}\mathcal{E}_{X,
\mathrm{par},p}(w,s)-\frac{4\pi}{2s-1}\Im(\sigma_{p}^{-1}w)^{s}\Im(\sigma_{p}^{-1}z)^{1-s}+ 
O\big(e^{-2\pi\Im(\sigma_{p}^{-1}z)}\big).
\end{align*}
Using the above expression, we find that the right-hand side of limit (\ref{prop5eqn2}) can be 
written as 
\begin{align*}
&\lim_{w\rightarrow z}\lim_{s\rightarrow 1}\bigg(\frac{4\pi\Im(\sigma_{p}^{-1}z)^{1-s}}{2s-1}
\mathcal{E}_{X,\mathrm{par},p}(w,s)-\frac{4\pi}{(s-1)\vx(X)}\bigg)+\\&\frac{4\pi}{\vx(X)}-
4\pi\Im(\sigma_{p}^{-1}z)+O\big(e^{-2\pi\Im(\sigma_{p}^{-1}z)}\big).
\end{align*}
To evaluate the above limit, we compute the Laurent expansions of 
$\mathcal{E}_{\mathrm{par},p}(w,s)$, $\Im(\sigma_{p}^{-1}z)^{1-s}$, and 
$(2s-1)^{-1}$ at $s=1$. The Laurent expansions 
of $\Im{(\sigma_{p}^{-1}z)^{1-s}}$ and $(2s-1)^{-1}$ at $s=1$ are easy to compute, 
and are of the form
\begin{align*}
&\Im{(\sigma_{p}^{-1}z)}^{1-s}=1 - (s-1)\log\big(\Im{(\sigma_{p}^{-1}z)}\big)
+ O\big((s-1)^{2}\big), \,\,\,\,\frac{1}{2s-1}  = 1- 2(s-1) + O\big((s-1)^{2}\big).
\end{align*}
Using the Laurent expansion of the Eisenstein series $\mathcal{E}_{\mathrm{par},p}(w,s)$ from equation 
\eqref{laurenteisenpar}, and combining it with above expressions, we compute
\begin{align}
&\lim_{w\rightarrow z}\bigg(g_{\mathrm{hyp}}(z,w)-\sum_{n=-\infty}^{\infty}g_{\mathbb{H}}
(\sigma_{p}^{-1}w,\gamma_{\infty}^{n}\sigma_{p}^{-1}z)\bigg)=4\pi \kappa_{X,p}(z)-
4\pi\Im(\sigma_{p}^{-1}z)-\notag\\&\frac{4\pi \log\big(\Im(\sigma_{p}^{-1}z)
\big)}{\vx(X)}-\frac{4\pi}{\vx(X)}+O\big(e^{-2\pi\Im(\sigma_{p}^{-1}z)}\big).\label{prop5eqn3}
\end{align}
From the Fourier expansion of Kronecker's limit function $\kappa_{X,p}(z)$ described 
in \eqref{fourierkappaeqn}, we have 
\begin{align*}
\kappa_{X,p}(z)  = \Im(\sigma_{p}^{-1}z)+k_{p,p}(0)-\frac{\log\big(\Im(\sigma_{p}^{-1}z)\big)}{\vx(X)}+
O\big(e^{-2\pi\Im(\sigma_{p}^{-1}z)}\big).
\end{align*}
As $z\in X$ approaches $p\in \mathcal{P}_{X}$, substituting the above estimate in the right-hand side 
of equation (\ref{prop5eqn3}), and combining it with equation \eqref{prop4eqn1}, we arrive at
\begin{align*}
&E_{X}(z)+H_{X}(z)= -\frac{8\pi\log\big(\Im(\sigma_{p}^{-1}z)\big)}{\vx(X)}-\frac{4\pi}{\vx(X)}+
4\pi k_{p,p}(0)+O\big(\Im(\sigma_{p}^{-1}z)^{-1}\big),
\end{align*}
which completes the proof of the proposition.
\end{proof}
\begin{rem}\label{remhyperbolic}
As the function $E_{X}(z)$ is zero at the cusps, from Proposition \ref{prop5}, we can conclude that 
$H_{X}(z)$ has $\log\log$-growth at the cusps. Moreover, the function $H(z)$ remains smooth for all 
$z\in X$. Hence, $H_{X}(z)\in C_{\ell,\ell\ell}(X)$ with $\mathrm{Sing}(H_{X}(z))=\emptyset$.

\vspace{0.20cm}
Furthermore, from equation \eqref{selbergconstant}, it follows that
\begin{equation}\label{H(z)trace}
\int_{X}H_{X}(z)\hyp(z)= 4\pi(c_{X}-1). 
\end{equation}
Using equation (\ref{prop4eqn}), we get
\begin{align*}
\del P_{X}(z)+ \del E_{X}(z)+\del H_{X}(z)= \del\lim_{w\rightarrow z}\big(g_{X,\mathrm{hyp}}(z,w)-
g_{\mathbb{H}}(z,w)\big).
\end{align*}
Since the integral
\begin{align*}
4\pi\int_{0}^{\infty}\bigg(\kxhyp(t;z,z)-K_{\mathbb{H}}(t;0)-\frac{1}{\vx(X)}\bigg)dt,
\end{align*}
as well as the integral of the derivatives of the integrand are absolutely 
convergent, we can take the Laplace operator $\del$ inside the integral. So we find
\begin{align}\label{delkdecomposition}
\del P_{X}(z)+\del E_{X}(z)+\del H_{X}(z)= 4\pi\int_{0}^{\infty}\del K_{X,\mathrm{hyp}}(t;z)dt.
\end{align}
\end{rem}
\begin{cor}\label{cor6}
For any $z\in X\backslash \mathcal{E}_{X}$, we have 
\begin{align*}
&\phi_{X}(z)=\frac{\big(H_{X}(z)+E_{X}(z)\big)}{2g_{X}}+\frac{1}{8\pi g_{X}}\int_{X}\gxhyp(z,\zeta)
\del P_{X}(\zeta)\hyp(\zeta)-\notag\\&\sum_{\mathfrak{e}\in\mathcal{E}_{X}}
\frac{m_{\mathfrak{e}}-1}{2g_{X}m_{\mathfrak{e}}}\gxhyp(z,\mathfrak{e})-\frac{C_{X,\mathrm{hyp}}}
{8g_{X}^{2}}-\frac{2\pi(c_{X}-1)}{g_{X}\vx(X)}-\frac{1}{2g_{X}}\int_{X}E_{X}(\zeta)\shyp(\zeta).
\end{align*}
\begin{proof}
Using formula \eqref{ddcreln}, and combining equations \eqref{phi(z)formula} and 
\eqref{delkdecomposition}, we have
\begin{align}
&\phi_{X}(z)= \frac{1}{2g_{X}}\int_{X}\gxhyp(z,\zeta)
\big(-d_{\zeta}d_{\zeta}^{c}\big(E_{X}(\zeta)+H_{X}(\zeta)\big)\big)+\notag\\&\frac{1}{8\pi g_{X}}\int_{X}
\gxhyp(z,\zeta)\del P_{X}(\zeta)\hyp(z)-\frac{C_{X,\mathrm{hyp}}}{8g_{X}^{2}}.\label{coreqn1}
\end{align} 
From Remarks \ref{remelliptic} and \ref{remhyperbolic}, we know that the functions $E_{X}(z)$ and 
$H_{X}(z)$ both belong to $C_{\ell,\ell\ell}(X)$ with $\mathrm{Sing}(E_{X}(z))=\mathcal{E}_{X}$ and 
$\mathrm{Sing}(H_{X}(z))=\emptyset$, respectively. Hence, from equation \eqref{ghypcurrent}, for any 
$z\in X\backslash \mathcal{E}_{X}$, we have the following relations 
\begin{align*}
-&\int_{X}\gxhyp(z,\zeta)d_{\zeta}d_{\zeta}^{c}E_{X}(\zeta)= \frac{E_{X}(z)}{2g_{X}}-
\sum_{\mathfrak{e}\in\mathcal{E}_{X}}\frac{m_{\mathfrak{e}}-1}{2g_{X}m_{\mathfrak{e}}}\gxhyp(z,\mathfrak{e})
-\frac{1}{2g_{X}}\int_{X}E_{X}(\zeta)\shyp(\zeta),\\-&\int_{X}\gxhyp(z,\zeta)d_{\zeta}d_{\zeta}^{c}
H_{X}(\zeta)=\frac{H_{X}(z)}{2g_{X}}-\frac{1}{2g_{X}}\int_{X}H_{X}(\zeta)\shyp(\zeta).
\end{align*}
Substituting the above two equations in equation \eqref{coreqn1} and using relation \eqref{H(z)trace} completes 
the proof of the corollary. 
\end{proof}
\end{cor}
\section{Bounds for hyperbolic Green's function}
In this section, we derive bounds for the hyperbolic Green's functions on compact subsets of 
$X$, and in the neighborhoods of cusps and elliptic fixed points. 

\vspace{0.2cm}
We begin by defining a compact subset $Y_{\varepsilon}$, for some $0<\varepsilon<1$, and we 
adapt the existing bounds for the hyperbolic heat kernel from \cite{jk}. We then use these 
bounds to bound the hyperbolic Green's function both on the compact subset $Y_{\varepsilon}$, 
and in the neighborhood of cusps and elliptic fixed points. 
\subsection{Bounds for hyperbolic Green's function}\label{section3.1}
\begin{notn}\label{epsilondefn}
For any $\delta > 0$ and a fixed $z,w\in X$, identifying $X$ with its fundamental domain, we 
define the set 
\begin{align*}
S_{\Gamma_{X}}(\delta;z,w)=\big\lbrace{\gamma\in\mathcal{H}(\Gamma_{X})\cup\lbrace\id\rbrace \big|\,d_{\mathbb{H}}(z,
\gamma w) <\delta\big\rbrace}.
\end{align*}
Let $0<\varepsilon <\min\lbrace 1, \ell_{X}\rbrace$ be any number such that the following 
conditions holds true:

\vspace{0.2cm}
(1) For any cusp $p\in \mathcal{P}_{X}$, let $U_{\varepsilon}(p)$ denote an open coordinate disk 
of radius $\varepsilon$ around $p$. Then, we have $\Im(\sigma_{p}^{-1}z)\geq \Im(\sigma_{p}^{-1}
\gamma z)$, where $\sigma_{p}$ is a scaling matrix of the cusp $p$. Furthermore, for 
$p,q\in\mathcal{P}_{X}$ and $p\not =q$, we have
\begin{align*}
U_{\varepsilon}(p)\cap U_{\varepsilon}(q)=\emptyset.
\end{align*}

\vspace{0.2cm}
(2) For any elliptic fixed point $\mathfrak{e}\in \mathcal{E}_{X}$, let $U_{\varepsilon}(\mathfrak{e})$ 
denote an open coordinate disk around $\mathfrak{e}$ such that $d_{\mathbb{H}}(z,\mathfrak{e})=\varepsilon$ for all 
$z\in \partial U_{\varepsilon}(\mathfrak{e})$. Furthermore for $\mathfrak{e},\mathfrak{f}\in
\mathcal{E}_{X}$ and $\mathfrak{e}\not =\mathfrak{f}$, we have
\begin{align*}
U_{\varepsilon}(\mathfrak{e})\cap U_{\varepsilon}(\mathfrak{f})=\emptyset.
\end{align*}

\vspace{0.2cm}
(3) For any elliptic fixed point $\mathfrak{e}\in\mathcal{E}_{X}$, $z\in \partial U_{\varepsilon}(
\mathfrak{e})$ and $\gamma\in\Gamma_{X}$, we have
\begin{align*}
 d_{\mathbb{H}}(z,\gamma\mathfrak{e})\geq \varepsilon.
\end{align*}
Furthermore, for any $p\in \mathcal{P}_{X}$ and any $\mathfrak{e}\in\mathcal{E}_{X}$, we have
\begin{align*}
 U_{\varepsilon}(p)\cap U_{\varepsilon}(\mathfrak{e})=\emptyset.
\end{align*}
We fix an $\varepsilon $ satisfying the above three conditions and put
\begin{align*}
Y_{\varepsilon}= X\backslash\Bigg(\bigcup_{p\in\mathcal{P}_{X}}U_{\varepsilon}(p)\cup 
\bigcup_{\mathfrak{e}\in \mathcal{E}_{X}}U_{\varepsilon}(\mathfrak{e})\Bigg), \quad
Y^{\mathrm{par}}_{\varepsilon}= X\backslash\Bigg(\bigcup_{p\in\mathcal{P}_{X}}U_{\varepsilon}(p)\Bigg),
\quad Y^{\mathrm{ell}}_{\varepsilon}= X\backslash\Bigg(\bigcup_{\mathfrak{e}\in\mathcal{E}_{X}}U_{\varepsilon}
(\mathfrak{e})\Bigg).
\end{align*}
Furthermore, for any cusp $p\in\mathcal{P}_{X}$, any elliptic fixed point $\mathfrak{e}\in\mathcal{E}_{X}$, put
\begin{align*}
Y^{\mathrm{par}}_{\varepsilon,p}=X\backslash U_{\varepsilon}(p), \quad Y^{\mathrm{ell}}_{\varepsilon,\mathfrak{e}}=
X\backslash U_{\varepsilon}(\mathfrak{e}),
\end{align*}
respectively. For brevity of notation, we identify the fundamental domains associated to the compact subsets 
$Y_{\varepsilon}$, $Y^{\mathrm{par}}_{\varepsilon}$, and  $Y^{\mathrm{ell}}_{\varepsilon}$ again by the same symbols. 
\end{notn}
The computations carried out in the following two remarks will come handy in the calculations 
that follow. 
\begin{lem}\label{remellipticalc1}
Let $\mathfrak{e}\in\mathcal{E}_{X}$ be an elliptic fixed point. Then, for any $\gamma\in\Gamma_{X}$, and 
$z\in \partial U_{\varepsilon}(\mathfrak{e})$, we have the following upper bound
\begin{align}\label{remellipticalceqn5}
 \sinh^{2}\big(d_{\mathbb{H}}(z,\gamma z)\slash 2\big)\leq 7\coth(\varepsilon\slash 2)
 \sinh^{2}\big(d_{\mathbb{H}}(z,\gamma \mathfrak{e})\slash 2\big).
\end{align}
\begin{proof}
For $z\in \partial U_{\varepsilon}(\mathfrak{e})$ and any $\gamma\in\Gamma_{X}$, from condition (3), which the fixed $\varepsilon$ satisfies, we 
have 
\begin{align}
 d_{\mathbb{H}}(z,\gamma \mathfrak{e}) \geq \varepsilon&\Longrightarrow \frac{\sinh^{2}\big(d_{\mathbb{H}}
 (z,\gamma \mathfrak{e})\slash 2\big)}{\sinh^{2}(\varepsilon\slash 2)}\geq 1;\label{remellipticalceqn1}\\
d_{\mathbb{H}}(z,\gamma z)\leq d_{\mathbb{H}}(z,\gamma\mathfrak{e})+d_{\mathbb{H}}(\gamma z,\gamma\mathfrak{e})
=d_{\mathbb{H}}(z,\gamma\mathfrak{e})+\varepsilon&\Longrightarrow \sinh^{2}\big(d_{\mathbb{H}}(z,
\gamma z)\slash 2\big)\leq\sinh^{2}\big(d_{\mathbb{H}}(z,\gamma\mathfrak{e})\slash 2\big).\label{remellipticalceqn2}
\end{align}
For any $z\in \partial U_{\varepsilon}(\mathfrak{e})$ and $\gamma\in\Gamma_{X}$, observe that 
\begin{align}
 \sinh^{2}\big((d_{\mathbb{H}}(z,\gamma\mathfrak{e})
+\varepsilon)\slash 2\big)=\sinh^{2}\big(d_{\mathbb{H}}(z,\gamma\mathfrak{e})\slash 2\big)\cosh^{2}
(\varepsilon\slash 2)+\notag\\\cosh^{2}\big(d_{\mathbb{H}}(z,\gamma\mathfrak{e})\slash 2\big)\sinh^{2}
(\varepsilon\slash 2)+\sinh\big(d_{\mathbb{H}}(z,\gamma\mathfrak{e})\slash 2\big)\cosh\big(d_{\mathbb{H}}
(z,\gamma\mathfrak{e})\slash 2\big)\sinh(\varepsilon)= \notag\\ 2\sinh^{2}\big(d_{\mathbb{H}}(z,\gamma
\mathfrak{e})\slash 2\big)\cosh^{2}(\varepsilon\slash 2)+\sinh^{2}(\varepsilon\slash 2)+\sinh
\big(d_{\mathbb{H}}(z,\gamma\mathfrak{e})\slash 2\big)\cosh\big(d_{\mathbb{H}}
(z,\gamma\mathfrak{e})\slash 2\big)\sinh(\varepsilon)\label{remellipticalceqn3}.
\end{align}
Using inequality \eqref{remellipticalceqn1} and the fact that $\sinh\big(d_{\mathbb{H}}(z,\gamma\mathfrak{e})\slash 2\big)
\leq \cosh\big(d_{\mathbb{H}}(z,\gamma\mathfrak{e})\slash 2\big)$, we estimate the second and third terms on the 
right-hand side of above equation
\begin{align*}
\sinh^{2}(\varepsilon\slash 2)+\sinh\big(d_{\mathbb{H}}(z,\gamma\mathfrak{e})\slash 2\big)\cosh\big(
d_{\mathbb{H}}(z,\gamma\mathfrak{e})\slash 2\big)\sinh(\varepsilon)\leq\notag\\ \sinh^{2}\big(d_{\mathbb{H}}(z,
\gamma \mathfrak{e})\slash 2\big)+\frac{ \sinh^{2}\big(d_{\mathbb{H}}(z,\gamma \mathfrak{e})\slash 2\big)}{\sinh^{2}
(\varepsilon\slash 2)}\sinh(\varepsilon)+\sinh^{2}\big(d_{\mathbb{H}}(z,\gamma\mathfrak{e})\slash 2\big)
\sinh(\varepsilon).
\end{align*}
Combining equation \eqref{remellipticalceqn3} with the above inequality, and using the fact that 
$0<\varepsilon<1$ (which implies that $0<\sinh(\varepsilon\slash 2)+\cosh(\varepsilon\slash 2)<2$, 
and $1<\cosh(\varepsilon\slash 2)< \cot(\varepsilon\slash 2)$), we find
\begin{align}
\sinh^{2}\big((d_{\mathbb{H}}(z,\gamma\mathfrak{e})+\varepsilon)\slash 2\big)\leq  \sinh^{2}\big(
d_{\mathbb{H}}(z,\gamma \mathfrak{e})\slash 2\big)\big(1+2\cosh^{2}(\varepsilon\slash 2)+2\coth(\varepsilon\slash 2)+\sinh
(\varepsilon)\big)\leq\notag\\\sinh^{2}\big(d_{\mathbb{H}}(z,\gamma \mathfrak{e})\slash 2\big)\big(
3\coth(\varepsilon\slash 2)+2\cosh(\varepsilon\slash 2)\big(\sinh(\varepsilon\slash 2)+
\cosh(\varepsilon\slash 2)\big)\big)\leq\notag\\ 7\coth(\varepsilon\slash 2)\sinh^{2}\big(d_{\mathbb{H}}(z,\gamma \mathfrak{e})
\slash 2\big).\label{remellipticalceqn4}
\end{align}
Finally combining the above upper bound with inequality \eqref{remellipticalceqn1} completes the proof of the lemma. 
\end{proof}
\end{lem}
\begin{lem}\label{remellipticalc2}
Let $\mathfrak{e}\in\mathcal{E}_{X}$ be an elliptic fixed point. Then, for any $\gamma\in\Gamma_{X}$,  
$z\in \partial U_{\varepsilon\slash 2}(\mathfrak{e})$, and $w\in \partial U_{\varepsilon}(
\mathfrak{e})$, we have the following upper bound
\begin{align}\label{remellipticalc2eqn3}
\sinh^{2}\big(d_{\mathbb{H}}(z,\gamma z)\slash 2\big)\leq 14\coth(\varepsilon\slash 4)\sinh^{2}
\big(d_{\mathbb{H}}(z,\gamma w)\slash 2\big).
\end{align}
\begin{proof}
For any $\gamma\in\Gamma_{X}$, $z\in \partial U_{\varepsilon\slash 2}(\mathfrak{e})$, and 
$w\in \partial U_{\varepsilon}(\mathfrak{e})$, from the choice of $\varepsilon$ (i.e., condition (3) 
which the fixed $\varepsilon$ satisfies), we have
\begin{align}
&d_{\mathbb{H}}(z,\gamma w)+ d_{\mathbb{H}}(z,\mathfrak{e})\geq d_{\mathbb{H}}(\gamma w,\mathfrak{e})
\Longrightarrow d_{\mathbb{H}}(z,\gamma w)\geq \varepsilon\slash 2 \Longrightarrow\frac{\sinh^{2}\big(d_{\mathbb{H}}
(z,\gamma w)\slash 2\big)}{\sinh^{2}(\varepsilon\slash 4)}\geq 1;\label{remellipticalc2eqn1}\\
&d_{\mathbb{H}}(z,\gamma z)\leq d_{\mathbb{H}}(z,\gamma w)+d_{\mathbb{H}}(\gamma w,\gamma z)\leq 
d_{\mathbb{H}}(z,\gamma w)+\varepsilon\Longrightarrow \notag\\&\sinh^{2}\big(d_{\mathbb{H}}(z,\gamma z)
\slash 2\big)\leq \sinh^{2}\big((d_{\mathbb{H}}(z,\gamma w)+\varepsilon)\slash 2\big).\label{remellipticalc2eqn2}
\end{align}
Using computation \eqref{remellipticalceqn3} from Lemma \ref{remellipticalc1}, we have
\begin{align*}
&\sinh^{2}\big((d_{\mathbb{H}}(z,\gamma w)+\varepsilon)\slash 2\big)=  2\sinh^{2}\big(d_{\mathbb{H}}(z,
\gamma w)\slash 2\big)\cosh^{2}(\varepsilon\slash 2)+\notag\\&\sinh^{2}(\varepsilon\slash 2)+\sinh\big(d_{\mathbb{H}}(z,
\gamma w)\slash 2\big)\cosh\big(d_{\mathbb{H}}(z,\gamma w)\slash 2\big)\sinh(\varepsilon).
\end{align*}
Using inequality \eqref{remellipticalc2eqn1}, and the fact that $\sinh\big(d_{\mathbb{H}}(z,
\gamma w)\slash 2\big)\leq \cosh\big(d_{\mathbb{H}}(z,\gamma w)\slash 2\big)$, 
we arrive at
\begin{align*}
&\sinh^{2}\big((d_{\mathbb{H}}(z,\gamma w)+\varepsilon)\slash 2\big)\leq 
 \\&\sinh^{2}\big(d_{\mathbb{H}}(z,\gamma w)\slash 2\big)\bigg(2\cosh^{2}(\varepsilon\slash 2 )+
 \frac{\sinh^{2}(\varepsilon\slash 2)}{\sinh^{2}(\varepsilon\slash 4)}+\sinh(\varepsilon)+\frac{\sinh(\varepsilon)}{
\sinh^{2}(\varepsilon\slash 4)}\bigg)=\\&\sinh^{2}\big(d_{\mathbb{H}}(z,\gamma w)\slash 2\big)
\bigg(2\cosh^{2}(\varepsilon\slash 2 )+4\cosh^{2}(\varepsilon\slash 4)+\sinh(\varepsilon)+4\coth(\varepsilon\slash 4)\cosh(\varepsilon\slash 2)\bigg)
\end{align*}
Using the fact that  $0< \varepsilon < 1$ (which implies that $\cosh^{2}(\varepsilon\slash 4)\leq \cosh^{2}(\varepsilon\slash 2)$, $
\cosh(\varepsilon\slash 2)\leq 1.13$, $\sinh(\varepsilon)\leq 1.18 $, and $1 <\coth(\varepsilon\slash 4)$), 
we arrive at the following estimate 
\begin{align*}
&\sinh^{2}\big((d_{\mathbb{H}}(z,\gamma w)+\varepsilon)\slash 2\big)\leq 
14\coth(\varepsilon\slash 4)\sinh^{2}\big(d_{\mathbb{H}}(z,\gamma w)\slash 2\big),
\end{align*}
which together with inequality \eqref{remellipticalc2eqn2} completes the proof of the lemma. 
\end{proof}
\end{lem}
\begin{defn}
From equations \eqref{fouriereisen} and \eqref{eeisbound}, it follows that the following quantities are 
well-defined 
\begin{align}
&C_{X,\mathrm{par}}=\sup_{z\in X}\sum_{p\in\mathcal{P}_{X}}\big(\epar(z,2)-
\Im(\sigma_{p}^{-1}z)^{2}\big) ,\label{cpardefn}\\& C_{X,\mathrm{ell}}=\sup_{z\in X}
c_{X,\mathrm{ell}}\sum_{\mathfrak{e}\in\mathcal{E}_{X}}(m_{\mathfrak{e}}-1)\big(\eeis(z,2)-
\sinh^{-2}\big(\rho(\sigma_{\mathfrak{e}}^{-1}z)\big)\big).\label{ceisdefn}
\end{align}
\end{defn}
\begin{lem}
We have the following upper bounds
\begin{align}
&\sup_{z\in Y_{\varepsilon}^{\mathrm{par}}}P_{X}(z)\leq 
-6|\mathcal{P}_{X}|\,\log\varepsilon+32C_{X,\mathrm{par}}\label{estimateP}\\&
\sup_{z\in Y_{\varepsilon}^{\mathrm{ell}}}E_{X}(z)\leq -\sum_{\mathfrak{e}\in\mathcal{E}_{X}}\,
(m_{\mathfrak{e}}-1)\,\log\big(\tanh^{2}(\varepsilon)\slash c_{X,\mathrm{ell}}\big)
+C_{X,\mathrm{ell}}.\label{estimateE}
\end{align}
\begin{proof}
Combining estimate \eqref{cpardefn} with the estimates from the proof of Lemma 
\ref{lem3} (estimate \eqref{lem3useful}), we arrive at the following upper bound
\begin{align*}
&\sup_{z\in Y_{\varepsilon}^{\mathrm{par}}}P_{X}(z)\leq 32\sum_{p\in\mathcal{P}_{X}}\bigg(\Im(\sigma_{p}^{-1}z)^{2}+  
32\big(\mathcal{E}_{X,\mathrm{par},p}(z,2)-\Im(\sigma_{p}^{-1}z)^{2}\big)\bigg)\leq\notag \\-
&\frac{16|\mathcal{P}_{X}|\log\varepsilon}{\pi}+32C_{X,\mathrm{par}}\leq 
-6|\mathcal{P}_{X}|\,\log\varepsilon+32C_{X,\mathrm{par}},
\end{align*}
which proves \eqref{estimateP}. 

\vspace{0.2cm}
Combining estimate \eqref{ceisdefn} with the estimates from the proof of Lemma \ref{lem1} (estimates 
\eqref{lem1eqn2} and \eqref{boundE(z)}), and using the fact that $c_{X,\mathrm{ell}}\geq 1$, we arrive at the following estimate
\begin{align}
&\sup_{z\in Y_{\varepsilon}^{\mathrm{ell}}}E_{X}(z)\leq\sup_{z\in Y_{\varepsilon}^{\mathrm{ell}}}
\sum_{\mathfrak{e}\in\mathcal{E}_{X}}\sum_{n=1}^{m_{\mathfrak{e}}-1}\log\bigg(1+\frac{1}{\sin^{2}(n\pi\slash m_{\mathfrak{e}})\sinh^{2}(\rho(\sigma_{
\mathfrak{e}}^{-1}z))}\bigg)+\notag\\&\sup_{z\in Y_{\varepsilon}^{\mathrm{ell}}}c_{X,\mathrm{ell}}
\sum_{\mathfrak{e}\in\mathcal{E}_{X}}\bigg((m_{\mathfrak{e}}-1)\big(\eeis(z,2)-\sinh^{-2}\big(\rho(\sigma_{
\mathfrak{e}}^{-1}z)\big)\big)\bigg)\leq\notag\\ &\sup_{z\in Y_{\varepsilon}^{\mathrm{ell}}}
\bigg(-\sum_{\mathfrak{e}\in\mathcal{E}_{X}}\,(m_{\mathfrak{e}}-1)\,\log\big(
\tanh^{2}(\rho(\sigma_{\mathfrak{e}}^{-1}z))\slash c_{X,\mathrm{ell}}\big)\bigg)+C_{X,\mathrm{ell}}.\label{lemestimateeqn1}
\end{align}
For any $\mathfrak{e}\in\mathcal{E}_{X}$, from condition (2) which the fixed $\varepsilon$ 
satisfies, we find
\begin{align}
&\sup_{z\in Y_{\varepsilon}^{\mathrm{ell}}}\bigg(-\log\big(\tanh^{2}(\rho(
\sigma_{\mathfrak{e}}^{-1}z))\slash c_{X,\mathrm{ell}}\big)\bigg)=\sup_{z\in Y_{\varepsilon}^{\mathrm{ell}}}
\bigg(-\log\big(\tanh^{2}(d_{\mathbb{H}}(z,\mathfrak{e}))\slash c_{X,\mathrm{ell}}\big)\bigg)
\leq \notag\\&\sup_{z\in\partial U_{\varepsilon}(\mathfrak{e})}\bigg(-\log\big(
\tanh^{2}(d_{\mathbb{H}}(z,\mathfrak{e}))\slash c_{X,\mathrm{ell}}\big)\bigg)=
-\log\big(\tanh^{2}(\varepsilon)\slash c_{X,\mathrm{ell}}\big)
.\label{lemestimateeqn2}
\end{align}
Combining inequalities \eqref{lemestimateeqn1} and \eqref{lemestimateeqn2}, establishes upper 
bound \eqref{estimateE}. 
\end{proof}
\end{lem}
\begin{defn}
With notation as in section 1, for any $\delta \geq \delta_{X}$, $\alpha > 0$, and $z,w \in Y_{\varepsilon}$, put 
\begin{align*}
&K_{X,\mathrm{hyp}}^{\alpha,\delta}(t;z,w)= \\&K_{X,\mathrm{hyp}}(t;z,w)-\sum_{n:\,0\leq\lambda_{X,n} < \alpha}
\varphi_{X,n}(z)\varphi_{X,n}(w)e^{-\lambda_{X,n}t}-\sum_{\gamma\in S_{\Gamma_{X}}(\delta;z,w)}
K_{\mathbb{H}}(t;d_{\mathbb{H}}(z,\gamma w)).
\end{align*}
\end{defn}
The following theorem is an adaption of Lemma 4.2 in \cite{jk} to the case where $X$ admits cusps 
and elliptic fixed points. 
\begin{lem}\label{lem3.1}
For any $\alpha\in(0,\lambda_{X,1})$, $\delta \geq \delta_{X}$, and 
$z,w \in Y_{\varepsilon}$, we have the following  upper bounds:

\vspace{0.2cm}(a) 
For $0 < t < t_{0}$, then
\begin{align}
&\big{|}K_{X,\mathrm{hyp}}^{\alpha,\delta}(t;z,w)\big{|}\leq\notag \\&\frac{1}{\vx(X)}+ 
\frac{c_{0}\sinh(\ell_{X})\sinh(\delta)}{8\delta^{2}\sinh^{2}(\ell_{X}\slash 2)}+ 
\frac{c_{0}e^{2\ell_{X}}}{2\pi \sinh^{2}(\ell_{X}\slash 2)}+\sum_{\gamma\in\mathcal{P}(\Gamma_{X})}\kh(t;z,\gamma w)+
\sum_{\gamma\in\mathcal{E}(\Gamma_{X})}\kh(t;z,\gamma w);\label{lem3.1eqn1}
\end{align}
(b) If $t\geq t_{0}$, then
\begin{align}\label{lem3.1eqn2}
&\big{|}K_{X,\mathrm{hyp}}^{\alpha,\delta}(t;z,w)\big{|}\leq\frac{1}{2}\big(\pkxhyp(t;z)+\pkxhyp(t;w) \big) +e^{-\beta
(t-t_{0})}\, C_{X}^{HK}+\frac{c_{\infty} \sinh(\delta + \ell_{X})\,e^{-t\slash 4}}{\sinh(\ell_{X})}.
\end{align}
\end{lem}
\begin{proof}
For any $\alpha\in(0,\lambda_{X,1})$, $\delta \geq \delta_{X}$, $z,w \in Y_{\varepsilon}$, and 
$0 < t < t_{0}$, adapting the arguments from the proof of Lemma 4.2 in \cite{jk}, we have
\begin{align*}
&\big{|}K_{X,\mathrm{hyp}}^{\alpha,\delta}(t;z,w)\big{|}\leq \\&\frac{1}{\vx(X)}+ \sum_{\gamma\not\in
S_{\Gamma_{X}}(\delta;z,w)}K_{\mathbb{H}}(t;z,\gamma w)+ \sum_{\gamma\in\mathcal{P}(\Gamma_{X})}\kh(t;z,\gamma w)+
\sum_{\gamma\in\mathcal{E}(\Gamma_{X})}\kh(t;z,\gamma w).
\end{align*}
Estimate \eqref{lem3.1eqn1} now follows from restricting the arguments from the same proof to hyperbolic elements of 
$\Gamma_{X}$, and from the observation that the length of the shortest geodesic $\ell_{X}$ corresponds to the 
injectivity radius $r_{X}$ in the proof of Lemma 4.2 in \cite{jk}.   

For notational brevity, put
\begin{align*}
K(t;z)= \sum_{n=1}^{\infty}\varphi_{X,n}(z)\varphi_{X,n}(w)
e^{-\lambda_{X,n}t}+\frac{1}{4\pi}\sum_{p\in\mathcal{P}_{X}}\int_{0}^{\infty}\big|\epar\big(z,1\slash 2+ir\big)
\big|^{2}e^{-(r^{2}+ 1\slash4)t}dr.
\end{align*}
For $t\geq t_{0}$, again from the proof of Lemma 4.2 in \cite{jk}, we have
\begin{align*}
&\big{|}K_{X,\mathrm{hyp}}^{\alpha,\delta}(t;z,w)\big{|}\leq \frac{1}{2}\big(K(t;z)+K(t;w)\big)
+\sum_{\gamma\in S_{\Gamma_{X}}(\delta;z,w)}K_{\mathbb{H}}(t;d_{\mathbb{H}}(z,\gamma w))\leq \\&
\frac{1}{2}\big(\kxhyp(t;z)+\kxhyp(t;w) \big)+\sum_{\gamma\in S_{\Gamma_{X}}(\delta;z,w)}K_{\mathbb{H}}
(t;d_{\mathbb{H}}(z,\gamma w)).
\end{align*}
Adapting the arguments from the proof of Lemma 4.2 in \cite{jk} to $\mathcal{H}(\Gamma_{X})$, 
we find 
\begin{align*}
\sum_{\gamma\in S_{\Gamma_{X}}(\delta;z,w)}K_{\mathbb{H}}
(t;d_{\mathbb{H}}(z,\gamma w))\leq \frac{c_{\infty} \sinh(\delta + \ell_{X})\,e^{-t\slash 4}}{\sinh(\ell_{X})}.
\end{align*}
Now it suffices to show that
\begin{align*}
\kxhyp(t;z)=\pkxhyp(t;z)+\big(\kh(t;0)+\ekxhyp(t;z) +\hkxhyp(t;z)\big)\leq\\\pkxhyp(t;z)+ 
e^{-\beta (t-t_{0})}\, C_{X}^{HK}.
\end{align*}
As in the proof of Lemma 4.2 in \cite{jk}, put
\begin{align}
h(t;z)=e^{\beta t} \big(\kh(t;0)+\ekxhyp(t;z) +\hkxhyp(t;z)\big). 
\end{align}
From equation \eqref{khdecreasingrem}, for a fixed $z\in Y_{\varepsilon}$, it follows that for all 
$t\geq t_{0}$, the function $h(t;z)$ is a monotone decreasing function in $t$. Hence, 
following arguments as in the proof of Lemma 4.2 in \cite{jk}, we arrive at 
\begin{align*}
&\big(\kh(t;0)+\ekxhyp(t;z) +\hkxhyp(t;z)\big)\leq \\&e^{-\beta (t-t_{0})}\big(\kh(t_{0};0)+
\ekxhyp(t_{0};z) +\hkxhyp(t_{0};z)\big)\leq e^{-\beta (t-t_{0})}\, C_{X}^{HK}, 
\end{align*}
which completes the proof of the lemma.
\end{proof}
\begin{prop}\label{prop3.2}
For any $\alpha\in(0,\lambda_{X,1})$, $\delta > 0$, and $z,w \in Y_{\varepsilon}$, we have the following  upper bound
\begin{align*}
\bigg|\gxhyp(z,w)-\sum_{\gamma\in S_{\Gamma_{X}}(\delta;z,w)}g_{\mathbb{H}}(z,\gamma w)\bigg|\leq \boundb, 
\end{align*}
where for $\delta \geq\delta_{X}$, we have
\begin{align*}
\boundb =4\pi\bigg(\frac{1}{\vx(X)} + \frac{c_{0}\sinh(\ell_{X})\sinh(\delta)}{8\delta^{2}\sinh^{2}(\ell_{X}\slash 2)}+
\frac{c_{0}e^{2\ell_{X}}}{2\pi\sinh^{2}(\ell_{X}\slash 2)}+\frac{4c_{\infty}\sinh(\delta+ \ell_{X})}{\sinh(\ell_{X})}+
\frac{C_{X}^{HK}}{\beta}\bigg)+\\ 7\,|\mathcal{P}_{X}|\,(\log\varepsilon)^{2}+41\, C_{X,\mathrm{par}}+
14\coth\big(\varepsilon\slash 4\big)\bigg(-\sum_{\mathfrak{e}\in
\mathcal{E}_{X}}\,(m_{\mathfrak{e}}-1)\,\log\big(\tanh^{2}( \varepsilon\slash 2)\slash c_{X,\mathrm{ell}}
\big)+C_{X,\mathrm{ell}}\bigg);  
\end{align*}
and for $\delta \leq\delta_{X}$, we have
\begin{align*}
B_{X,\varepsilon,\alpha,\delta}=B_{X,\varepsilon,\alpha,\delta_{X}}+\frac{\sinh(\delta_{X}+
\ell_{X})}{\sinh(\ell_{X})}\big|\log\big(\tanh^{2}(\delta\slash 2)\big)\big|.
\end{align*}
\begin{proof}
For any $\alpha\in(0,\lambda_{X,1})$, $\delta > 0$, and 
$z,w \in Y_{\varepsilon}$, we have
\begin{align*}
 \bigg|\gxhyp(z,w)-\sum_{\gamma\in S_{\Gamma_{X}}(\delta;z,w)}g_{\mathbb{H}}(z,\gamma w)\bigg|=
\int_{0}^{t_{0}} \big{|}K_{\mathrm{hyp}}^{\alpha,\delta}(t;z,w)\big{|}dt+\int_{t_{0}}^{\infty} 
\big{|}K_{\mathrm{hyp}}^{\alpha,\delta}(t;z,w)\big{|}dt.
\end{align*}
From Lemma \ref{lem3.1}, and using the fact that the heat kernel $\kh(t;\eta)$ is positive for all 
$t\geq 0$ and $\eta \geq 0$, and that $0<t_{0}< 1$, we have the following inequality
\begin{align*}
&\bigg|\gxhyp(z,w)-\sum_{\gamma\in S_{\Gamma_{X}}(\delta;z,w)}g_{\mathbb{H}}(z,\gamma w)\bigg|\leq
\sup_{z,w\in Y_{\varepsilon}}\bigg(P_{X}(z)+\sum_{\gamma\in\mathcal{P}(\Gamma_{X})}\gh(z,\gamma w)+
\sum_{\gamma\in\mathcal{E}(\Gamma_{X})}\gh(z,\gamma w)\bigg)\\&+
4\pi\bigg(\frac{1}{\vx(X)} + \frac{c_{0}\sinh(\ell_{X})\sinh(\delta)}{8\delta^{2}\sinh^{2}(\ell_{X}\slash 2)}+
\frac{c_{0}e^{2\ell_{X}}}{2\pi\sinh^{2}(\ell_{X}\slash 2)}+\frac{4c_{\infty}\sinh(\delta+ \ell_{X})}{\sinh(\ell_{X})}+
\frac{C_{X}^{HK}}{\beta}\bigg).
\end{align*}
For $z,w\in Y_{\varepsilon}$, we are left to bound the term 
\begin{align}\label{prop3.2exp}
P_{X}(z)+\sum_{\gamma\in\mathcal{P}(\Gamma_{X})}\gh(z,\gamma w)+
\sum_{\gamma\in\mathcal{E}(\Gamma_{X})}\gh(z,\gamma w). 
\end{align}
From upper bound \eqref{estimateP}, we have the following upper bound for the first term 
\begin{align}\label{prop3.2estimate1}
\sup_{z\in Y_{\varepsilon}}P_{X}(z)\leq 
\sup_{z\in Y_{\varepsilon}^{\mathrm{par}}}P_{X}(z)\leq -6\,|\mathcal{P}_{X}|\,\log\varepsilon+32\,C_{X,\mathrm{par}}.
\end{align}
Now, for $z\in Y_{\varepsilon\slash 2}^{\mathrm{par}}$, a fixed $w\in Y_{\varepsilon}^{\mathrm{par}}$, and $z\not =w$, 
observe that
\begin{align*}
\del \sum_{\gamma\in\mathcal{P}(\Gamma_{X})}\gh(z,\gamma w) =0;
\end{align*}
from equation \eqref{delP}, for $z=w$, we find that
\begin{align*}
\del \sum_{\gamma\in\mathcal{P}(\Gamma_{X})}\gh(z,\gamma z)=\del P_{X}(z) \leq 0.
\end{align*}
Hence, for $z\in Y_{\varepsilon\slash 2}^{\mathrm{par}}$, and a fixed $w\in Y_{\varepsilon}^{\mathrm{par}}$, the second 
term in expression \eqref{prop3.2exp} is a superharmonic function in the variable $z$. So from the maximum principle for superharmonic 
functions, we deduce that
\begin{align*}
\sup_{z,w\in Y_{\varepsilon}} \sum_{\gamma\in\mathcal{P}(\Gamma_{X})}\gh(z,\gamma w)\leq 
\sup_{\substack{z\in Y_{\varepsilon\slash 2}^{\mathrm{par}}\\w\in Y_{\varepsilon}^{\mathrm{par}}}}
\sum_{\gamma\in\mathcal{P}(\Gamma_{X})}\gh(z,\gamma w)\leq \sup_{\substack{z\in \partial U_{\varepsilon\slash 2}(p)
\\w\in Y_{\varepsilon}^{\mathrm{par}}}} \sum_{\gamma\in\mathcal{P}(\Gamma_{X})}\gh(z,\gamma w),
\end{align*}
for some cusp $p\in\mathcal{P}_{X}$. From the definition of $\gh(z,w)$ from \eqref{defngh} and from condition 
(1) which the fixed $\varepsilon$ satisfies, for any $\gamma\in\Gamma_{X}$, $z\in \partial U_{
\varepsilon\slash 2}(p)$ and $w\in Y_{\varepsilon}^{\mathrm{par}}$, we derive
\begin{align*}
&\gh(z,\gamma w)= \gh(\sigma_{p}^{-1}z,\sigma_{p}^{-1}\gamma w)=\log\bigg(1+\frac{4\Im(\sigma_{p}^{-1}z)\Im(\sigma_{p}^{-1}
\gamma w)}{|\sigma_{p}^{-1} z-\sigma_{p}^{-1}\gamma w|^{2}}\bigg)\leq \\&\log\bigg(1+ \frac{4\Im(
\sigma_{p}^{-1}z)^{2}}{\big(\Im(\sigma_{p}^{-1}z)-\Im(\sigma_{p}^{-1}\gamma w)\big)^{2}} \bigg)\leq \frac{4
\Im(\sigma_{p}^{-1}z)^{2}}{(\log 2)^{2}}\leq 9\Im(\sigma_{p}^{-1}z)^{2},
\end{align*}
where $\sigma_{p}$ is a scaling matrix for the cusp $p\in\mathcal{P}_{X}$. Using the above inequality, 
we arrive at 
\begin{align}
&\sup_{\substack{z\in \partial U_{\varepsilon\slash 2}(p)\\w\in Y_{\varepsilon}^{\mathrm{par}}}} \sum_{\gamma\in\mathcal{P}(\Gamma_{X})}
\gh(z,\gamma w)\leq\sup_{z\in \partial U_{\varepsilon\slash 2}(p)}9\sum_{\gamma\in\mathcal{P}(\Gamma_{X})}
\Im(\sigma_{p}^{-1}\gamma z)^{2}=\sup_{z\in \partial U_{\varepsilon\slash 2}(p)}9\sum_{p\in \mathcal{P}_{X}}
\Im(\sigma_{p}^{-1} z)^{2}+\notag\\&\sup_{z\in \partial U_{\varepsilon\slash 2}(p)}9\sum_{p\in \mathcal{P}_{X}}\big(
\epar(z,2)-\Im(\sigma_{p}^{-1} z)^{2}\big)\leq |\mathcal{P}_{X}|\,\big(\log(\varepsilon\slash 2)\big)^{2}
+9\,C_{X,\mathrm{par}}.\label{prop3.2estimate2}
\end{align}
Hence, combining upper bounds \eqref{prop3.2estimate1} and \eqref{prop3.2estimate2}, and using the fact 
that $0<\varepsilon < 1$ (which implies that $-\log\varepsilon\leq (\log(\varepsilon\slash 2)^{2}$), we arrive at the following 
upper bound for the first two terms in expression \eqref{prop3.2exp}
\begin{align}\label{prop3.2estimate3}
P_{X}(z)+\sum_{\gamma\in\mathcal{P}(\Gamma_{X})}\gh(z,\gamma w)\leq 
7\,|\mathcal{P}_{X}|\,\big(\log(\varepsilon\slash 2)\big)^{2}+41\,C_{X,\mathrm{par}}. 
\end{align}
For $z\in Y_{\varepsilon\slash 2}^{\mathrm{ell}}$, a fixed $w\in Y_{\varepsilon}^{\mathrm{ell}}$, and $z\not = w$, 
observe that 
\begin{align*}
\del \sum_{\gamma\in\mathcal{E}(\Gamma_{X})} \gh(z,\gamma w)=0;
\end{align*}
from equation \eqref{delE}, for $z=w$, we find that
\begin{align*}
\del \sum_{\gamma\in\mathcal{E}(\Gamma_{X})} \gh(z,\gamma z)\leq 0.
\end{align*}
Hence, for $z\in Y_{\varepsilon\slash 2}^{\mathrm{ell}}$, and a fixed $w\in Y_{\varepsilon}^{\mathrm{ell}}$, the third 
term in the expression \eqref{prop3.2exp} is a superharmonic function in the variable $z$. So from the 
maximum principle for superharmonic functions, we deduce that
\begin{align*}
\sup_{z,w\in Y_{\varepsilon}}\sum_{\gamma\in\mathcal{E}(\Gamma_{X})} \gh(z,\gamma w)\leq 
\sup_{\substack{z\in\partial Y_{\varepsilon\slash 2}^{\mathrm{ell}}\\w\in 
Y_{\varepsilon,\mathfrak{e}}^{\mathrm{ell}}}}\sum_{\gamma\in\mathcal{E}(\Gamma_{X})}
\gh(z,\gamma w)=\sup_{\substack{z\in\partial U_{\varepsilon\slash 2}(\mathfrak{e})\\w\in Y_{
\varepsilon,\mathfrak{e}}^{\mathrm{ell}}}}\sum_{\gamma\in\mathcal{E}(\Gamma_{X})}\gh(z,\gamma w),
\end{align*}
for some elliptic fixed point $\mathfrak{e}\in\mathcal{E}_{X}$. Similarly for $w\in Y_{\varepsilon,
\mathfrak{e}}^{\mathrm{ell}}$ and a fixed $z\in  U_{\varepsilon\slash 2}(\mathfrak{e})$, the third term in expression 
\eqref{prop3.2exp} is a superharmonic function in the variable $w$. Hence, we arrive at 
\begin{align*}
\sup_{\substack{z\in\partial U_{\varepsilon\slash 2}(\mathfrak{e})\\w\in Y_{
\varepsilon,\mathfrak{e}}^{\mathrm{ell}}}}\sum_{\gamma\in\mathcal{E}(\Gamma_{X})}\gh(z,\gamma w)=
\sup_{\substack{z\in\partial U_{\varepsilon\slash 2}(\mathfrak{e})\\w\in \partial U_{\varepsilon}
(\mathfrak{e})}}\sum_{\gamma\in\mathcal{E}(\Gamma_{X})}\gh(z,\gamma w).
\end{align*}
From equation \eqref{ghineq}, recall that 
\begin{align*}
\sum_{\gamma\in\mathcal{E}(\Gamma_{X})}\gh(z,\gamma w)= \sum_{\gamma\in\mathcal{E}(\Gamma_{X})}
\log\bigg(1+\frac{1}{\sinh^{2}\big(d_{\mathbb{H}}(z,\gamma w)\slash 2\big)}
\bigg). 
\end{align*}
Combining upper bound \eqref{remellipticalc2eqn3} from Lemma \ref{remellipticalc2} with upper bound 
\eqref{estimateE}, for any $\gamma\in\Gamma_{X}$, $z\in \partial U_{\varepsilon\slash 2}(\mathfrak{e})$, 
and $w\in\partial U_{\varepsilon}(\mathfrak{e})$, we derive 
\begin{align*}
\sum_{\gamma\in\mathcal{E}(\Gamma_{X})}\gh(z,\gamma w)\leq \sum_{\gamma\in\mathcal{E}(\Gamma_{X})}
\log\bigg(1+\frac{14\coth(\varepsilon\slash 4)}{\sinh^{2}\big(d_{\mathbb{H}}(z,\gamma z)\slash 2\big)}\bigg) 
\leq \sup_{z\in\partial U_{\varepsilon\slash 2}(\mathfrak{e})}14\coth(\varepsilon\slash 4)\,E(z)\leq\\
14\coth\big(\varepsilon\slash 4\big)\bigg(-\sum_{\mathfrak{e}\in
\mathcal{E}_{X}}\,(m_{\mathfrak{e}}-1)\,\log\big(\tanh^{2}( \varepsilon\slash 2)\slash c_{X,\mathrm{ell}}
\big)+C_{X,\mathrm{ell}}\bigg).
\end{align*}
Combining the above inequality with upper bound \eqref{prop3.2estimate3} completes the proof of the proposition.
\end{proof}
\end{prop}
\begin{notn}
For the rest of this article, put
\begin{align}\label{vartilde}
\widetilde{\varepsilon}=2\log\Bigg(\frac{1+\sqrt{1+\big(3\log(\varepsilon\slash 2)\big)^{2}}}{3\log(\varepsilon\slash 2)}\Bigg). 
\end{align}
\end{notn}
\begin{cor}\label{cor3.8}
For any $\alpha\in(0,\lambda_{X,1})$, $\delta \in (0, \widetilde{\varepsilon})$, $z\in \partial 
Y_{\varepsilon\slash 2}^{\mathrm{par}}$, and $w \in Y_{\varepsilon}$, we have the following  upper bound 
\begin{align*}
\big|\gxhyp(z,w)\big|\leq \boundbtwo.
\end{align*}
\begin{proof}
Without loss of generality, we may assume that $z\in \partial U_{\varepsilon\slash 2}(p)$, for some 
cusp $p\in\mathcal{P}_{X}$. For any $\gamma\in\Gamma_{X}$, $z\in \partial U_{\varepsilon\slash 2}(p)$, and 
$w \in Y_{\varepsilon}$, recall that 
\begin{align}
u(z,\gamma w)=\sinh^{2}\big(d_{\mathbb{H}}(z,\gamma w)\slash 2\big)= \frac{|z-\gamma w|^{2}}{4\Im(z)\Im(\gamma w)}\geq 
\frac{|\Im(z)-\Im(\gamma w)|^{2}}{4\Im(z)\Im(\gamma w)}.
\end{align}
From condition (1), which the fixed $\varepsilon$ satisfies, we derive
\begin{align*}
&\sinh^{2}\big(d_{\mathbb{H}}(z,\gamma w)\slash 2\big) \geq 
\frac{\big(\log(\varepsilon)-\log(\varepsilon\slash 2)\big)^{2}}{4\big(\log(\varepsilon\slash 2)\big)^{2}}
\Longrightarrow \sinh \big(d_{\mathbb{H}}(z,\gamma w)\slash 2\big)\geq \frac{1}{3\log(\varepsilon\slash
2)}.
\end{align*}
From the above inequality, it follows that for any $\gamma\in\Gamma_{X}$, $z\in\partial  
 U_{\varepsilon\slash 2}(p)$, and $w \in Y_{\varepsilon}$, we get
$d_{\mathbb{H}}(z,\gamma w)\geq \widetilde{\varepsilon}$. Now for any $\alpha\in(0,\lambda_{X,1})$ 
and $\delta \in (0, \widetilde{\varepsilon})$, from Proposition \ref{prop3.2}, we arrive at
\begin{align*}
\sup_{\substack{z\in\partial U_{\varepsilon\slash 2}(p)\\w\in Y_{\varepsilon}}}\bigg|\gxhyp(z,w)-
\sum_{\gamma\in S_{\Gamma_{X}}(\delta;z,w)}g_{\mathbb{H}}(z,\gamma w)\bigg|\leq\sup_{z,w\in 
Y_{\varepsilon\slash 2}}\big|\gxhyp(z,w)\big|\leq \boundbtwo,
\end{align*}
which completes the proof of the corollary. 
\end{proof}
\end{cor}
\begin{cor}\label{cor3.9}
Let $\mathfrak{e}\in\mathcal{E}_{X}$ be an elliptic fixed point. Then, for any 
$\alpha\in (0,\lambda_{X,1})$, $\delta\in (0,\varepsilon)$, and $z\in Y_{\varepsilon}$, we have the 
following  upper bound 
\begin{align*}
\big|\gxhyp(z,\mathfrak{e})\big|\leq \boundb.
\end{align*} 
\begin{proof}
For any $\alpha\in (0,\lambda_{X,1})$, $\delta\in (0,\varepsilon)$, and $z\in Y_{\varepsilon}$, from 
condition (3) which the fixed 
$\varepsilon$ satisfies, we find
\begin{align*}
\bigg|\gxhyp(z,\mathfrak{e})-\sum_{\gamma\in S_{\Gamma_{X}}(\delta;z,\mathfrak{e})}g_{\mathbb{H}}(z,\gamma \mathfrak{e})\bigg|=
\big|\gxhyp(z,\mathfrak{e})\big|.
\end{align*}
Following similar arguments as in the proof of Proposition \ref{prop3.2}, we get
\begin{align*}
&\big|\gxhyp(z,\mathfrak{e})\big|\leq 
\sup_{z\in Y_{\varepsilon}}\bigg(P_{X}(z)+\sum_{\gamma\in\mathcal{P}(\Gamma_{X})}\gh(z,\gamma \mathfrak{e})+
\sum_{\gamma\in\mathcal{E}(\Gamma_{X})}\gh(z,\gamma \mathfrak{e})\bigg)+\\&
4\pi\bigg(\frac{1}{\vx(X)} + \frac{c_{0}\sinh(\ell_{X})\sinh(\delta)}{8\delta^{2}\sinh^{2}(\ell_{X}\slash 2)}+
\frac{c_{0}e^{2\ell_{X}}}{2\pi\sinh^{2}(\ell_{X}\slash 2)}+\frac{4c_{\infty}\sinh(\delta+ \ell_{X})}{\sinh(\ell_{X})}+
\frac{C_{X}^{HK}}{\beta}\bigg).
\end{align*}
We estimate the first two terms on the right-hand side of above inequality by the same quantities as 
in the proof of Proposition \ref{prop3.2}. For the third term, from similar arguments as in the proof 
of Proposition \ref{prop3.2}, and using the upper bound from Lemma \ref{remellipticalc1} (i.e., 
estimate \eqref{remellipticalceqn5}), we derive
\begin{align*}
\sup_{z\in Y_{\varepsilon}}\sum_{\gamma\in\mathcal{E}(\Gamma_{X})}\gh(z,\gamma \mathfrak{e})&=
\sup_{z\in \partial U_{\varepsilon}(\mathfrak{e})}\sum_{\gamma\in\mathcal{E}(\Gamma_{X})}\gh(z,\gamma \mathfrak{e})
\leq \sup_{z\in \partial U_{\varepsilon}(\mathfrak{e})} \sum_{\gamma\in\mathcal{E}(\Gamma_{X})}
\log\bigg(1+\frac{7\coth(\varepsilon\slash 2)}{\sinh^{2}\big(d_{\mathbb{H}}(z,\gamma z)\slash 2\big)}\bigg) 
\\&\leq \sup_{z\in\partial U_{\varepsilon}(\mathfrak{e})}7\coth(\varepsilon\slash 2)\,E(z)\leq 
\sup_{z\in\partial U_{\varepsilon\slash 2}(\mathfrak{e})}14\coth(\varepsilon\slash 4)\,E(z),
\end{align*}
which can be bounded again by the same estimate as in the proof of Proposition \ref{prop3.2}. 
Hence, we deduce that for hypothesis as in the statement of the corollary, we have the same 
bound for $\big|\gxhyp(z,\mathfrak{e})\big|$ as in Proposition \ref{prop3.2}, i.e., $\boundb$, which 
completes the proof of the corollary. 
\end{proof}
\end{cor}
\begin{cor}\label{cor3.10}
Let $p\in\mathcal{P}_{X}$ be any cusp. Then, for any $\alpha\in (0,\lambda_{X,1})$, 
$\delta > 0$, $z\in Y_{\varepsilon}^{\mathrm{par}}$, and $w\in U_{\varepsilon}(p)$, we have 
\begin{align*}
\gxhyp(z,w)-\sum_{\gamma\in S_{\Gamma_{X}}(\delta;z,w)}g_{\mathbb{H}}(z,\gamma w)=-\frac{4\pi}{\vx(X)}
\log\bigg(\frac{\log|\vartheta_{p}(w)|}{\log\varepsilon}\bigg) +h_{\delta,p}(z,w),
\end{align*} 
where $h_{\delta,p}(z,w)$ is a harmonic function in the variable $w\in U_{\varepsilon}(p)$, 
which satisfies the following upper bound
\begin{align*}
\sup_{z\in U_{\varepsilon}(p)}\big| h_{\delta,p}(z,w)\big|\leq \boundb.
\end{align*}
\begin{proof}
For any $\delta > 0$, a fixed $z\in Y_{\varepsilon}^{\mathrm{par}}$, and $w\in U_{\varepsilon}(p)$, 
both the functions
\begin{align*}
\gxhyp(z,w)-\sum_{\gamma\in S_{\Gamma_{X}}(\delta;z,w)}g_{\mathbb{H}}(z,\gamma w)\,,\,\,\,
\,\,\,
-\frac{4\pi}{\vx(X)}\log\bigg(\frac{\log|\vartheta_{p}(z)|}{\log\varepsilon}\bigg)
\end{align*}
are solutions of differential equation \eqref{diffeqnghyp}. So we find that
\begin{align*}
\gxhyp(z,w)-\sum_{\gamma\in S_{\Gamma_{X}}(\delta;z,w)}g_{\mathbb{H}}(z,\gamma w)=
-\frac{4\pi}{\vx(X)}\log\bigg(\frac{\log|\vartheta_{p}(z)|}{\log\varepsilon}\bigg) +h_{\delta,p}(z,w),
\end{align*}
where $h_{\delta,p}(z,w)$ is a harmonic function in the variable $z\in U_{\varepsilon}(p)$. 

\vspace{0.2cm}
As $h_{\delta,p}(z,w)$ is a harmonic function, $|h_{\delta,p}(z,w)|$ is a subharmonic function. 
So for a fixed $z\in Y_{\varepsilon}^{\mathrm{par}}$, from the maximum principle for subharmonic 
functions and Proposition \ref{prop3.2}, we arrive at the upper bound
\begin{align*}
&\sup_{w\in U_{\varepsilon}(p)}\big{|}h_{\delta,p}(z,w)\big{|}=\sup_{w\in\partial U_{\varepsilon}(p)}
\big{|}h_{\delta,p}(z,w)\big{|}=\bigg|\gxhyp(z,w)-\sum_{\gamma\in S_{\Gamma}(\delta;z,w)}
g_{\mathbb{H}}(z,\gamma w)\bigg|\leq B_{\varepsilon,\alpha,\delta},
\end{align*}
for any $\alpha\in (0,\lambda_{X,1})$ and $\delta > 0$. The proof of the corollary follows from the fact that 
the upper bound derived above does not depend on the fixed $z\in Y_{\varepsilon}^{\mathrm{par}}$.
\end{proof}
\end{cor}
\begin{cor}\label{cor3.11}
Let $p,q\in\mathcal{P}_{X}$ and $p\not =q$ be two cusps. Then, for any $\alpha\in (0,\lambda_{X,1})$, 
$\delta > 0$, $z\in U_{\varepsilon}(p)$, and $w\in U_{\varepsilon}(q)$, we have
\begin{align*}
&\gxhyp(z,w)-\sum_{\gamma\in S_{\Gamma_{X}}(\delta;z,w)}g_{\mathbb{H}}(z,\gamma w)=\\-&\frac{4\pi}{\vx(X)}\log\bigg(\frac{\log|\vartheta_{p}(w)|}{\log\varepsilon}\bigg) 
-\frac{4\pi}{\vx(X)}\log\bigg(\frac{\log|\vartheta_{q}(w)|}{\log\varepsilon}\bigg) +
h_{\delta,p,q}(z,w),
\end{align*}
where $h_{\delta,p,q}(z,w)$ is a harmonic function in both the variables $z\in U_{\varepsilon}(p)$ and $w\in U_{\varepsilon}(q)$, which 
satisfies the following upper bound
\begin{align*}
\sup_{\substack{z\in U_{\varepsilon}(p)\\z\in U_{\varepsilon}(q)}}
\big| h_{\delta,p,q}(z,w)\big|\leq \boundb.
\end{align*}
\begin{proof}
The proof of the corollary follows from similar arguments as in Corollary \ref{cor3.10}. 
\end{proof}
\end{cor}
\begin{cor}\label{cor3.12}
Let $p\in\mathcal{P}_{X}$ be any cusp. Then, for any $\alpha\in (0,\lambda_{X,1})$, 
$\delta > 0$, and $z,w\in U_{\varepsilon}(p)$, we have 
\begin{align*}
&\gxhyp(z,w)-\sum_{\gamma\in S_{\Gamma_{X}}(\delta;z,w)\backslash\lbrace\id\rbrace}g_{\mathbb{H}}(z,\gamma w)-
\sum_{\gamma\in\Gamma_{X,p}}\gh(z,\gamma w)=\\-&\frac{4\pi}{\vx(X)}\log\bigg(\frac{\log|\vartheta_{p}
(z)|}{\log\varepsilon}\bigg) -\frac{4\pi}{\vx(X)}\log\bigg(\frac{\log|\vartheta_{p}(w)|}{\log
\varepsilon}\bigg) +h_{\delta,p,p}(z,w),
\end{align*}  
where $h_{\delta,p,p}(z,w)$ is a harmonic function in both the variables $z\in U_{\varepsilon}(p)$ and $w\in U_{\varepsilon}(q)$, 
which satisfies the following upper bound
\begin{align}\label{cor3.12eqn}
\sup_{z,w\in U_{\varepsilon}(p)}\bigg|h_{\delta,p,p}(z,w)\bigg|\leq \boundb.
\end{align}
\begin{proof}
For $z,w\in U_{\varepsilon}(p)$, the hyperbolic Green's function satisfies the 
differential equation \eqref{diffeqnghyp}. For $z,w\in U_{\varepsilon}(p)$, put
\begin{align*}
h(z,w)=-\frac{4\pi}{\vx(X)}\log\bigg(\frac{\log|\vartheta_{p}(z)|}{\log\varepsilon}\bigg) -
\frac{4\pi}{\vx(X)}\log\bigg(\frac{\log|\vartheta_{p}(w)|}{\log\varepsilon}\bigg)+\\[0.4em]
\sum_{\gamma\in S_{\Gamma_{X}}(\delta;z,w)\backslash\lbrace\id\rbrace}g_{\mathbb{H}}(z,\gamma w)+
\sum_{\gamma\in\Gamma_{X,p}}\gh(z,\gamma w).
\end{align*}
Observe that for $z\not = w$, $d_{z}d_{z}^{c}h(z,w)=\shyp(z)$. So, if we show that both the 
functions $h(z,w)$ and $\gxhyp(z,w)$ admit the same type of singularity when $z=w$ on 
$U_{\varepsilon}(p)$, we can conclude that 
\begin{align*}
\gxhyp(z,w)=h(z,w)+h_{\delta,p,p}(z,w),
\end{align*}
where $h_{\delta,p,p}(z,w)$ is a harmonic function in both the variables $z,w\in U_{\varepsilon}(p)$. 
Moreover, from similar arguments as in Corollary \ref{cor3.10}, 
we can conclude that the function $h_{\delta,p,p}(z,w) $ satisfies the asserted upper
bound \eqref{cor3.12eqn}.

For any $z\in U_{\varepsilon}(p)$, from equations \eqref{phi} and \eqref{gcanbounded}, we find that
\begin{align*}
\lim_{w\rightarrow z}\big(\gxhyp(z,w)+\log|\vartheta_{z}(w)|^{2}\big)\,&\,=
\lim_{w\rightarrow z}\big(\gxcan(z,w)+\log|\vartheta_{z}(w)|^{2}\big)+2\phi_{X}(z)\\ &\,
= -\frac{8\pi}{\vx(X)}\log\bigg(\frac{\log|\vartheta_{p}(z)|}{\log\varepsilon}\bigg) +O_{z}(1),
\end{align*}
where the contribution from the term $O_{z}(1)$ is a smooth function which remains bounded 
for all $z\in U_{\varepsilon}(p)$ and for $z=p$. 

Now observe that
\begin{align}
&\lim_{w\rightarrow z}\big(h(z,w)+\log|\vartheta_{z}(w)|^{2}\big)=-\frac{8\pi}{\vx(X)}\log\bigg(
\frac{\log|\vartheta_{p}(z)|}{\log\varepsilon}\bigg)+\notag\\&\lim_{w\rightarrow z}\bigg(
\sum_{\gamma\in\Gamma_{X,p}\backslash \lbrace\id\rbrace}\gh(z,\gamma w)+\gh(z,w)+
\log|\vartheta_{z}(w)|^{2}\bigg)+O_{z}(1),\label{cor3.12eqn1}
\end{align}
where the contribution from the term $O_{z}(1)$ is a smooth function which remains bounded 
for all $z\in U_{\varepsilon}(p)$ and for $z=p$. For $z\in U_{\varepsilon}(p)$, from equation 
\eqref{lem3usefulbound} from proof of Lemma \ref{lem3}, and from the definition of $\gh(z,w)$, i.e., 
equation \eqref{defngh}, the second term on the right-side of equation \eqref{cor3.12eqn1} 
simplifies to give
\begin{align*}
&\lim_{w\rightarrow z}\bigg(
\sum_{\gamma\in\Gamma_{X,p}\backslash \lbrace\id\rbrace}\gh(z,\gamma w)+\gh(z,w)+
\log|\vartheta_{p}(w)-\vartheta_{p}(z)|^{2}\bigg)=\\&
P_{\mathrm{gen},p}(z)-4\pi\Im(\sigma_{p}^{-1}z)+\lim_{w\rightarrow z}\big(\gh(\sigma_{p}^{1}z,\sigma_{p}^{-1}w)+
\log\big|1-e^{2\pi i(w-z)}\big|^{2}\big)=\\
&P_{\mathrm{gen},p}(z)-4\pi\Im(\sigma_{p}^{-1}z)+\log\big(4\Im(\sigma_{p}^{-1}z)^{2}\big)+\log(4\pi^{2})
=O_{z}(1),
\end{align*}
which together with equation \eqref{cor3.12eqn1} completes the proof of the corollary.
\end{proof} 
\end{cor}
\begin{cor}\label{cor3.13}
Let $\mathfrak{e},\mathfrak{f}\in\mathcal{E}_{X}$ and $\mathfrak{e}\not=\mathfrak{f}$ be two 
elliptic fixed points. Then, for any $\alpha\in (0,\lambda_{X,1})$, 
$\delta > 0$, $z\in U_{\varepsilon}(\mathfrak{e})$, and $w\in U_{\varepsilon}(\mathfrak{f})$, 
we have 
\begin{align*}
&\gxhyp(z,w)-\sum_{\gamma\in S_{\Gamma_{X}}(\delta;z,w)}\gh(z,\gamma w)=\\&
-\frac{4\pi\log\big(1-|\vartheta_{\mathfrak{e}}(z)|^{2
\slash m_{\mathfrak{e}}}\big)}{\vx(X)}-\frac{4\pi\log\big(1-|\vartheta_{\mathfrak{f}}(w)|^{2
\slash m_{\mathfrak{f}}}\big)}{\vx(X)}+h_{\delta,\mathfrak{e},\mathfrak{f}}(z,w),
\end{align*}
where $h_{\delta,\mathfrak{e},\mathfrak{f}}(z,w)$ is a harmonic function in both the variables $z\in 
U_{\varepsilon}(\mathfrak{e})$ and $w\in U_{\varepsilon}(\mathfrak{e})$, which 
satisfies the following upper bound
\begin{align*}
\sup_{\substack{z\in U_{\varepsilon}(\mathfrak{e})\\w\in U_{\varepsilon}(\mathfrak{f})}}\bigg| 
h_{\delta,\mathfrak{e},\mathfrak{f}}(z,w)\bigg|\leq \boundb;
\end{align*}
furthermore, for $z,w\in U_{\varepsilon}(\mathfrak{e})$, we have 
\begin{align*}
&\gxhyp(z,w)-\sum_{\gamma\in S_{\Gamma_{X}}(\delta;z,w)\backslash\lbrace\id\rbrace}
\gh(z,\gamma w)-\sum_{\gamma\in\Gamma_{X,\mathfrak{e}}}\gh(z,\gamma w)=\\&
-\frac{4\pi\log\big(1-|\vartheta_{\mathfrak{e}}(z)|^{2
\slash m_{\mathfrak{e}}}\big)}{\vx(X)}-\frac{4\pi\log\big(1-|\vartheta_{\mathfrak{e}}(w)|^{2
\slash m_{\mathfrak{e}}}\big)}{\vx(X)}+h_{\delta,\mathfrak{e},\mathfrak{e}}(z,w),
\end{align*}
where $h_{\delta,\mathfrak{e},\mathfrak{e}}(z,w)$ is a harmonic function in both the variables 
$z,w\in U_{\varepsilon}(\mathfrak{e})$, which satisfies the following upper bound
\begin{align*}
\sup_{z\in U_{\varepsilon}(\mathfrak{e})}
\bigg| h_{\delta,\mathfrak{e},\mathfrak{e}}(z,w)\bigg|\leq \boundb;
\end{align*}
\begin{proof}
The proof of the corollary follows from arguments similar to the ones employed in the proofs 
of Corollaries \ref{cor3.11} and \ref{cor3.12}.    
\end{proof}
\end{cor}
\section{Bounds for canonical Green's function}\label{section3.2}
In this section, we obtain bounds for the canonical Green's function on the compact subset 
$Y_{\varepsilon}$ of $X$. From equation \eqref{phi}, to derive bounds for the canonical Green's 
function $\gxcan(z,w)$, it suffices to derive bounds for the function $\phi_{X}(z)$, and for the hyperbolic Green's 
function $\gxhyp(z,w)$. From last section, we have bounds for $\gxhyp(z,w)$, and it remains 
to bound the function $\phi_{X}(z)$. Recall that from Corollary \ref{cor6}, we have
\begin{align}
&\phi_{X}(z)=\frac{\big(H_{X}(z)+E_{X}(z)\big)}{2g_{X}}+\frac{1}{8\pi g_{X}}\int_{X}\gxhyp(z,\zeta)
\del P_{X}(\zeta)\hyp(z)-\notag\\&\sum_{\mathfrak{e}\in\mathcal{E}_{X}}
\frac{m_{\mathfrak{e}}-1}{2g_{X}m_{\mathfrak{e}}}\gxhyp(z,\mathfrak{e})-\frac{C_{X,\mathrm{hyp}}}
{8g_{X}^{2}}-\frac{2\pi(c_{X}-1)}{g_{X}\vx(X)}-\frac{1}{2g_{X}}\int_{X}E_{X}(\zeta)\shyp(\zeta).
\label{phi(z)formula1}
\end{align}
Using analysis from the sections 2 and 3, it is easy to bound almost all the quantities involved 
in the above expression for $\phi_{X}(z)$ excepting the integral 
\begin{align*}
\frac{1}{8\pi g_{X}}\int_{X}\gxhyp(z,\zeta)
\del P_{X}(\zeta)\hyp(z),
\end{align*}
which we now accomplish.
\begin{lem}\label{lem3.16}
For $z\in Y_{\varepsilon}$, we have the equality of integrals
\begin{align*}
&\int_{X}g_{X,\mathrm{hyp}}(z,\zeta)\del P_{X}(\zeta)\hyp(\zeta)= 4\pi P_{X}(z)-
4\pi\int_{Y_{\varepsilon\slash 2}^{\mathrm{par}}}P_{X}(\zeta)
\shyp(\zeta) +\\&4\pi\sum_{p\in \mathcal{P}_{X}}\bigg(\int_{\partial U_{\varepsilon\slash 2}(p)}
g_{X,\mathrm{hyp}}(z,\zeta)d_{\zeta}^{c}P_{X}(\zeta)-\int_{\partial U_{\varepsilon\slash 2}(p)}
P_{X}(\zeta)d_{\zeta}^{c}g_{\mathrm{hyp}}(z,\zeta)\bigg)+\\&
\sum_{p\in \mathcal{P}_{X}}\int_{U_{\varepsilon\slash 2}(p)}g_{X,\mathrm{hyp}}(z,\zeta)\del P_{X}
(\zeta)\hyp(\zeta).
\end{align*}
\begin{proof}
Observe that we have the following decomposition 
\begin{align}
&\int_{X}g_{X,\mathrm{hyp}}(z,\zeta)\del P_{X}(\zeta)\hyp(\zeta)=-4\pi\int_{X}
g_{X,\mathrm{hyp}}(z,\zeta)d_{\zeta}d_{\zeta}^{c} P_{X}(\zeta)=\notag\\
-4\pi&\int_{Y_{\varepsilon\slash 2}^{\mathrm{par}}}g_{X,\mathrm{hyp}}(z,\zeta)d_{\zeta}d_{
\zeta}^{c}P_{X}(\zeta)+\sum_{p\in\mathcal{P}_{X}}\int_{U_{\varepsilon\slash 2}(p)}
g_{X,\mathrm{hyp}}(z,\zeta)\del P_{X}(\zeta)\hyp(\zeta).\label{lem3.16eqn1}
\end{align}
Let $U_{r}(z)$ denote an open coordinate disk of radius $r$ around $z\in Y_{\varepsilon}$ 
with $r$ small enough such that $ U_{r}(z)\subsetneq Y_{\varepsilon\slash 2}^{\mathrm{par}}$. From 
equation \eqref{diffeqnghyp} and from Stokes's theorem, we have
\begin{align}
&-\int_{Y_{\varepsilon\slash 2}^{\mathrm{par}}}g_{X,\mathrm{hyp}}(z,\zeta)
d_{\zeta}d_{\zeta}^{c}P_{X}(\zeta)+\int_{Y_{\varepsilon\slash 2}^{\mathrm{par}}}P_{X}(\zeta)
\shyp(\zeta)=\notag\\&\lim_{r\rightarrow 0}\bigg(-\int_{Y_{\varepsilon\slash 2}^{\mathrm{par}}\backslash U_{r}(z)}g_{X,\mathrm{hyp}}(z,\zeta)
d_{\zeta}d_{\zeta}^{c}P_{X}(\zeta)+\int_{Y_{\varepsilon\slash 2\backslash U_{r}(z)}^{
\mathrm{par}}}P_{X}(\zeta)d_{\zeta}d_{\zeta}^{c}g_{\mathrm{hyp}}(z,\zeta)\bigg)=\notag\\&
\lim_{r\rightarrow 0}\Bigg(\int_{\partial U_{r}(z)}g_{X,\mathrm{hyp}}(z,\zeta)d_{\zeta}^{c}P_{X}(\zeta)-\int_{\partial 
U_{r}(z)}P_{X}(\zeta)d_{\zeta}^{c}g_{\mathrm{hyp}}(z,\zeta)\Bigg)+\notag\\&
\sum_{p\in\mathcal{ P}_{X}}\bigg(\int_{\partial U_{\varepsilon\slash 2}(p)}g_{X,\mathrm{hyp}}
(z,\zeta)d_{\zeta}^{c}P_{X}(\zeta)-\int_{\partial U_{\varepsilon\slash 2}(p)}P_{X}(\zeta)
d_{\zeta}^{c}g_{\mathrm{hyp}}(z,\zeta)\bigg).\label{lem3.16eqn2}
\end{align}
Using the fact that the function $P_{X}(\zeta)$ is smooth at $z$, and as $\zeta$ approaches $z$, 
the hyperbolic Green's function $\gxhyp(z,\zeta)$ satisfies
\begin{align*}
\gxhyp(z,\zeta)=-\log|\vartheta_{z}(\zeta)|^{2}+O_{z}(1),
\end{align*}
we derive that 
\begin{align*}
\lim_{r\rightarrow 0}\bigg(\int_{\partial U_{r}(z)}g_{X,\mathrm{hyp}}(z,\zeta)d_{\zeta}^{c}P_{X}(\zeta)-
\int_{\partial U_{r}(z)}P_{X}(\zeta)d_{\zeta}^{c}g_{\mathrm{hyp}}(z,\zeta)\bigg)= P_{X}(z).
\end{align*}
Combining the above equation with equations \eqref{lem3.16eqn1} and \eqref{lem3.16eqn2} 
completes the proof of the lemma. 
\end{proof} 
\end{lem}
\begin{cor}\label{cor3.17}
For any $z\in Y_{\varepsilon}^{\mathrm{par}}$, we have 
\begin{align}
&\phi_{X}(z)= \frac{\big(P_{X}(z)+E_{X}(z)+H_{X}(z)\big)}{2g_{X}}+\frac{1}{8\pi g_{X}}
\sum_{p\in \mathcal{P}_{X}}\int_{U_{\varepsilon\slash 2}(p)}g_{X,\mathrm{hyp}}(z,\zeta)\del P_{X}
(\zeta)\hyp(\zeta)+ \notag\\&\frac{1}{2g_{X}}
\sum_{p\in \mathcal{P}_{X}}\bigg(\int_{\partial U_{\varepsilon\slash 2}(p)}
g_{X,\mathrm{hyp}}(z,\zeta)d_{\zeta}^{c}P_{X}(\zeta)-\int_{\partial U_{\varepsilon\slash 2}(p)}
P_{X}(\zeta)d_{\zeta}^{c}g_{\mathrm{hyp}}(z,\zeta)\bigg)-\frac{2\pi(c_{X}-1)}{g_{X}\vx(X)}-\notag\\&
\frac{1}{2g_{X}}\int_{Y_{\varepsilon\slash 2}^{\mathrm{par}}}P_{X}(\zeta)\shyp(\zeta)-
\frac{C_{X,\mathrm{hyp}}}{8g_{X}^{2}}+\sum_{\mathfrak{e}\in\mathcal{E}_{X}}
\frac{m_{\mathfrak{e}}-1}{2g_{X}m_{\mathfrak{e}}}\gxhyp(z,\mathfrak{e})
-\frac{1}{2g_{X}}\int_{X}E_{X}(\zeta)\shyp(\zeta).
\label{phiunfolding}
\end{align}
\begin{proof}
The proof of the corollary follows directly from combining equation \eqref{phi(z)formula1} 
and Lemma \ref{lem3.16}.
\end{proof}
\end{cor}
\begin{lem}\label{lem3.18}
For any $\alpha\in(0,\lambda_{X,1})$ and $\delta\in (0,\ell_{X})$, we have the following upper bound
\begin{align*}
\sup_{z\in Y_{\varepsilon}}\frac{\big|P_{X}(z)+E_{X}(z)+H_{X}(z)\big|}{2g_{X}} \leq
\frac{\boundbtwo}{2g_{X}} .
\end{align*}
\begin{proof}
For any $\alpha\in(0,\lambda_{X,1})$ and $\delta\in (0,\ell_{X})$, from equation \eqref{prop4eqn}, 
we have
\begin{align*}
&\sup_{z\in Y_{\varepsilon}}\big|P_{X}(z)+E_{X}(z)+H_{X}(z)\big|=\sup_{z\in Y_{
\varepsilon}}\lim_{w\rightarrow z}\bigg|\gxhyp(z,w)-\sum_{\gamma\in S_{\Gamma_{X}}(
\delta;z,w)}\gh(z,\gamma w)\bigg|\leq \\&\sup_{z\in Y_{\varepsilon\slash 2}}
\lim_{w\rightarrow z}\bigg|\gxhyp(z,w)-\sum_{\gamma\in S_{\Gamma_{X}}(\delta;z,w)}\gh(z,\gamma w)\bigg|,
\end{align*}
and the proof of the lemma follows from Proposition \ref{prop3.2}.
\end{proof}
\end{lem}
\begin{prop}\label{prop3.19}
For any $\alpha\in(0,\lambda_{X,1})$ and $\delta\in (0,
\widetilde{\varepsilon})$, we have the following upper bound
\begin{align*}
&\frac{1}{8\pi g_{X}}\sup_{z\in Y_{\varepsilon}}\sum_{p\in \mathcal{P}_{X}}\bigg{|}
\int_{U_{\varepsilon\slash 2}(p)}\gxhyp(z,\zeta)\del P_{X}(\zeta)\hyp(\zeta)\bigg{|} \leq\\
-&\frac{|\mathcal{P}_{X}|\,C_{X,\mathrm{par}}^{\mathrm{aux}}}{4g_{X}\,\log(\varepsilon\slash 2)}\,
\bigg(\,\boundbtwo +\frac{4\pi}{\vx(X)}\bigg).
\end{align*}
\begin{proof}
Observe the inequality
\begin{align}
&\sup_{z\in Y_{\varepsilon}}\sum_{p\in \mathcal{P}_{X}}\bigg{|}\int_{U_{\varepsilon\slash 2}
(p)}g_{X,\mathrm{hyp}}(z,\zeta)\del P_{X}(\zeta)\hyp(\zeta)\bigg{|} \leq\sup_{\zeta\in X}\big{|}\del 
P_{X}(\zeta)\big{|}\times\notag\\&\sup_{z\in Y_{\varepsilon}}\sum_{p\in \mathcal{P}_{X}}
\bigg{|}\int_{U_{\varepsilon\slash 2}(p)}g_{X,\mathrm{hyp}}(z,\zeta)\hyp(\zeta)\bigg{|}=
C_{X,\mathrm{par}}^{\mathrm{aux}}\,\bigg(\sup_{z\in Y_{\varepsilon}}\sum_{p\in \mathcal{P}_{X}}
\bigg{|}\int_{U_{\varepsilon\slash 2}(p)}g_{X,\mathrm{hyp}}(z,\zeta)\hyp(\zeta)\bigg{|}\bigg) .
\label{prop3.19eqn0}
\end{align} 
For any $p\in\mathcal{P}_{X}$, $z\in Y_{\varepsilon}$, and $\zeta\in U_{\varepsilon
\slash 2}(p)$, from arguments as in Corollary \ref{cor3.10}, we have
\begin{align}\label{prop3.19eqn1}
\gxhyp(z,\zeta)=-\frac{4\pi}{\vx(X)}\log\bigg(\frac{\log|\vartheta_{p}(\zeta)|}{\log
(\varepsilon\slash 2)}\bigg)+g_{p}(z,\zeta),
\end{align}
where $g_{p}(z,\zeta)$ is a harmonic function in the variable $\zeta$. From maximum principle for 
harmonic functions and from Corollary \ref{cor3.8}, we have the following upper bound
\begin{align}\label{prop3.19eqn2}
\sup_{\substack{z\in Y_{\varepsilon}\\\zeta\in U_{\varepsilon\slash 2}(p)}}\big|g_{p}(z,
\zeta)\big|=\sup_{\substack{z\in Y_{\varepsilon}\\\zeta\in \partial U_{\varepsilon
\slash 2}(p)}}\big|g_{p}(z,\zeta)\big|=\sup_{\substack{z\in Y_{\varepsilon}\\\zeta\in \partial U_{\varepsilon
\slash 2}(p)}}\big|\gxhyp(z,\zeta)\big|\leq\notag\\\sup_{\substack{z\in Y_{\varepsilon}\\\zeta\in 
\partial Y_{\varepsilon\slash 2}^{\mathrm{par}}}}\big|\gxhyp(z,\zeta)\big|\leq\boundbtwo,
\end{align}
for any $\alpha\in(0,\lambda_{X,1})$ and $\delta\in (0,\widetilde{\varepsilon})$. 

\vspace{0.2cm}
For any $p\in\mathcal{P}_{X}$, we make the following computations
\begin{align*}
&\int_{U_{\varepsilon\slash 2}(p)}\hyp(\zeta)=\int_{0}^{\varepsilon\slash 2}\int_{0}^{2\pi}
\frac{rdrd\theta}{(r\log r)^{2}}=2\pi\int_{0}^{\varepsilon\slash 2}\frac{d(\log r)}{(\log r)^{2}}=
-\frac{2\pi}{\log(\varepsilon\slash 2)},\\
&\int_{U_{\varepsilon\slash 2}(p)}\log\big(-\log|\vartheta_{p}(\zeta)|\big)\hyp(\zeta)=
\int_{0}^{\varepsilon\slash 2}\int_{0}^{2\pi}\frac{r\log\big{(}-\log r\big{)}drd\theta}{(r\log r)^{2}}
=\\&2\pi\int_{0}^{\varepsilon\slash 2}\frac{\log\big{(}-\log r\big{)}d(\log r)}{(\log r)^{2}}
=-\frac{2\pi\big(\log\big(-\log(\varepsilon\slash 2)\big)+1\big)}{\log(\varepsilon\slash 2)}.
\end{align*}
For any $p\in\mathcal{P}_{X}$, using inequality \eqref{prop3.19eqn2}, and the above computations, we 
derive
\begin{align}
&\bigg{|}\int_{U_{\varepsilon\slash 2}(p)}g_{p}(z,\zeta)\hyp(\zeta)\bigg|\leq -\frac{2\pi\boundbtwo}{
\log(\varepsilon\slash 2)},\label{prop3.19eqn3}\\&
\bigg{|}\int_{U_{\varepsilon\slash 2}(p)}\frac{4\pi}{\vx(X)}\log\bigg(
\frac{\log|\vartheta_{p}(\zeta)|}{\log(\varepsilon\slash 2)} \bigg)\bigg{|}\hyp(\zeta)=\notag\\
&\int_{U_{\varepsilon\slash 2}(p)}\frac{4\pi}{\vx(X)}\log
\bigg(\frac{-\log|\vartheta_{p}(\zeta)|}{-\log(\varepsilon\slash 2)}\bigg)\hyp(\zeta)=-
\frac{8\pi^{2}}{\vx(X)\log(\varepsilon\slash 2)}.\label{prop3.19eqn4} 
\end{align}
For any $p\in\mathcal{P}_{X}$, using equation \eqref{prop3.19eqn1}, and the above computations 
\eqref{prop3.19eqn3} and \eqref{prop3.19eqn4}, we arrive at
\begin{align}
\bigg{|}\int_{U_{\varepsilon\slash 2}(p)}g_{X,\mathrm{hyp}}(z,\zeta)\hyp(\zeta)\bigg{|} \leq 
-\frac{2\pi}{\log(\varepsilon\slash 2)}\bigg(\boundbtwo+\frac{4\pi}{\vx(X)}\bigg)
\end{align}
Combining the above upper bound with inequality \eqref{prop3.19eqn0} completes the proof of the 
corollary. 
\end{proof} 
\end{prop}
\begin{rem}\label{firstline}
For any $z\in Y_{\varepsilon}$, combining Lemma \ref{lem3.18} and Proposition \ref{prop3.19}, 
we obtain the following upper bound for the first line on the right-hand side of equation \eqref{phiunfolding}
\begin{align*}
\frac{\boundbtwo}{2g_{X}}-\frac{|\mathcal{P}_{X}|\,C_{X,\mathrm{par}}^{\mathrm{aux}}}{4g_{X}\,
\log(\varepsilon\slash 2)}\,\bigg(\,\boundbtwo +\frac{4\pi}{\vx(X)}\bigg),
\end{align*}
for any $\alpha\in(0,\lambda_{X,1})$ and $\delta\in \big(0,\min\lbrace{\ell_{X},\widetilde{\varepsilon}
\rbrace}\big)$.
\end{rem}
\begin{prop}\label{prop3.20}
For any $\alpha\in(0,\lambda_{X,1})$ and $\delta\in (0,\widetilde{\varepsilon})$, we have the following upper bound
\begin{align*}
\frac{1}{2g_{X}}\sup_{z\in Y_{\varepsilon}}\sum_{p\in\mathcal{P}_{X}}\bigg{|}\int_{\partial U_{
\varepsilon\slash 2}(p)}g_{X,\mathrm{hyp}}(z,\zeta)d_{\zeta}^{c}P_{X}(\zeta)\bigg{|} \leq 
\frac{|\mathcal{P}_{X}|\, \boundbtwo}{2g_{X}}.
\end{align*}
\begin{proof}
From Corollary \ref{cor3.8} and Stokes's theorem, we have the elementary estimate 
\begin{align}
&\sup_{z\in Y_{\varepsilon}}\sum_{p\in \mathcal{P}_{X}}\bigg{|}\int_{\partial 
U_{\varepsilon\slash 2}(p)}\gxhyp(z,\zeta)d_{\zeta}^{c}P_{X}(\zeta)\bigg{|}\leq
\sup_{\substack{z\in Y_{\varepsilon}\\\zeta \in\partial Y_{\varepsilon\slash 2}^{\mathrm{par}}
}}\big{|}\gxhyp(z,\zeta)\big{|}\cdot\bigg(\sum_{p\in \mathcal{P}_{X}}\bigg{|}\int_{\partial 
U_{\varepsilon\slash 2}(p)}d_{\zeta}^{c}P_{X}(\zeta)\bigg{|}\bigg)\notag\\&\leq
\boundbtwo\cdot \bigg(\sum_{p\in \mathcal{P}_{X}}\int_{\partial U_{\varepsilon\slash 2}(p)}
\big{|}d_{\zeta}d_{\zeta}^{c}P_{X}(\zeta)\big{|}\bigg)\leq\frac{\boundbtwo}{4\pi}
\cdot\bigg(\int_{X}\big|\del P_{X}(\zeta)\big{|}\hyp(\zeta)\bigg)
\label{prop3.20eqn1}
\end{align}
for any $\alpha\in(0,\lambda_{X,1})$ and $\delta\in (0,\widetilde{\varepsilon})$. 

\vspace{0.2cm}
Let $U_{r}(p)$ denote an open coordinate disk of radius $r$ around a 
parabolic fixed point $p\in\mathcal{P}_{X}$. Put 
\begin{align*}
Y_{r}^{\mathrm{par}}=X\backslash \bigcup_{p\in \mathcal{ P}_{X}}U_{r}(p).
\end{align*}
For every $z\in X$, from formula \eqref{delP}, we know that $\big|\del P_{X}(\zeta)\big{|}=- \del 
P_{X}(\zeta)$. Then, using Stokes's theorem, we find
\begin{align}
&\int_{X}\big|\del P_{X}(\zeta)\big|\hyp(\zeta)=4\pi\lim_{r\rightarrow 0}\int_{Y_{r}^{\mathrm{par}}}
d_{\zeta}d_{\zeta}^{c}P_{X}(\zeta)=\notag\\&4\pi\sum_{p\in \mathcal{P}_{X}}\lim_{r\rightarrow 0}
\int_{\partial U_{r}(p)}d_{\zeta}^{c}P_{X}(\zeta)=-4\pi|\mathcal{P}_{X}|
\lim_{r\rightarrow 0}\int_{0}^{2\pi}\frac{r}{2}\frac{\partial P_{X}(\zeta)}{\partial r}
\frac{d\theta}{2\pi},\label{prop3.20eqn2}
\end{align}
for any $p\in\mathcal{P}_{X}$. Now from Lemma \ref{lem3}, for any $z\in \partial U_{r}(p)$, we have
\begin{align}
P_{X}(\zeta) = 4\pi\Im(\sigma_{p}^{-1}\zeta)-\log\big(4\Im(\sigma_{p}^{-1}\zeta)^2\big)+ O_{\zeta}(1)=
-2\log r-2\log\big(-\log r\big)+O(1)\notag\\\Longrightarrow \frac{r}{2}\frac{\partial 
P_{X}(\zeta)}{\partial r}=-1-\frac{2}{r\log r}+O(r)\Longrightarrow
 -4\pi\,|\mathcal{P}_{X}|\lim_{r\rightarrow 0}\int_{0}^{2\pi}\frac{r}{2}\frac{\partial P_{X}(\zeta)}{\partial r}
\frac{d\theta}{2\pi}=4\pi|\mathcal{P}_{X}|.\label{prop3.20eqn3}
\end{align} 
Combining computations \eqref{prop3.20eqn2} and \eqref{prop3.20eqn3} with upper bound \eqref{prop3.20eqn1}, 
completes the proof of the proposition.
\end{proof}
\end{prop}
\begin{prop}\label{prop3.23}
We have the following upper bound
\begin{align*}
\frac{1}{2g_{X}}\sup_{z\in Y_{\varepsilon}}\sum_{p\in \mathcal{P}_{X}}\bigg{|}\int_{\partial U_{\varepsilon\slash 2}
(p)}P_{X}(\zeta)d_{\zeta}^{c}g_{X,\mathrm{hyp}}(z,\zeta)\bigg{|}\leq -\frac{3\,|\mathcal{P}_{X}|
\log(\varepsilon\slash 2)}{g_{X}}+ \frac{16\,C_{X,\mathrm{par}}}{g_{X}}.
\end{align*}
\begin{proof}
Since $P(\zeta)$ is a non-negative function on $X$, using Stokes's theorem, we derive 
\begin{align*}
\sup_{z\in Y_{\varepsilon}}\sum_{p\in \mathcal{P}_{X}}\bigg{|}
\int_{\partial U_{\varepsilon\slash 2}(p)}P_{X}(\zeta)d_{\zeta}^{c}g_{X,\mathrm{hyp}}(z,\zeta)\bigg{|}
\leq\\\sup_{\zeta\in Y_{\varepsilon\slash 2}^{\mathrm{par}}}P_{X}(\zeta)\cdot
\bigg(\sup_{z\in Y_{\varepsilon}}\sum_{p\in \mathcal{P}_{X}}\bigg{|}
\int_{\partial U_{\varepsilon\slash 2}(p)}d_{\zeta}d_{\zeta}^{c}g_{X,\mathrm{hyp}}(z,\zeta)
\bigg{|}\bigg)=\\\sup_{\zeta\in Y_{\varepsilon\slash 2}^{\mathrm{par}}}P_{X}(\zeta)\cdot
\bigg(\sup_{z\in Y_{\varepsilon}}\sum_{p\in \mathcal{P}_{X}}\bigg{|}
\int_{\partial U_{\varepsilon\slash 2}(p)}\shyp(\zeta)\bigg{|}\bigg)\leq \sup_{z\in 
Y_{\varepsilon\slash 2}^{\mathrm{par}}}P_{X}(\zeta),
\end{align*} 
and the proof of the proposition follows directly from estimate \eqref{estimateP}. 
\end{proof}
\end{prop}
\begin{rem}\label{secondline}
For any $z\in Y_{\varepsilon}$, combining Propositions \ref{prop3.20} and \ref{prop3.23}, 
we obtain the following upper bound for the second line on the right-hand side of equation 
\eqref{phiunfolding}
\begin{align*}
\frac{|\mathcal{P}_{X}|\, \boundbtwo}{2g_{X}}-\frac{3\,|\mathcal{P}_{X}|\log(\varepsilon\slash 2)}{
g_{X}}+ \frac{16\,C_{X,\mathrm{par}}}{g_{X}} +\frac{2\pi\,|c_{X}-1|}{g_{X}\vx(X)},
\end{align*}
for any $\alpha\in(0,\lambda_{X,1})$ and $\delta\in (0,\widetilde{\varepsilon})$.
\end{rem}
\begin{prop}\label{prop3.25}
We have the following upper bound
\begin{align*}
\frac{1}{2g_{X}}\bigg{|}\int_{Y_{\varepsilon\slash 2}^{\mathrm{par}}}P_{X}(z)\shyp(z)\bigg{|}\leq 
-\frac{|\mathcal{P}_{X}|\log(\varepsilon\slash 2)}{g_{X}}. 
\end{align*}
\begin{proof}
Since $P_{X}(z)$ is a non-negative function on $X$, we have
\begin{align}\label{prop3.25eqn1}
\bigg{|}\int_{Y_{\varepsilon\slash 2}^{\mathrm{par}}}P_{X}(z)\shyp(z) \bigg{|}\leq\int_{Y_{
\varepsilon\slash 2,p}^{\mathrm{par}}}P_{X}(z)\shyp(z)=\sum_{p\in\mathcal{P}_{X}}
\sum_{\eta\in\Gamma_{X,p}\backslash\Gamma_{X}}\int_{Y_{\varepsilon\slash 2,p}^{\mathrm{par}}}
P_{\mathrm{gen},p}(\eta z)\shyp(z) .
\end{align}
The interchange of summation and integration in the above equation is valid, provided that the 
latter series converges absolutely. As the function $P_{X}(z)$ is a non-negative function, to prove the 
absolute convergence of the latter series, it suffices to prove that
\begin{align}\label{prop3.25toprove}
\sum_{p\in\mathcal{P}_{X}}\sum_{\eta\in\Gamma_{X,p}\backslash\Gamma_{X}}\int_{Y_{\varepsilon\slash 2,p}^{\mathrm{par}}}
P_{\mathrm{gen},p}(\eta z)\shyp(z)\leq  -2\,|\mathcal{P}_{X}|\log(\varepsilon\slash 2).
\end{align}
For every $p\in\mathcal{P}_{X}$, after making the substitution $z\mapsto \eta^{-1}\sigma_{p}z$, from the 
$\mathrm{PSL}_{2}(\mathbb{R})$-invariance of the metric $\shyp(z)$, from estimate \eqref{lem2eqn2} 
from proof of Lemma \ref{lem2}, and using the fact that $2\pi\leq \vx(X)$, we get
\begin{align*}
&\sum_{p\in \mathcal{P}_{X}}\sum_{\eta\in\Gamma_{X,p}\backslash\Gamma_{X}}\int_{Y_{\varepsilon\slash 2,
p}^{\mathrm{par}}}P_{\mathrm{gen},p}(\eta z)\shyp(z)=\sum_{p\in \mathcal{P}_{X}}\sum_{\eta\in
\Gamma_{X,p}\backslash\Gamma_{X}}\int_{\sigma_{p}^{-1}\eta Y_{\varepsilon\slash 2,p}^{\mathrm{par}}}
P_{\mathrm{gen},p}(\sigma_{p} z)\shyp(z)=\\&\frac{1}{\vx(X)}
\sum_{p\in \mathcal{P}_{X}}\int_{0}^{-\log(\varepsilon\slash 2)\slash 2\pi}\int_{0}^{1}
P_{\mathrm{gen},p}(\sigma_{p} z)\frac{dxdy}{y^{2}}\leq \\&\frac{1}{\vx(X)}\sum_{p\in \mathcal{P}_{X}}
\int_{0}^{-\log(\varepsilon\slash 2)\slash2\pi}\int_{0}^{1}32y^{2}\frac{dxdy}{y^{2}}=
-\frac{16\,|\mathcal{P}_{X}|\log(\varepsilon\slash 2)}{\pi\vx(X)}\leq -
2\,|\mathcal{P}_{X}|\log(\varepsilon\slash 2),
\end{align*}
which proves upper bound \eqref{prop3.25toprove}, and completes the proof of the proposition. 
\end{proof}
\end{prop}
\begin{prop}\label{prop3.26}
We have the following upper bound
\begin{align*}
\frac{\big|C_{X,\mathrm{hyp}}\big|}{8g_{X}^{2}} \leq \frac{2\pi \left(d_{X}+1\right)^{2}}{\lambda_{X,1}
\vx(X)}.
\end{align*}
\begin{proof}
Recall that $C_{X,\mathrm{hyp}}$ is defined as 
\begin{align*}
&C_{X,\mathrm{hyp}}= \\&\int_{X}\int_{X}\gxhyp(\zeta,\xi)\bigg(\int_{0}^{\infty}\del 
\kxhyp(t;\zeta)dt\bigg)\bigg(\int_{0}^{\infty}\del \kxhyp(t;\xi)dt\bigg)\hyp(\xi)\hyp(\zeta).
\end{align*}
From formulae \eqref{phi}, \eqref{phi(z)formula}, we have
\begin{align}
&\del\phi_{X}(z)=\frac{4\pi\can(z)}{\hyp(z)}-\frac{4\pi}{\vx(X)}\Longrightarrow \int_{X}\del\phi_{X}(z)
\hyp(z)=0,\label{prop3.26eqn0}\\&\phi_{X}(z)= \frac{1}{2g_{X}}\int_{X}g_{X,\mathrm{hyp}}(z,\zeta)\bigg(\int_{0}^{
\infty}\del K_{X,\mathrm{hyp}}(t;\zeta)dt\bigg)\hyp(\zeta)-\frac{C_{X,\mathrm{hyp}}}{8g_{X}^{2}},\notag
\end{align}
respectively. So combining the above two equations, we get
\begin{align}
-&\frac{1}{4\pi}\int_{X}\phi_{X}(z)\del\phi_{X}(z)\hyp(z)=\notag\\-&\frac{1}{2g_{X}}\int_{X}\int_{X}
g_{X,\mathrm{hyp}}(z,\zeta)\bigg(\int_{0}^{\infty}\del K_{X,\mathrm{hyp}}(t;\zeta)dt\bigg)\hyp(\zeta)
\can(z).\label{prop3.26eqn2}
\end{align}
Observe that
\begin{align*}
\int_{X}g_{X,\mathrm{hyp}}(z,\zeta)\bigg(\int_{0}^{\infty}\del K_{X,\mathrm{hyp}}(t;\zeta)dt\bigg)
\hyp(\zeta)=2g_{X}\phi_{X}(z)+\frac{C_{X,\mathrm{hyp}}}{4g_{X}}\in C_{\ell,\ell\ell}(X).
\end{align*}
So combining equations \eqref{keyidentity} and \eqref{prop3.26eqn2}, we derive
\begin{align}
&\int_{X}\phi_{X}(z)\del\phi_{X}(z)\hyp(z)=\frac{\pi}{ g_{X}^{2}}\int_{X}\int_{X}
g_{X,\mathrm{hyp}}(z,\zeta)\left(\int_{0}^{\infty}\del K_{X,\mathrm{hyp}}(t;\zeta)dt\right)\times\notag
\\&\left(\int_{0}^{\infty}\del K_{X,\mathrm{hyp}}(t;z)dt\right)
\hyp(\zeta)\hyp(z)=\frac{\pi C_{X,\mathrm{hyp}}}{ g_{X}^{2}}.\label{prop3.26eqn1}
\end{align}
Using equation \eqref{prop3.26eqn0}, we have
\begin{align}\label{prop3.26eqn3}
&\sup_{z\in X}|\del\phi_{X}(z)|\leq \sup_{z\in X}\bigg{|}\frac{4\pi\can(z)}{\vx(X)\shyp(z)}\bigg{|} +
\frac{4\pi}{\vx(X)}=\frac{4\pi\left(d_{X}+1\right)}{\vx(X)},
\end{align}
where $d_{X}$ is as defined in \eqref{defndx}. As the function $\phi_{X}(z)\in L^{2}(X)$, it admits a 
spectral expansion of the form \eqref{spectralf}. So from the arguments used to prove Proposition 4.1 
in \cite{jkannals}, we have 
\begin{align}\label{prop3.26eqn4}
&\bigg|\int_{X}\phi_{X}(z)\del \phi_{X}(z)\hyp(z) \bigg|
\leq\sup_{z\in X} \frac{|\del\phi_{X}(z)|^{2}}{\lambda_{X,1}}\int_{X}\hyp(z).
\end{align}
Hence, from equation \eqref{prop3.26eqn1}, and combining estimates \eqref{prop3.26eqn3} and 
\eqref{prop3.26eqn4}, we arrive at the estimate
\begin{align*}
\big|C_{X,\mathrm{hyp}}\big| =\frac{ g_{X}^{2}}{\pi}\bigg| \int_{X}\phi_{X}(z)\del \phi_{X}(z) \hyp(z)
\bigg|\leq \\\frac{g_{X}^{2}}{\pi\lambda_{X,1}}\int_{X}|\del \phi_{X}(z)|^{2}\hyp(z)\leq\frac{16\pi 
g_{X}^{2}\left(d_{X}+1\right)^{2}}{\lambda_{X,1}\vx(X)},
\end{align*}
which completes the proof of the proposition.
\end{proof}
\end{prop}
\begin{lem}\label{lem3.27}
We have the following upper bound
\begin{align*}
\frac{1}{2g_{X}}\int_{X}E_{X}(\zeta)\shyp(\zeta)\leq \frac{5 \,c_{X,\mathrm{ell}}}{
g_{X}\vx(X)}\sum_{\mathfrak{e}\in\mathcal{E}_{X}}(m_{\mathfrak{e}}-1).  
\end{align*}
\begin{proof}
For any $z\in X$ and equation \eqref{lem1eqn1}, we have
\begin{align*}
&\int_{X}E_{X}(\zeta)\shyp(\zeta)=\int_{X} \sum_{\mathfrak{e}\in \mathcal{E}_{X}}\sum_{\eta\in\Gamma_{X,
\mathfrak{e}}\backslash\Gamma_{X}}\sum_{n=1}^{m_{\mathfrak{e}}-1}g_{\mathbb{H}}(\sigma_{\mathfrak{e}
}^{-1}\eta z,\gamma_{i}^{n}\sigma_{\mathfrak{e}}^{-1}\eta z)\shyp(\zeta)=\notag\\
&\sum_{\mathfrak{e}\in \mathcal{E}_{X}}\sum_{\eta\in\Gamma_{X,\mathfrak{e}}\backslash\Gamma_{X}}
\sum_{n=1}^{m_{\mathfrak{e}}-1}\int_{X} g_{\mathbb{H}}(\sigma_{\mathfrak{e}
}^{-1}\eta z,\gamma_{i}^{n}\sigma_{\mathfrak{e}}^{-1}\eta z)\shyp(\zeta).
\end{align*}
The interchange of summation and integration in the above equation is valid, provided that the 
latter series converges absolutely. As the function $E_{X}(z)$ is a non-negative function, to prove the 
absolute convergence of latter series, it suffices to prove 
\begin{align}\label{lem3.27toprove}
\sum_{\mathfrak{e}\in \mathcal{E}_{X}}\sum_{\eta\in\Gamma_{X,\mathfrak{e}}\backslash\Gamma_{X}}
\sum_{n=1}^{m_{\mathfrak{e}}-1}\int_{X} g_{\mathbb{H}}(\sigma_{\mathfrak{e}
}^{-1}\eta z,\gamma_{i}^{n}\sigma_{\mathfrak{e}}^{-1}\eta z)\shyp(\zeta) \leq 
\frac{9 \,c_{X,\mathrm{ell}}\,|\mathcal{E}_{X}|}{\vx(X)}\sum_{\mathfrak{e}\in
\mathcal{E}_{X}}(m_{\mathfrak{e}}-1). 
\end{align}
For any $\mathfrak{e}\in\mathcal{E}_{X}$, $\gamma_{i}\in \Gamma_{X,\mathfrak{e}}$, and $\eta\in
\Gamma_{X,\mathfrak{e}}\backslash\Gamma_{X}$, from computation 
\eqref{lem1eqn2}, and from definition of constant $c_{X,\mathrm{ell}}$ in \eqref{defncellsmall}, we have
\begin{align}\label{lem3.27eqn1}
g_{\mathbb{H}}(\sigma_{\mathfrak{e}
}^{-1}\eta z,\gamma_{i}^{n}\sigma_{\mathfrak{e}}^{-1}\eta z)=\log\bigg(1+\frac{1}{\sin^{2}(n\pi
\slash m_{\mathfrak{e}})\sinh^{2}(\rho(\sigma_{\mathfrak{e}}^{-1}\eta z))}\bigg)\leq \\
c_{X,\mathrm{ell}}\log\bigg(1+\frac{1}{\sinh^{2}(\rho(\sigma_{\mathfrak{e}}^{-1}\eta z))}\bigg).
\end{align}
Furthermore, recall that the hyperbolic metric $\hyp(z)$ in elliptic coordinates is given by 
\begin{equation*}
\hyp(z)=\sinh(\rho(z))d\rho\wedge d\theta.
\end{equation*}
From estimate \eqref{lem3.27eqn1}, we find
\begin{align}
&\sum_{\mathfrak{e}\in \mathcal{E}_{X}}\sum_{\eta\in\Gamma_{X,\mathfrak{e}}\backslash\Gamma_{X}}
\sum_{n=1}^{m_{\mathfrak{e}}-1}\int_{X} g_{\mathbb{H}}(\sigma_{\mathfrak{e}
}^{-1}\eta z,\gamma_{i}^{n}\sigma_{\mathfrak{e}}^{-1}\eta z)\shyp(\zeta) \leq \notag\\&
c_{X,\mathrm{ell}}\sum_{\mathfrak{e}\in \mathcal{E}_{X}}(m_{\mathfrak{e}}-1)\sum_{\eta\in\Gamma_{X,\mathfrak{e}}\backslash
\Gamma_{X}}\int_{X} \log\bigg(1+\frac{1}{\sinh^{2}(\rho(\sigma_{\mathfrak{e}}^{-1}
\eta z))}\bigg)\shyp(z).\label{lem3.27eqn2}
\end{align}
For every $\mathfrak{e}\in\mathcal{E}_{X}$, after making the substitution 
$z\mapsto \eta^{-1}\sigma_{\mathfrak{e}}z$, from the $\mathrm{PSL}_{2}(\mathbb{R})$-invariance of the 
metric $\shyp(z)$, we compute
\begin{align*}
&\sum_{\eta\in\Gamma_{X,\mathfrak{e}}\backslash
\Gamma_{X}}\int_{X} \log\bigg(1+\frac{1}{\sinh^{2}(\rho(\sigma_{\mathfrak{e}}^{-1}
\eta z))}\bigg)\shyp(z)=\\&\int_{0}^{\infty}\int_{0}^{2\pi}\log\big(\coth^{2}(\rho(z))\big)
\frac{\sinh(\rho(z))d\rho\wedge d\theta}{\vx(X)}=\frac{4\pi \log 2}{\vx(X)}\leq \frac{9}{\vx(X)},
\end{align*}
which together with upper bound \eqref{lem3.27eqn2} proves upper bound \eqref{lem3.27toprove}, and 
completes the proof of the lemma. 
\end{proof}
\end{lem}
\begin{rem}\label{thirdline}
For any elliptic fixed point $\mathfrak{e}\in\mathcal{E}_{X}$, from Corollary \ref{cor3.9}, we have 
\begin{align*}
\sup_{z\in Y_{\varepsilon}}\bigg(\sum_{\mathfrak{e}\in\mathcal{E}_{X}}\frac{m_{\mathfrak{e}}-1}{2
g_{X}m_{\mathfrak{e}}}\,\big|\gxhyp(z,\mathfrak{e})\big|\bigg)\leq\sup_{z\in Y_{\varepsilon\slash 2}}
\bigg(\sum_{\mathfrak{e}\in\mathcal{E}_{X}}\frac{m_{\mathfrak{e}}-1}{2g_{X}m_{\mathfrak{e}}}
\,\big|\gxhyp(z,\mathfrak{e})\big|\bigg)\leq\frac{|\mathcal{E}_{X}|\boundbtwo}{2g_{X}},
\end{align*}
for any $\alpha\in (0,\lambda_{X,1})$ and $\delta\in (0,\varepsilon)$. For any $z\in Y_{\varepsilon}^{\mathrm{par}}$, combining Propositions \ref{prop3.25} and 
\ref{prop3.26}, and Lemma \ref{lem3.27} with the above upper bound, we obtain the following upper 
bound for the third line on the right-hand side of equation \eqref{phiunfolding}
\begin{align*}
\frac{|\mathcal{E}_{X}|\boundbtwo}{2g_{X}}-\frac{|\mathcal{P}_{X}|\log(\varepsilon\slash 2)}{g_{X}}
+\frac{5 \,c_{X,\mathrm{ell}}}{g_{X}\vx(X)}\sum_{\mathfrak{e}\in\mathcal{E}_{X}}(m_{\mathfrak{e}}-1)
+\frac{2\pi \left(d_{X}+1\right)^{2}}{\lambda_{X,1}\vx(X)},
\end{align*}
for any $\alpha\in (0,\lambda_{X,1})$ and $\delta\in (0,\varepsilon)$. 
\end{rem}
\begin{thm}\label{thm3.29}\label{boundc}
For any $\alpha\in(0,\lambda_{X,1})$ and $\delta\in \big(0,\min
\lbrace \varepsilon,\widetilde{\varepsilon}\rbrace\big)$, we have the following upper bound
\begin{align}
&\hspace{4cm}\sup_{z\in Y_{\varepsilon}^{\mathrm{par}}}\big|\phi_{X}(z)\big|\leq \boundc,\notag\\
&\mathrm{where}\, \boundc=\frac{\boundbtwo}{2g_{X}}\bigg(|\mathcal{P}_{X}|\bigg(1-
\frac{C_{X,\mathrm{par}}^{\mathrm{aux}}}{2\,\log(\varepsilon\slash 2)}\bigg)+
|\mathcal{E}_{X}|+1\bigg)-\frac{4\,|\mathcal{P}_{X}|\,\log(\varepsilon\slash 2)}{g_{X}}+
\frac{16\,C_{X,\mathrm{par}}}{g_{X}}+\notag\\ &\frac{5 \,c_{X,\mathrm{ell}}}{g_{X}\vx(X)}
\sum_{\mathfrak{e}\in\mathcal{E}_{X}}(m_{\mathfrak{e}}-1)+\frac{2\pi \left(d_{X}+1\right)^{2}}{
\lambda_{X,1}\vx(X)}+\frac{2\pi\,|c_{X}-1|}{g_{X}\vx(X)}-\frac{\pi\,|\mathcal{P}_{X}|\,
C_{X,\mathrm{par}}^{\mathrm{aux}}}{g_{X}\vx(X)\,\log(\varepsilon\slash 2)}.\label{boundceqn}
\end{align}
\begin{proof}
The proof of the theorem follows from Corollary \ref{cor3.17}, and combining the upper bounds stated 
in Remarks \ref{firstline}, \ref{secondline}, and \ref{thirdline}.
\end{proof}
\end{thm}
\begin{cor}\label{phicusp}
Let $p\in \mathcal{P}_{X}$ be any cusp. Then, for any $\alpha\in(0,\lambda_{X,1})$, 
$\delta\in \big(0,\min\lbrace \varepsilon,\widetilde{\varepsilon}\rbrace\big)$, and 
$z\in U_{\varepsilon}(p)$, we have
\begin{align*}
\phi_{X}(z)=-\frac{4\pi}{\vx(X)}
\log\bigg(\frac{\log|\vartheta_{p}(w)|}{\log\varepsilon}\bigg)+\phi_{p}(z), 
\end{align*}
where $\phi_{p}(z)$ is a subharmonic function for $z\in U_{\varepsilon}(p)$, which satisfies the 
following upper bound 
\begin{align*}
\sup_{z\in U_{\varepsilon(p)}} |\phi_{p}(z)|\leq \boundc.
\end{align*}
\begin{proof}
For any $p\in \mathcal{P}_{X}$ and $z\in U_{\varepsilon}(p)$, using equation \eqref{phi}, we find 
\begin{align*}
\del \bigg(\phi_{X}(z)+\frac{4\pi}{\vx(X)}\log\bigg(\frac{\log|\vartheta_{p}(w)|}{\log\varepsilon}
\bigg)\bigg)=\frac{4\pi\can(z)}{\hyp(z)}\geq 0,
\end{align*}
which implies that 
\begin{align*}
\phi_{p}(z)=\bigg(\phi_{X}(z)+\frac{4\pi}{\vx(X)}\log\bigg(\frac{\log|\vartheta_{p}(w)|}{
\log\varepsilon}\bigg)\bigg)
\end{align*}
is a subharmonic function. From Theorem \ref{boundc} and maximum principle for subharmonic functions, 
we derive
\begin{align*}
\sup_{z\in U_{\varepsilon}(p)} |\phi_{p}(z)|=\sup_{z\in \partial U_{\varepsilon}(p)} |\phi_{p}(z)|
=\sup_{z\in \partial U_{\varepsilon}(p)} |\phi(z)|\leq \boundc,
\end{align*}
which completes the proof of the lemma. 
\end{proof}
\end{cor}
\begin{cor}\label{phielliptic}
Let $\mathfrak{e}\in \mathcal{E}_{X}$ be any elliptic fixed point. Then, for any $\alpha\in(0,\lambda_{X,1})$, 
$\delta\in \big(0,\min\lbrace \varepsilon,\widetilde{\varepsilon}\rbrace\big)$, and 
$z\in U_{\varepsilon}(\mathfrak{e})$, we have
\begin{align*}
\phi_{X}(z)=-\frac{4\pi\log\big(1-|\vartheta_{\mathfrak{e}}(z)|^{2
\slash m_{\mathfrak{e}}}\big)}{\vx(X)}+\phi_{\mathfrak{e}}(z), 
\end{align*}
where $\phi_{\mathfrak{e}}(z)$ is a subharmonic function on $z\in U_{\varepsilon}(\mathfrak{e})$, 
which satisfies the following upper bound 
\begin{align*}
 \sup_{z\in U_{\varepsilon}(\mathfrak{e})}|\phi_{\mathfrak{e}}(z)|\leq \boundc.
\end{align*}
\begin{proof}
The proof of the corollary follows from similar arguments as in Corollary \ref{phicusp}.  
\end{proof}
\end{cor}
\begin{thm}\label{boundgcan}
For any $\alpha\in(0,\lambda_{X,1})$ and $\delta\in \big(0,\min
\lbrace \varepsilon,\widetilde{\varepsilon}\rbrace\big)$, we have the following upper bounds
\begin{align}
\sup_{z,w\in Y_{\varepsilon}}\big|\gxhyp(z,w)-\gxcan(z,w)\big|\leq 2\boundc ;\label{boundgcaneqn1}\\
\sup_{z,w\in Y_{\varepsilon}}\bigg|\gxcan(z,w)-\sum_{\gamma\in S_{\Gamma_{X}}(\delta;z,w)}
g_{\mathbb{H}}(z,\gamma w)\bigg|\leq 2\boundc+\boundb.\label{boundgcaneqn2}
\end{align} 
\begin{proof}
Upper bound \eqref{boundgcaneqn1} follows directly from formula \eqref{phi} and Theorem \ref{boundc}.  
From triangle inequality, for any $z,w\in Y_{\varepsilon}$, we have
\begin{align}\label{triangleinequality}
\bigg|\gxcan(z,w)-\sum_{\gamma\in S_{\Gamma_{X}}(\delta;z,w)}g_{\mathbb{H}}(z,\gamma w)\bigg| \leq
&\big|\gxcan(z,w)-\gxhyp(z,w)\big|+\notag\\&\bigg|\gxhyp(z,w)-\sum_{\gamma\in S_{\Gamma_{X}}
(\delta;z,w)}g_{\mathbb{H}}(z,\gamma w)\bigg|.
\end{align}
Hence, upper bound \eqref{boundgcaneqn2} follows directly from combining Theorem \ref{boundc} and Proposition 
\ref{prop3.2}.
\end{proof}
\end{thm}
\begin{cor}\label{corcusps}
Let $p,q\in\mathcal{P}_{X}$ and $p\not= q$ be two cusps. Then, for any $\alpha\in(0,\lambda_{X,1})$ and 
$\delta\in \big(0,\min\lbrace \varepsilon,\widetilde{\varepsilon}\rbrace\big)$, we have the 
following upper bounds
\begin{align}
&\sup_{\substack{z\in U_{\varepsilon}(p)\\w\in U_{\varepsilon}(q)}}\bigg|\gxcan(z,w)-\sum_{\gamma\in 
S_{\Gamma_{X}}(\delta;z,w)}g_{\mathbb{H}}(z,\gamma w)\bigg|\leq 2\boundc+\boundb;\label{gcanpq}\\
&\sup_{z,w\in U_{\varepsilon}(p)}\bigg|\gxcan(z,w)-\sum_{\gamma\in S_{\Gamma_{X}}(\delta;z,w)\backslash\lbrace\id\rbrace}g_{\mathbb{H}}(z,\gamma w)-
\sum_{\gamma\in\Gamma_{X,p}}\gh(z,\gamma w)\bigg|\leq 2\boundc+\boundb.\label{gcanpp}
\end{align}
\begin{proof}
Upper bound \eqref{gcanpq} follows directly from triangle inequality \eqref{triangleinequality}, and combining 
Corollaries \ref{cor3.11} and \ref{phicusp}.

Similarly upper bound \eqref{gcanpp} follows directly from triangle inequality 
\eqref{triangleinequality}, and combining Corollaries \ref{cor3.12} and \ref{phicusp}. 
\end{proof}
\end{cor}
\begin{rem}
Let $p,q\in\mathcal{P}_{X}$ and $p\not= q$ be two cusps. Then, for any $\alpha\in(0,\lambda_{X,1})$ and 
$\delta\in \big(0,\min \varepsilon,\widetilde{\varepsilon}\rbrace\big)$, from 
upper bound \eqref{gcanpq}, we have the following upper bound
\begin{align}\label{boundgcanpq}
\bigg|\gxcan(p,q)-\sum_{\gamma\in S_{\Gamma_{X}}(\delta;z,w)}g_{\mathbb{H}}(p,\gamma q)\bigg|=
\big|\gxcan(p,q)\big|\leq 2\boundc+\boundb. 
\end{align}
In an upcoming article, we will derive an upper bound for $\gxcan(p,q)$ using 
a different method, and the upper bound does not depend on the choice of $\varepsilon$. 
\end{rem}
\begin{cor}\label{corelliptic}
Let $\mathfrak{e},\mathfrak{f}\in\mathcal{E}_{X}$ and $\mathfrak{e}\not= \mathfrak{f}$ be two 
elliptic fixed points. Then, for any $\alpha\in(0,\lambda_{X,1})$ and $\delta\in \big(0,
 \varepsilon,\widetilde{\varepsilon}\rbrace\big)$, we have the following upper bounds
\begin{align*}
&\sup_{\substack{z\in U_{\varepsilon}(\mathfrak{e})\\w\in U_{\varepsilon}(\mathfrak{f})}}\bigg|
\gxcan(z,w)-\sum_{\gamma\in S_{\Gamma_{X}}(\delta;z,w)}g_{\mathbb{H}}(z,\gamma w)\bigg|
\leq 2\boundc+\boundb \notag\\&\sup_{z,w\in U_{\varepsilon}(\mathfrak{e})}\bigg|\gxcan(z,w)-
\sum_{\gamma\in S_{\Gamma_{X}}(\delta;z,w)\backslash\lbrace\id\rbrace}\gh(z,\gamma w)-
\sum_{\gamma\in\Gamma_{X,\mathfrak{e}}}\gh(z,\gamma w)\bigg|\leq 2\boundc+\boundb.
\end{align*}
\begin{proof}
The proof of the corollary follows from triangle inequality \ref{triangleinequality}, and combining 
Corollaries \ref{phielliptic} and \ref{cor3.13}.
\end{proof}
\end{cor}
\section{Bounds for families of modular curves}\label{section3.3}
In this section, we investigate the bounds obtained in previous subsections for  
certain sequences of Riemann orbisurfaces similar to the study conducted in Section 5 of \cite{jk}. 

We start by recalling the definition of an admissible sequence of non-compact hyperbolic Riemann 
orbisurfaces of finite volume.  
\begin{defn} 
Let $\lbrace X_{N}\rbrace_{N\in\mathcal{N}}$ indexed by $N\in\mathcal{N}\subseteq \mathbb{N}$ be a set 
of non-compact hyperbolic Riemann orbisurfaces of finite volume of genus $g_{N}\geq1$, which can be realized as a quotient space $\Gamma_{X_{N}}\backslash
\mathbb{H}$, where $\Gamma_{X_{N}}$ is a Fuchsian subgroup of the first kind acting by 
fractional linear transformations on the upper half-plane $\mathbb{H}$. 
We say that the sequence is \emph{admissible} if it is one of the following two types:
\vspace{0.2cm}

(1) If $\mathcal{N}=\mathbb{N}$ and $N\in\mathcal{N}$, then $X_{N+1}$ is a finite degree cover of 
$X_{N}$.

\vspace{0.2cm} 
(2) For $N\in\mathbb{N}_{> 0}$, let
\begin{align*}
& Y_{0}(N)=\Gamma_{0}(N)\backslash\mathbb{H},\quad Y_{1}(N)=\Gamma_{1}(N)\backslash\mathbb{H},\quad
Y(N)=\Gamma(N)\backslash\mathbb{H},
\end{align*}
with the congruence subgroups $\Gamma_{0}(N)$, $\Gamma_{1}(N)$, $\Gamma(N)$, respectively. In each of 
the three cases above, let $\mathcal{N}\subseteq \mathbb{N}$ be such that $Y_{0}(N)$, $Y_{1}(N)$, 
$Y(N)$ has genus bigger than zero for $N\in\mathcal{N}$, respectively. 
We then consider here the families $\lbrace X_{N}\rbrace_{N\in\mathcal{N}}$ given by 
\begin{align*}
\lbrace Y_{0}(N)\rbrace_{N\in\mathcal{N}},\,\,\,\lbrace Y_{1}(N)\rbrace_{N\in\mathcal{N}},\,\,\,\lbrace Y(N)\rbrace_{N\in\mathcal{N}}.   
\end{align*}
Denote by $q_{\mathcal{N}}\in\mathcal{N}$ the minimal element of the indexing set $\mathcal{N}$; 
in Case (1) $q_{\mathcal{N}}=0$ and in Case (2) $q_{\mathcal{N}}$ is the smallest prime in $\mathcal{N}$. For example, 
we can choose $q_{\mathcal{N}}=11$. 
\end{defn}
\begin{rem}
It is to be noted that the family of hyperbolic modular curves  do not form a single tower of hyperbolic Riemann orbisurfaces, 
hence, the distinction in the above definition. However, they form a different structure which we call 
a net. We refer the reader to Section 5 of \cite{jkannals} for further details. 
\end{rem}
\begin{notn}
Let $\lbrace X_{N}\rbrace_{N\in\mathcal{N}}$ be an admissible sequence of non-compact hyperbolic 
Riemann orbisurfaces of finite volume. We fix an $0<\varepsilon<1$ satisfying the conditions 
elucidated in Notation \ref{epsilondefn} for the Riemann orbisurface $X_{q_{\mathcal{N}}}$. 

\vspace{0.2cm}
Then, for any $N\in\mathcal{N}$, to emphasize the dependence on $N$, we denote the open coordinate disks around 
a cusp $p\in\mathcal{P}_{X_{N}}$ and an elliptic fixed point $\mathfrak{e}\in\mathcal{E}_{X_{N}}$ 
described in Notation \ref{epsilondefn} by $U_{N,\varepsilon}(p)$ and $U_{N,\varepsilon}(\mathfrak{e})$, 
respectively. Furthermore, we denote the compact subset $Y_{\varepsilon}$ associated to the Riemann orbisurface $X_{N}$ 
by $Y_{N,\varepsilon}$. 
\end{notn}
\begin{lem}\label{lem1bounds}
Let $\lbrace X_{N}\rbrace_{N\in\mathcal{N}}$ be an admissible sequence of non-compact hyperbolic 
Riemann orbisurfaces of finite volume. Then, we have the following upper bounds:

\vspace{0.2cm}
(1) For any $N\in\mathcal{N}$, we have 
\begin{align*}
d_{X_{N}}=O_{X_{q_{\mathcal{N}}}}(1).
\end{align*}
(2) For any $N\in\mathcal{N}$, we have
\begin{align*}
 c_{X_{N}}=O_{X_{q_{\mathcal{N}}}}\bigg(\frac{g_{X_{N}}}{\lambda_{X_{N},1}}\bigg).
\end{align*} 
(3) For any $N\in\mathcal{N}$, we have
\begin{align*}
 \ell_{X_{N}}=O_{X_{q_{\mathcal{N}}}}(1).
\end{align*}
(4) For any $N\in\mathcal{N}$, we have
\begin{align*}
 C_{X_{N}}^{HK}=O_{X_{q_{\mathcal{N}}}}(1).
\end{align*}
\begin{proof}
The first three assertions follow directly from Lemma 5.3 of \cite{jk}. 
Assertion (4) follows from employing arguments similar to the ones used to prove assertion (d) 
in Lemma 5.3 of \cite{jk}.  
\end{proof}
\end{lem}
\begin{notn}
For $\Gamma\subset \mathrm{PSL}_{2}(\mathbb{R})$ a Fuchsian subgroup of the first kind, let $\mathcal{M}_{\mathrm{par}}(\Gamma)$ denote 
the set of maximal parabolic subgroups of $\Gamma$. Note that for $P\in\mathcal{M}_{\mathrm{par}}(\Gamma)$, we have 
$P=\langle \gamma_{P} \rangle\in \mathcal{M}_{\mathrm{par}}(\Gamma)$, where $\gamma_{P}$ denotes a generator of the maximal parabolic 
subgroup $P$. Furthermore, there exists a scaling matrix $\sigma_{P}$ satisfying the condition
\begin{align}\label{scalingp}
\sigma_{P}^{-1}\gamma_{P}\sigma_{P}=\gamma_{\infty},\,\,\,
\mathrm{where}\,\,\gamma_{\infty}=\left(\begin{array}{ccc} 1 & 1\\ 0 & 1  \end{array}\right).
\end{align}
\end{notn}
\begin{rem}\label{bijectionparabolic}
Let $\Gamma$ be a subgroup of finite index in $\Gamma_{0}\subset \mathrm{PSL}_{2}(\mathbb{R})$, a 
Fuchsian subgroup of the first kind. Then, there is a bijection 
\begin{align*}
 \varphi: \mathcal{M}_{\mathrm{par}}(\Gamma)\longrightarrow \mathcal{M}_{\mathrm{par}}(\Gamma_{0}),
\end{align*}
which is given as follows. For each $P\in \mathcal{M}_{\mathrm{par}}(\Gamma)$, there exists a maximal 
parabolic subgroup $P_{0} \subset \Gamma_{0}$ containing $P$, and we set  $\varphi(P)=P_{0}$; 
the inverse map is given by $\varphi^{-1}(P_{0})=P_{0}\cap \Gamma.$ 

Furthermore, the scaling matrices $\sigma_{P_{0}}$ and $\sigma_{P}$ of the 
parabolic subgroups $P_{0}$ and $P$, respectively, can be chosen such that they satisfy the relation
\begin{align}\label{scalingmatrixparablic}
 \sigma_{P_{0}}=\sigma_{P}\left(\begin{array}{ccc} 1\slash \sqrt{n_{P_{0}P}}& 0\\0  &\sqrt{n_{P_{0}P}} \end{array}\right),
\end{align}
 where $n_{P_{0}P}=[P_{0}:P]$.
\end{rem}
\begin{prop}\label{prop1bounds}
Let $\lbrace X_{N}\rbrace_{N\in\mathcal{N}}$ be an admissible sequence of non-compact hyperbolic 
Riemann orbisurfaces of finite volume. Then, we have the following upper bounds:

\vspace{0.2cm}
(1) For any $N\in\mathcal{N}$, we have
\begin{align*}
C_{X_{N},\mathrm{par}}=O_{X_{q_{\mathcal{N}}}}(1).
\end{align*}
(2) For any $N\in\mathcal{N}$, we have
\begin{align*}
C_{X_{N},\mathrm{par}}^{\mathrm{aux}}=O_{X_{q_{\mathcal{N}}}}(1).
\end{align*}
(3) For any $N\in\mathcal{N}$, we have
\begin{align*}
&c_{X_{N},\mathrm{ell}}=O_{X_{q_{\mathcal{N}}}}(1);\quad\frac{5 \,c_{X_{N},\mathrm{ell}}}{g_{X_{N}}\vx(X_{N})}\sum_{\mathfrak{e}\in\mathcal{E}_{X_{N}}}
(m_{\mathfrak{e}}-1)= O_{X_{q_{\mathcal{N}}}}\bigg(\frac{|\mathcal{E}_{X_{N}}|}{g_{X_{N}}}\bigg).
\end{align*}
(4) For any $N\in\mathcal{N}$, we have
\begin{align*}
 C_{X,\mathrm{ell}}=O_{X_{q_{\mathcal{N}}}}(1). 
\end{align*}
\begin{proof}
We first prove assertion (1) for $\lbrace X_{N}\rbrace_{N\in\mathcal{N}}$, an admissible sequence 
of Riemann orbisurfaces of type (1). In order to do so, we need to consider the pair of Riemann 
orbisurfaces $X_{N}$ and $X_{q_{\mathcal{N}}}$, where $X_{N}$ is a finite degree cover of $X_{q_{\mathcal{N}}}$.

\vspace{0.2cm}
For any $N\in\mathcal{N}$ and $X_{N}=\Gamma_{X_{N}}\backslash \mathbb{H}$, from equation 
\eqref{cpardefn}, recall that
\begin{align*}
&C_{X_{N},\mathrm{par}}=\sup_{z\in X_{N}}\sum_{p\in\mathcal{P}_{X_{N}}}
\big(\mathcal{E}_{X_{N},\mathrm{par}}(z,2)-\Im(\sigma_{p}^{-1}z)^{2}\big).
\end{align*}
Consider the set 
\begin{align*}
\mathbb{P}(\Gamma_{X_{N}})= \big\lbrace\Gamma_{X_{N},p}\,|\,p\in \mathcal{P}_{X_{N}}\big\rbrace,
\end{align*}
where $\Gamma_{X_{N},p}$ denotes the stabilizer subgroup of the cusp $p\in\mathcal{
P}_{X_{N}}$. Keeping in mind that the set $\mathcal{P}_{X_{N}}$ is in bijection with the set of 
conjugacy classes of maximal parabolic subgroups of $\Gamma_{X_{N}}$, for any $z\in\mathbb{H}$, we have the equality  
\begin{align}
&\bigcup_{p\in\mathcal{P}_{X_{N}}}\bigcup_{\substack{\eta\in \Gamma_{X_{N},p}\backslash \Gamma_{X_{N}}
\\\eta\not=\id}}\eta^{-1}\Gamma_{X_{N},p}\eta=\bigcup_{\substack{P\in \mathcal{M}_{\mathrm{par}}
(\Gamma_{X_{N}})\\P\not\in \mathbb{P}(\Gamma_{X_{N}})}}P\notag\\\Longrightarrow& 
\sum_{p\in\mathcal{P}_{X_{N}}}\big(\mathcal{E}_{X_{N},\mathrm{par}}(z,2)-
\Im(\sigma_{p}^{-1}z)^{2}\big)=\sum_{\substack{P\in \mathcal{M}_{\mathrm{par}}
(\Gamma_{X_{N}})\\P\not\in \mathbb{P}(\Gamma_{X_{N}})}}\Im\big(\sigma_{P}^{-1}z\big)^{2}.\label{prop1boundseqn1}
\end{align}
From Remark \ref{bijectionparabolic}, we have a bijective map
\begin{align*}
\varphi_{N,q_{\mathcal{N}}}:\mathcal{M}_{\mathrm{par}}\big(\Gamma_{X_{N}}\big)\longrightarrow 
\mathcal{M}_{\mathrm{par}}\big(\Gamma_{X_{q_{\mathcal{N}}}}\big),
\end{align*}
sending $P\in \mathcal{M}_{\mathrm{par}}(\Gamma_{X_{N}})$ to $P_{0}=\varphi_{N,q_{\mathcal{N}}}(P)\in 
\mathcal{M}_{\mathrm{par}}(\Gamma_{X_{q_{\mathcal{N}}}})$. 
Then, for $z\in\mathbb{H}$, using the relation stated in equation (\ref{scalingmatrixparablic}), we have
\begin{align}\label{relnsigmaqsigmap}
y_{P}=\Im(\sigma_{P}^{-1}z)= \left(\begin{array}{ccc} 1\slash\sqrt{n_{P_{0}P}}& 0\\0  &\sqrt{n_{P_{0}P}}
\end{array}\right)\Im(\sigma_{P_{0}}^{-1}z)=\frac{y_{P_{0}}}{n_{P_{0}P}},
\end{align}
where $n_{P_{0}P}=[P_{0}:P].$ For $z\in\mathbb{H}$, using relations \eqref{prop1boundseqn1} and 
\eqref{relnsigmaqsigmap}, and the bijection between the sets $\mathcal{M}_{\mathrm{par}}(\Gamma_{X_{N}})$ and $\mathcal{M}_{\mathrm{par}}(\Gamma_{X_{q_{\mathcal{N}}}})$, 
we derive
\begin{align*}
\sum_{\substack{P\in \mathcal{M}_{\mathrm{par}}
(\Gamma_{X_{N}})\\P\not\in \mathbb{P}(\Gamma_{X_{N}})}}\Im\big(\sigma_{P}^{-1}z\big)^{2}\leq 
\sum_{\substack{P_{0}\in \mathcal{M}_{\mathrm{par}}(\Gamma_{X_{q_{\mathcal{N}}}})\\P_{0}\not\in 
\mathbb{P}(\Gamma_{X_{q_{\mathcal{N}}}})}}\frac{\Im\big(\sigma_{P_{0}}^{-1}z\big)^{2}}{n_{P_{0}P}^{2}}
\leq \sum_{\substack{P_{0}\in \mathcal{M}_{\mathrm{par}}(\Gamma_{X_{q_{\mathcal{N}}}})\\P_{0}\not\in
\mathbb{P}(\Gamma_{X_{q_{\mathcal{N}}}})}}\Im\big(\sigma_{P_{0}}^{-1}z\big)^{2},
\end{align*}
using which, we deduce that
\begin{align*}
C_{X_{N},\mathrm{par}}\leq  C_{X_{q_{\mathcal{N}}},\mathrm{par}}=O_{X_{q_{\mathcal{N}}}}(1), 
\end{align*}
which proves assertion (1) for the case of an admissible sequence of type (1). 

\vspace{0.2cm}
We now prove assertion (1) for $\lbrace X_{N}\rbrace_{N\in\mathcal{N}}$, an admissible sequence of Riemann 
orbisurfaces of type (2). We prove assertion (1) only for the sequence of modular curves 
$\lbrace Y_{0}(N)\rbrace_{N\in\mathcal{N}}$, as the proof extends with notational changes to the other 
sequences of modular curves $\lbrace Y_{1}(N)\rbrace_{N\in\mathcal{N}}$ and $\lbrace Y(N)\rbrace_{N\in\mathcal{N}}$.

\vspace{0.2cm}
For any $N\in\mathcal{N}$ the modular curve $Y_{0}(N)$ is a finite degree cover of 
$Y_{0}(1)=\mathrm{PSL}_{2}(\mathbb{Z})\backslash \mathbb{H}$. Extending our notation to the modular 
curve $Y_{0}(1)$, and adapting the arguments from the proof for admissible sequences of Riemann 
orbisurfaces of type (1), for $N\in\mathcal{N}$, we have 
\begin{align*}
 C_{Y_{0}(N),\mathrm{par}}=O(1),\Longrightarrow
 C_{Y_{0}(N),\mathrm{par}}=O_{Y_{0}(q_{\mathcal{N}})}(1).
\end{align*}
This completes the proof for assertion (1).

\vspace{0.2cm}
For the case of admissible sequences of Riemann orbisurfaces of type (1), assertion (2) has been 
established as Proposition 5.4 in \cite{K}. Using Proposition 5.4 from \cite{K} and adapting 
the arguments from proof of assertion (1), trivially proves assertion (2) for the case of 
admissible sequences of Riemann orbisurfaces of type (2). 

\vspace{0.2cm}
We first prove assertion (3) for $\lbrace X_{N}\rbrace_{N\in\mathcal{N}}$, an admissible sequence 
of Riemann orbisurfaces of type (1). We again the consider a pair of Riemann orbisurfaces 
$X_{N}$ and $X_{q_{\mathcal{N}}}$, where $X_{N}$ is a finite degree cover of $X_{q_{\mathcal{N}}}$.

\vspace{0.2cm}
For any $N\in\mathcal{N}$, from equation \eqref{defncellsmall}, recall that
\begin{align*}
c_{X_{N},\mathrm{ell}}=\max\big\lbrace 1\slash\sin^{2}(n\pi\slash m_{\mathfrak{e}})\big|\,\mathfrak{e}\in
\mathcal{E}_{X_{N}},0\,< n\leq m_{\mathfrak{e}}-1 \big\rbrace . 
\end{align*}
Observe that
\begin{align*}
\big\lbrace m_{\mathfrak{e}}\big|\,\mathfrak{e}\in\mathcal{E}_{X_{N}}\big\rbrace\subseteq
\big\lbrace m_{\mathfrak{e}}\big|\,\mathfrak{e}\in\mathcal{E}_{X_{q_{\mathcal{N}}}}\big\rbrace,\quad 
\sum_{\mathfrak{e}\in\mathcal{E}_{X_{N}}}( m_{\mathfrak{e}}-1)\leq |\mathcal{E}_{X_{N}}|
\sum_{\mathfrak{e}\in\mathcal{E}_{X_{q_{\mathcal{N}}}}}( m_{\mathfrak{e}}-1),
\end{align*}
which along with the inequality $g_{X_{N}}\leq \vx(X_{N})$, trivially proves assertion (3) or 
admissible sequences of Riemann orbisurfaces of type (1). 

\vspace{0.2cm}
Adapting similar arguments as the ones used to prove assertion (1) for admissible sequences of 
Riemann orbisurfaces of type (2), trivially proves assertion (3) for admissible sequences of 
Riemann orbisurfaces of type (2). 

\vspace{0.2cm}
Assertion (4) follows easily from similar arguments as the ones used to prove assertions (1), (2), 
and (3). 
\end{proof}
\end{prop}
\begin{prop}\label{ghypuniform}
Let $\lbrace X_{N}\rbrace_{N\in\mathcal{N}}$ be an admissible sequence of non-compact hyperbolic 
Riemann orbisurfaces of finite volume. Then, for any $N\in\mathcal{N}$, 
$\alpha\in(0,\lambda_{X_{N},1})$, and $\delta > 0$, we have the following estimate
\begin{align*}
\sup_{z,w\in Y_{N,\varepsilon}}\bigg|g_{X_{N},\mathrm{hyp}}(z,w)-\sum_{\gamma\in S_{\Gamma_{X_{N}}}
(\delta;z,w)}g_{\mathbb{H}}(z,\gamma w)\bigg|=O_{X_{q_{\mathcal{N}}},\varepsilon,\alpha,\delta}(1). 
\end{align*}
\begin{proof}
The proof of the proposition from similar arguments as the ones used to prove Theorem 5.5 in \cite{jk}, 
and using Lemma \ref{lem1bounds} and Propositions \ref{prop3.2} and \ref{prop1bounds}.
\end{proof}
\end{prop}
\begin{thm}\label{phiuniform}
Let $\lbrace X_{N}\rbrace_{N\in\mathcal{N}}$ be an admissible sequence of non-compact hyperbolic 
Riemann orbisurfaces of finite hyperbolic volume. Then, for any $N\in\mathcal{N}$, we have the 
following estimates
\begin{align}
&\sup_{z,w\in Y_{N,\varepsilon}}\big|g_{X_{N},\mathrm{can}}(z,w)-g_{X_{N},\mathrm{hyp}}(z,w)
\big|=O_{X_{q_{\mathcal{N}}},\varepsilon}\bigg(\frac{\big(|\mathcal{P}_{X_{N}}|+|\mathcal{E}_{X_{N}}|
\big)}{g_{X_{N}}}\bigg(1+\frac{1}{\lambda_{X_{N},1}}\bigg)\bigg);\label{phiuniformeqn1}\\
&\sup_{z,w\in Y_{N,\varepsilon}}\bigg|g_{X_{N},\mathrm{can}}(z,w)-\sum_{\gamma\in 
S_{\Gamma_{X_{N}}}(\delta;z,w)}g_{\mathbb{H}}(z,\gamma w)\bigg|=O_{X_{q_{\mathcal{N}}},\varepsilon,
\delta}\bigg(\frac{\big(|\mathcal{P}_{X_{N}}|+|\mathcal{E}_{X_{N}}|\big)}{g_{X_{N}}}\bigg(1+
\frac{1}{\lambda_{X_{N},1}}\bigg)\bigg). \label{phiuniformeqn2}
\end{align}
\begin{proof}
Estimate \eqref{phiuniformeqn1} follows from similar arguments as the ones used to prove Theorem 5.6 in 
\cite{jk}, and using Lemma \ref{lem1bounds}, and Propositions \ref{boundgcan} and \ref{prop1bounds}. 

\vspace{0.2cm}
Estimate \eqref{phiuniformeqn2} follows from similar arguments as the ones used to prove Corollary 
5.7 in \cite{jk}, and using Proposition \ref{ghypuniform} and estimate \eqref{phiuniformeqn1}. 
\end{proof}
\end{thm}
\begin{cor}\label{gcancusp}
Let $\lbrace X_{N}\rbrace_{N\in\mathcal{N}}$ be an admissible sequence of non-compact hyperbolic 
Riemann orbisurfaces of finite hyperbolic volume. For any $N\in\mathcal{N}$, let $p,q\in\mathcal{P}_{
X_{N}}$ and $p\not= q$ be two cusps. Then, for any $\delta > 0$, we have the following estimates
\begin{align*}
&\sup_{\substack{z\in U_{N,\varepsilon}(p)\\w\in U_{N,\varepsilon}(q)}}\bigg|
g_{X_{N},\mathrm{can}}(z,w)-\sum_{\gamma\in S_{\Gamma_{X_{N}}}(\delta;z,w)}g_{\mathbb{H}}
(z,\gamma w)\bigg| =O_{X_{q_{\mathcal{N}}},\varepsilon,
\delta}\bigg(\frac{\big(|\mathcal{P}_{X_{N}}|+|\mathcal{E}_{X_{N}}|\big)}{g_{X_{N}}}\bigg(1+
\frac{1}{\lambda_{X_{N},1}}\bigg)\bigg);\\&\sup_{z,w\in U_{N,\varepsilon}(p)}\bigg|g_{X_{N},
\mathrm{can}}(z,w)-\sum_{\gamma\in S_{\Gamma_{X_{N}}}(\delta;z,w)\backslash\lbrace\id\rbrace}
g_{\mathbb{H}}(z,\gamma w)-\sum_{\gamma\in\Gamma_{X_{N},p}}\gh(z,\gamma w)\bigg| =\\&
\quad O_{X_{q_{\mathcal{N}}},\varepsilon,\delta}\bigg(\frac{\big(|\mathcal{P}_{X_{N}}|+
|\mathcal{E}_{X_{N}}|\big)}{g_{X_{N}}}\bigg(1+\frac{1}{\lambda_{X_{N},1}}\bigg)\bigg).
\end{align*}
\begin{proof}
The proof of the corollary follows directly from Corollary \ref{corcusps} and Theorem 
\ref{phiuniform}. 
\end{proof}
\end{cor}
\begin{cor}\label{gcanelliptic}
Let $\lbrace X_{N}\rbrace_{N\in\mathcal{N}}$ be an admissible sequence of non-compact hyperbolic 
Riemann orbisurfaces of finite hyperbolic volume. For any $N\in\mathcal{N}$, let $\mathfrak{e},\mathfrak{f}\in
\mathcal{E}_{X_{N}}$ and $\mathfrak{e}\not= \mathfrak{f}$ be two elliptic fixed points. Then, for 
any $\delta > 0$, we have the following estimates 
\begin{align*}
&\sup_{\substack{z\in U_{N,\varepsilon}(\mathfrak{e})\\w\in U_{N,\varepsilon}(\mathfrak{f})}}
\bigg|g_{X_{N},\mathrm{can}}(z,w)-\sum_{\gamma\in S_{\Gamma_{X_{N}}}(\delta;z,w)}g_{\mathbb{H}}
(z,\gamma w)\bigg|=O_{X_{q_{\mathcal{N}}},\varepsilon,\delta}\bigg(\frac{\big(|\mathcal{P}_{X_{N}}|+
|\mathcal{E}_{X_{N}}|\big)}{g_{X_{N}}}\bigg(1+\frac{1}{\lambda_{X_{N},1}}\bigg)\bigg);
\\&\sup_{z,w\in U_{N,\varepsilon}(\mathfrak{e})}\bigg|g_{X_{N},\mathrm{can}}(z,w)-
\sum_{\gamma\in S_{\Gamma_{X_{N}}}(\delta;z,w)\backslash\lbrace\id\rbrace}\gh(z,\gamma w)-
\sum_{\gamma\in\Gamma_{X_{N},\mathfrak{e}}}\gh(z,\gamma w)\bigg|=\\&\quad
O_{X_{q_{\mathcal{N}}},\varepsilon,\delta}\bigg(\frac{\big(|\mathcal{P}_{X_{N}}|+
|\mathcal{E}_{X_{N}}|\big)}{g_{X_{N}}}\bigg(1+\frac{1}{\lambda_{X_{N},1}}\bigg)\bigg).
\end{align*} 
\begin{proof}
The proof of the corollary follows directly from Corollary \ref{corelliptic} and 
Theorem \ref{phiuniform}.  
\end{proof}
\end{cor}
\begin{rem}\label{remonbounds}
Consider the admissible sequence of modular curves $\lbrace Y_{0}(N)\rbrace_{N\in\mathcal{N}}$. 
For any $N\in\mathcal{N}$, the modular curve $Y_{0}(N)$ is a finite degree cover of 
$Y_{0}(1)=\mathrm{PSL}_{2}(\mathbb{Z})\backslash \mathbb{H}$. Furthermore, we have the following 
estimate for the genus $g_{Y_{0}(N)}$ of $Y_{0}(N)$
\begin{align*}
 g_{Y_{0}(N)}=O\big(N\log N\big).
\end{align*}
From Riemann-Hurwitz formula, we have the following estimates 
\begin{align*}
\big[\mathrm{PSL}_{2}(\mathbb{Z}):\Gamma_{0}(N)\big]=O\big(g_{Y_{0}(N)}\big),\quad 
|\mathcal{P}_{Y_{0}(N)}|=O\big(N\log N\big), \quad |\mathcal{E}_{Y_{0}(N)}|=
O_{\epsilon}\big( N^{\epsilon}\big),
\end{align*}
for any $\epsilon > 0$. We refer the reader to \cite{shimura}, p. 22-25 for details of the 
above estimates. 

\vspace{0.2cm}
Furthermore, from work of A. Selberg \cite{selberg}, we know that $\lambda_{Y_{0}(N),1}\geq 3\slash 16$. All 
the above estimates also hold true for the other sequences of modular curves 
$\lbrace Y_{1}(N)\rbrace_{N\in\mathcal{N}}$ and $\lbrace Y(N)\rbrace_{N\in\mathcal{N}}$.
\end{rem}
\begin{cor}
Let $\lbrace X_{N}\rbrace_{N\in\mathcal{N}}$, an admissible sequence of Riemann orbisurfaces 
of type (2). Then, for any $N\in\mathcal{N}$ and $\delta > 0$, we have the following estimate
\begin{align}\label{finalcoreqn1}
\sup_{z,w\in Y_{N,\varepsilon}}\bigg|g_{X_{N},\mathrm{can}}(z,w)-\sum_{\gamma\in 
S_{\Gamma_{X}}(\delta;z,w)}g_{\mathbb{H}}(z,\gamma w)\bigg|=
O_{X_{q_{\mathcal{N}}},\varepsilon,\delta}(1).
\end{align}
For any $N\in\mathcal{N}$, let $p,q\in\mathcal{P}_{X_{N}}$ and $p\not= q$ be two cusps. Then, for any 
$\delta > 0$, we have the following estimates
\begin{align}
&\sup_{\substack{z\in U_{N,\varepsilon}(p)\\w\in U_{N,\varepsilon}(q)}}\bigg|
g_{X_{N},\mathrm{can}}(z,w)-\sum_{\gamma\in S_{\Gamma_{X_{N}}}(\delta;z,w)}g_{\mathbb{H}}
(z,\gamma w)\bigg| =O_{X_{q_{\mathcal{N}}},\varepsilon,\delta}(1);\label{finalcoreqn2}
\\&\sup_{z,w\in U_{N,\varepsilon}(p)}\bigg|g_{X_{N},\mathrm{can}}(z,w)-
\sum_{\gamma\in S_{\Gamma_{X_{N}}}(\delta;z,w)\backslash\lbrace\id\rbrace}
g_{\mathbb{H}}(z,\gamma w)-\sum_{\gamma\in\Gamma_{X_{N},p}}\gh(z,\gamma w)\bigg| =
O_{X_{q_{\mathcal{N}}},\varepsilon,\delta}(1).\label{finalcoreqn3}
\end{align}
For any $N\in\mathcal{N}$, let $\mathfrak{e},\mathfrak{f}\in\mathcal{E}_{X_{N}}$ and $\mathfrak{e}\not= \mathfrak{f}$ be two 
elliptic fixed points. Then, for any $\delta > 0$, we have the following estimates
\begin{align}
&\sup_{\substack{z\in U_{N,\varepsilon}(\mathfrak{e})\\w\in U_{N,\varepsilon}(\mathfrak{f})}}
\bigg|g_{X_{N},\mathrm{can}}(z,w)-\sum_{\gamma\in S_{\Gamma_{X_{N}}}(\delta;z,w)}g_{\mathbb{H}}
(z,\gamma w)\bigg|=O_{X_{q_{\mathcal{N}}},\varepsilon,\delta}(1);\label{finalcoreqn4}
\\&\sup_{z,w\in U_{N,\varepsilon}(\mathfrak{e})}\bigg|g_{X_{N},\mathrm{can}}(z,w)-
\sum_{\gamma\in S_{\Gamma_{X_{N}}}(\delta;z,w)\backslash\lbrace\id\rbrace}\gh(z,\gamma w)-
\sum_{\gamma\in\Gamma_{X_{N},\mathfrak{e}}}\gh(z,\gamma w)\bigg|=
O_{X_{q_{\mathcal{N}}},\varepsilon,\delta}(1).\label{finalcoreqn5}
\end{align} 
\begin{proof}
Estimate \eqref{finalcoreqn1} follows directly from combining Remark \eqref{remonbounds} with 
Theorem \ref{phiuniform}. Estimates \eqref{finalcoreqn2} and \eqref{finalcoreqn3} follow directly from combining Remark \eqref{remonbounds} with 
Corollary \ref{gcancusp}. Estimates \eqref{finalcoreqn4} and \eqref{finalcoreqn5} follow directly from combining Remark \eqref{remonbounds} with 
Corollary \ref{gcanelliptic}. 
\end{proof}
\end{cor}
{\small{}}
\vspace{0.3cm}
{\small{
Department of Mathematics, \\University of Hyderabad, \\Prof. C.~R.~Rao Road, Gachibowli,\\
Hyderabad, 500046, India\\email: anilatmaja@gmail.com}}
\end{document}